\documentclass[11pt,oneside,reqno]{amsart}
\usepackage[latin9]{inputenc}
\usepackage{bm}
\usepackage{amsbsy}
\usepackage{amstext}
\usepackage{amsthm}
\usepackage{amssymb}
\usepackage{geometry}
\usepackage{setspace}
\usepackage{esint}
\usepackage[bookmarks=false,
 breaklinks=false,pdfborder={0 0 1},backref=false,colorlinks=false]
 {hyperref}

\makeatletter
\numberwithin{equation}{section}
\numberwithin{figure}{section}

\usepackage{amsthm}
\usepackage{mathrsfs}
\usepackage[noadjust]{cite}
\usepackage{stmaryrd}
\usepackage{enumitem}
\setlist[itemize]{leftmargin=*}
\setlist[enumerate]{leftmargin=*}
\usepackage{bbm}
\usepackage{color}

\usepackage{tikz}

\def\th@plain{\thm@notefont{}\itshape}
\def\th@definition{\thm@notefont{}\normalfont}

\makeatother

\theoremstyle{remark}
\newtheorem*{notation*}{\protect\notationname}
\theoremstyle{plain}
\newtheorem{thm}{\protect\theoremname}[section]
\theoremstyle{remark}
\newtheorem{rem}[thm]{\protect\remarkname}
\theoremstyle{plain}
\newtheorem{prop}[thm]{\protect\propositionname}
\newtheorem{lem}[thm]{\protect\lemmaname}
\theoremstyle{definition}
\newtheorem{defn}[thm]{\protect\definitionname}
\theoremstyle{remark}
\newtheorem*{acknowledgement*}{\protect\acknowledgementname}
\providecommand{\acknowledgementname}{Acknowledgement}
\providecommand{\definitionname}{Definition}
\providecommand{\lemmaname}{Lemma}
\providecommand{\notationname}{Notation}
\providecommand{\propositionname}{Proposition}
\providecommand{\remarkname}{Remark}
\providecommand{\theoremname}{Theorem}

\begin{document}
\title[Condensation and Metastability in the Supercritical Two-Species ZRP]{Condensation and Metastability in the Supercritical Two-Species Zero-Range
Process via Resolvent and $H^{1}$-Approximation}
\author{Seonwoo Kim and Seungchan Lee}
\begin{abstract}
In this article, we investigate the two-species zero-range process,
a multi-species generalization of the classical zero-range process.
First, we analyze its condensation regime, which directly parallels
its single-species counterpart. As our main result, we establish the
dynamical metastable behavior of the location of the single condensate,
showing that its motion is governed by a simple Markov chain on the
accelerated time scale $N^{1+\alpha}$, where $N$ denotes the total
number of particles in the system and the parameter $\alpha>1$ governs
the attractivity of the system. Another novelty of this work lies
in the proof technique, which integrates the recently developed resolvent
approach with the $H^{1}$-approximation method to rigorously characterize
metastability.
\end{abstract}

\address{Seonwoo Kim. Department of Mathematics, Yonsei University, Republic
of Korea.}
\email{seonwookim@yonsei.ac.kr}
\address{Seungchan Lee. Department of Mathematics, Yonsei University, Republic
of Korea.}
\email{seungchan2718@yonsei.ac.kr}

\maketitle
\tableofcontents{}

\section{\label{sec1}Introduction and Main Results}

Metastability is a generic physical phenomenon where a thermodynamic
or stochastic system remains trapped for an exponentially long time
in an apparent equilibrium state before undergoing a random, abrupt
transition to its true ground state. The rigorous mathematical analysis
of this phenomenon was initiated in the 1970s, with the foundational
work by Penrose and Lebowitz \cite{PL71} framing the problem within
non-equilibrium statistical mechanics. Over the subsequent decades,
the field developed along two powerful, complementary paths: the pathwise
(large deviations and Freidlin--Wentzell type) approach introduced
by Cassandro, Galves, Olivieri and Vares \cite{CGOV84}, and the potential-theoretic
method pioneered by Bovier, Eckhoff, Gayrard and Klein \cite{BEGK04}.
While the potential-theoretic framework provided sharp, pre-exponential
estimates for mean transition times (the celebrated Eyring--Kramers
formula) in reversible Markov processes via Dirichlet forms and capacities,
extending these quantitative asymptotics to non-reversible, non-gradient,
or infinite-dimensional particle systems long remained a central challenge.

Recent advances have resolved many of these longstanding bottlenecks
through the martingale approach to metastability, introduced by Beltr\'an
and Landim \cite{BL10}, which characterizes the coarse-grained trace
process on metastable valleys via a simple Markov chain on the set
of metastable indices---a framework now commonly referred to as the
Markov chain model reduction technique. Pushing this program into
new frontiers, one of the authors and Seo \cite{KS25} introduced
the $H^{1}$-approximation method, which analyzes both the Eyring--Kramers
formula and model reduction via the construction of sharp asymptotic
test functions for the equilibrium potential (see \eqref{eq:eq-pot-RW})
in the $H^{1}$-norm. More recently, Landim, Marcondes and Seo \cite{LMS25}
introduced the resolvent approach to metastability, which resolves
the existing barriers with non-reversible analytic techniques through
the analysis of microscopic and macroscopic resolvent equations (see
\eqref{eq:res-mic} and \eqref{eq:res-mac}), thereby establishing
an equivalent characterization of metastability alongside various
local mixing time criteria for complex non-reversible Markov chains
and interacting particle systems.

This paper builds directly upon this modern framework. We argue that
combining the $H^{1}$-approximation technique with the resolvent
approach provides a robust and efficient path to characterizing the
dynamical transitions between metastable states. To demonstrate the
power of this unified approach, we investigate the two-species zero-range
process in the condensing regime (see Section \ref{sec1.1} for a
precise definition of the model). This model is a multi-species generalization
of the well-known condensing zero-range process \cite{EH05}, whose
single-species condensation and metastability have been extensively
studied \cite{BL12a,Lan14,LMS23,LMS25,Seo19,Oh19} during the past
few years. While static and macroscopic properties---such as phase
separation, equivalence of ensembles, condensation mechanisms and
hydrodynamic limits---have been widely investigated in the literature
(see \cite{DSZ17,Gro04,Gro08} and the references therein), the dynamical
metastability of the multi-species model has remained out of reach
due to the highly intricate energy landscape governing the typical
transition paths. In this paper, we provide a quantitatively precise
characterization of this dynamical behavior by synthesizing the $H^{1}$-approximation
and resolvent frameworks. We emphasize that our choice of two species
is made primarily for clarity of presentation; the same arguments
and methods apply to general $k$-species models for any $k\ge2$
(see Remark \ref{rem:multi-species} for details).
\begin{notation*}
\label{nota:general}We gather a few general notations that are used
repeatedly throughout the article.
\begin{itemize}
\item $\mathbb{N}:=\{1,2,\dots\}$, $\mathbb{N}_{0}:=\{0\}\cup\mathbb{N}$,
and $\mathbb{R}_{+}:=[0,\infty)$.
\item $a\wedge b:=\min\{a,b\}$ and $a\vee b:=\max\{a,b\}$.
\item $\lfloor a\rfloor$ (resp. $\lceil a\rceil$) denotes the greatest
(resp. least) integer less (resp. greater) than or equal to $a$.
\item $\llbracket a,b\rrbracket:=[a,b]\cap\mathbb{Z}$.
\item $c,c',c''>0$ denote positive constants that may vary line by line,
and may depend on the underlying geometry, but do not depend on the
number of particles $N$ or parameter $\epsilon$, etc.
\item $a_{N}\ll b_{N}$ or $a_{N}=o(b_{N})$ if $\lim_{N\to\infty}a_{N}/b_{N}=0$.
\item $a_{N}=O(b_{N})$ if $|a_{N}|\le cb_{N}$ for all $N\ge1$.
\item $a_{N}\simeq b_{N}$ if $\lim_{N\to\infty}a_{N}/b_{N}=1$.
\item $a_{N}\asymp b_{N}$ if $c\le\liminf_{N\to\infty}a_{N}/b_{N}\le\limsup_{N\to\infty}a_{N}/b_{N}\le c'$.
\end{itemize}
\end{notation*}

\subsection{\label{sec1.1}Two-Species Zero-Range Process}

Consider a finite site space $S$ with
\[
\kappa:=|S|\ge2,
\]
and two irreducible Markovian jump rates $r_{1},r_{2}:S\times S\to[0,\infty)$
with unique stationary distributions $m_{1},m_{2}:S\to(0,\infty)$,
respectively. We denote them as underlying random walks of types,
or species, $i\in\{1,2\}$. We assume each $m_{i}$, $i=1,2$, is
normalized such that 
\begin{equation}
\max_{x\in S}m_{i}(x)=1.\label{eq:mi-normalized}
\end{equation}
Denote by $L_{i}$ the corresponding infinitesimal generator that
acts on $\mathbb{R}^{S}$:
\[
L_{i}f(x)=\sum_{y\in S}r_{i}(x,y)\left[f(y)-f(x)\right],\qquad f\in\mathbb{R}^{S}.
\]

For $N\in\mathbb{N}$, consider a particle configuration space $\Omega_{N}$
defined as
\[
\Omega_{N}:=\left\{ \bm{\eta}=(\eta_{1},\eta_{2})\in\mathbb{N}^{S}_{0}\times\mathbb{N}^{S}_{0}:|\eta_{1}|:=\sum_{x\in S}\eta_{1}(x)=A,\enspace|\eta_{2}|:=\sum_{x\in S}\eta_{2}(x)=B\right\} ,
\]
where the integers $A=A_{N}$ and $B=B_{N}$ satisfy
\begin{equation}
\lim_{N\to\infty}\frac{A_{N}}{N}=\rho,\qquad\lim_{N\to\infty}\frac{B_{N}}{N}=1-\rho,\label{eq:ANBN}
\end{equation}
for a fixed density value $\rho\in(0,1)$. Notice that for each $\bm{\eta}\in\Omega_{N}$,
\[
|\bm{\eta}|:=|\eta_{1}|+|\eta_{2}|=N,
\]
which denotes the conserved total number of particles. Condition \eqref{eq:ANBN}
implies that the proportion of type $1$ (resp. $2$) particles in
the system is asymptotically $\rho$ (resp. $1-\rho$).

Consider a continuous-time Markov chain $\{\bm{\eta}_{N}(t)\}_{t\ge0}=\{(\bm{\eta}_{N}(x,t))_{x\in S}\}_{t\ge0}$
on $\Omega_{N}$ defined by an infinitesimal generator $\mathcal{L}_{N}$
acting on functions $F\in\mathbb{R}^{\Omega_{N}}$ as
\begin{equation}
\mathcal{L}_{N}F(\bm{\eta}):=\sum_{i\in\{1,2\}}\sum_{x,y\in S}\bm{g}_{i}(\bm{\eta}(x))r_{i}(x,y)\left(F(\bm{\eta}-\delta^{i}_{x}+\delta^{i}_{y})-F(\bm{\eta})\right).\label{eq:gen}
\end{equation}
Here, 
\begin{itemize}
\item $\bm{g}_{1},\bm{g}_{2}:\mathbb{N}_{0}\times\mathbb{N}_{0}\to[0,\infty)$
such that $\bm{g}_{i}(k_{1},k_{2})=0$ if and only if $k_{i}=0$;
\item $\delta^{i}_{z}$ denotes the configuration which is defined as
\begin{equation}
\delta^{i}_{z}(w):=\begin{cases}
(1,0) & \text{if}\quad i=1,\enspace w=z,\\
(0,1) & \text{if}\quad i=2,\enspace w=z,\\
(0,0) & \text{otherwise}.
\end{cases}\label{eq:delta-zi-def}
\end{equation}
\end{itemize}
Let $R_{N}:\Omega_{N}\times\Omega_{N}\to[0,\infty)$ be the corresponding
transition rate function, with the convention that $R_{N}(\bm{\eta},\bm{\eta}):=0$.
Denote by $\mathbb{P}^{N}_{\bm{\eta}}$ and $\mathbb{E}^{N}_{\bm{\eta}}$
the law and the corresponding expectation, respectively, of the process
$\bm{\eta}_{N}$ starting from $\bm{\eta}\in\Omega_{N}$.

It is obvious that $\bm{\eta}_{N}$ is irreducible, thus it admits
a unique stationary state. According to \cite[Theorem 4.1]{Gro04},
the stationary state is of product type if and only if the compatibility
relations hold:
\begin{equation}
\bm{g}_{1}(m,n-1)\bm{g}_{2}(m,n)=\bm{g}_{2}(m-1,n)\bm{g}_{1}(m,n)\qquad\text{for all}\quad(m,n)\in\mathbb{N}^{2}.\label{eq:g-cond}
\end{equation}
In words, this implies that from $(m-1,n-1)$ to $(m,n)$, following
the two lattice upright paths $(m-1,n-1)\xrightarrow{\bm{g}_{1}}(m,n-1)\xrightarrow{\bm{g}_{2}}(m,n)$
and $(m-1,n-1)\xrightarrow{\bm{g}_{2}}(m-1,n)\xrightarrow{\bm{g}_{1}}(m,n)$
yield the same activation energy, in which $\bm{g}_{1}$ and $\bm{g}_{2}$
record the horizontal and vertical energy, respectively. In this regard,
let us define $\bm{g!}:\mathbb{N}_{0}\times\mathbb{N}_{0}\to\mathbb{R}$
as
\[
\bm{g!}(m,n):=\bm{g}_{1}(1,0)\cdots\bm{g}_{1}(m,0)\bm{g}_{2}(m,1)\cdots\bm{g}_{2}(m,n),
\]
with the convention that $\bm{g!}(0,0):=1$. By \eqref{eq:g-cond},
$\bm{g!}(m,n)$ equals the activation energy from $(0,0)$ to $(m,n)$.
Then, the stationary state $\nu_{N}$ is written as
\[
\nu_{N}(\bm{\eta}):=\frac{1}{z_{N}}\frac{\bm{m}^{\bm{\eta}}}{\bm{g!}(\bm{\eta})}:=\frac{1}{z_{N}}\prod_{x\in S}\left[\frac{m_{1}(x)^{\eta_{1}(x)}m_{2}(x)^{\eta_{2}(x)}}{\bm{g!}(\bm{\eta}(x))}\right],
\]
where $z_{N}:=\sum_{\bm{\eta}\in\Omega_{N}}\frac{\bm{m}^{\bm{\eta}}}{\bm{g!}(\bm{\eta})}$
is a normalizing constant which makes $\nu_{N}$ a probability measure.
\begin{rem}
If $r_{i}(\cdot,\cdot)$ is further \emph{reversible} with respect
to $m_{i}$ for each $i\in\{1,2\}$, i.e., if $m_{i}(x)r_{i}(x,y)=m_{i}(y)r_{i}(y,x)$
for all $x,y\in S$, then the detailed balance equation holds:
\[
\nu_{N}(\bm{\eta})\bm{g}_{i}(\bm{\eta}(x))r_{i}(x,y)=\nu_{N}(\bm{\eta}-\delta^{i}_{x}+\delta^{i}_{y})\bm{g}_{i}((\boldsymbol{\eta}-\delta^{i}_{x}+\delta^{i}_{y})(y))r_{i}(y,x)\qquad\text{for all}\quad x,y\in S,
\]
thus the process $\bm{\eta}_{N}$ is also \emph{reversible} with respect
to $\nu_{N}$. In general, the system is not necessarily reversible
with respect to $\nu_{N}$, making it a non-equilibrium stationary
state.
\end{rem}

\subsection{\label{sec1.2}Condensing Regime}

From now on, we consider a specific choice of $\bm{g}_{1},\bm{g}_{2}$
which works as a representative model to study the condensation and
metastability phenomena. Fix a constant $\alpha\in(0,\infty)$ and
assume that
\begin{equation}
\bm{g}_{1}(m,n):=\begin{cases}
0 & \text{if}\quad m=n=0,\\
\frac{m}{m+n}\frac{\mathfrak{a}(m+n)}{\mathfrak{a}(m+n-1)} & \text{if}\quad m+n\ge1,
\end{cases}\qquad\bm{g}_{2}(m,n):=\begin{cases}
0 & \text{if}\quad m=n=0,\\
\frac{n}{m+n}\frac{\mathfrak{a}(m+n)}{\mathfrak{a}(m+n-1)} & \text{if}\quad m+n\ge1,
\end{cases}\label{eq:g-def}
\end{equation}
where 
\begin{equation}
\mathfrak{a}(k):=\begin{cases}
1 & \text{if}\quad k=0,\\
k^{\alpha} & \text{if}\quad k\ge1.
\end{cases}\label{eq:a-def}
\end{equation}
It is easy to check that \eqref{eq:g-cond} holds for this choice,
and that
\[
\bm{g!}(m,n)=\frac{\mathfrak{a}(m+n)}{{m+n \choose n}}.
\]
For notational reasons, we rewrite the stationary state $\nu_{N}$
of the system as
\begin{equation}
\nu_{N}(\bm{\eta})=\frac{N^{\alpha}}{Z_{N}\binom{N}{A}}\frac{\bm{m}^{\bm{\eta}}}{\bm{g!}(\bm{\eta})}=\frac{N^{\alpha}}{Z_{N}{N \choose A}}\prod_{x\in S}\left[\frac{m_{1}(x)^{\eta_{1}(x)}m_{2}(x)^{\eta_{2}(x)}}{\bm{g!}(\bm{\eta}(x))}\right],\label{eq:nuN-def}
\end{equation}
where
\begin{equation}
Z_{N}=\frac{N^{\alpha}}{\binom{N}{A}}\sum_{\bm{\eta}\in\Omega_{N}}\frac{\bm{m}^{\bm{\eta}}}{\bm{g!}(\bm{\eta})}.\label{eq:ZN-def}
\end{equation}
In this way, $Z_{N}$ will converge to a positive constant as $N\to\infty$
(see Proposition \ref{prop:ZN-limit}).
\begin{rem}
Suppose temporarily that the two underlying random walks are identical:
$r_{1}\equiv r_{2}\equiv r$. Then, project $\Omega_{N}$ onto the
\emph{blind-species} space $\widehat{\Omega}_{N}$ by
\[
\Omega_{N}\ni\bm{\eta}=(\eta_{1},\eta_{2})\longmapsto\eta_{1}+\eta_{2}\in\widehat{\Omega}_{N}:=\left\{ \eta\in\mathbb{N}^{S}_{0}:|\eta|=N\right\} .
\]
Then, it is straightforward to check that the projected system is
still Markovian and corresponds to the infinitesimal generator
\[
\widehat{\mathcal{L}}_{N}F(\eta)=\sum_{x,y\in S}\bm{g}(\eta(x))r(x,y)\left(F(\eta-\delta_{x}+\delta_{y})-F(\eta)\right),
\]
where $\delta_{z}(w):={\bf 1}\{w=z\}$ and $\bm{g}:\mathbb{N}_{0}\to\mathbb{R}$
is defined as (cf. \eqref{eq:a-def})
\[
\bm{g}(k):=\begin{cases}
0 & \text{if}\quad k=0,\\
\frac{\mathfrak{a}(k)}{\mathfrak{a}(k-1)} & \text{if}\quad k\ge1.
\end{cases}
\]
Therefore, we recover the original zero-range process \cite{BL12a}
on the graph $S$ with $N$ particles.
\end{rem}

For the single-species zero-range process, the condensation and metastability
phenomena are very well understood. In \cite{GSS03,JMP00}, condensation
and non-condensation have been observed for $\alpha\ge1$ and $\alpha<1$,
respectively, indicating a clear phase transition at the value $\alpha_{\star}=1$,
in a more statistical mechanics point of view. Beltr\'an and Landim
\cite{BL12a} studied the reversible, supercritical case of $\alpha>1$
and proved that the successive metastable transitions of the condensates
is well approximated by an irreducible random walk on the graph. Later,
Landim \cite{Lan14} and Seo \cite{Seo19} generalized this result
to non-reversible setup, and Landim, Marcondes and Seo \cite{LMS23,LMS25}
studied the critical case of $\alpha=1$ in the symmetric case, which
becomes much more technical due to the weak attraction of particles.
Several other works \cite{AGL17,BCL26,BJL17,Cho24,JG24} have explored
other related regimes, such as the fast-rate counterpart, the thermodynamic
limit with a diverging number of sites, and the pre-stationary nucleation
dynamics.

Motivated by the phase transition of condensation in the original
zero-range process, one might expect our model to undergo a similar
phase transition depending on the parameter $\alpha\in(0,\infty)$.
This is indeed true, as summarized in the following. Recall from \eqref{eq:nuN-def}
that $\nu_{N}$ denotes the unique stationary state. Define
\[
S_{\star}:=\left\{ x\in S:m_{1}(x)=m_{2}(x)=1\right\} .
\]
Since each $m_{i}$ is normalized such that $\max_{x\in S}m_{i}(x)=1$,
the set $S_{\star}$ collects the sites that maximize both $m_{1}$
and $m_{2}$. Let us assume that $S_{\star}$ has at least two elements,
and let
\begin{equation}
\kappa_{\star}:=|S_{\star}|\ge2.\label{eq:kappa-star}
\end{equation}
For each $x\in S$, define
\begin{equation}
\mathcal{E}^{x}_{N}:=\left\{ \bm{\eta}\in\Omega_{N}:|\bm{\eta}(x)|\ge N-\ell_{N}\right\} ,\label{eq:ENx-def}
\end{equation}
where $(\ell_{N})_{N\ge1}$ is a sequence of integers such that
\begin{equation}
1\ll\ell_{N}\ll N.\label{eq:ellN-def}
\end{equation}
Let us denote by $\bm{\xi}^{x}_{N}\in\mathcal{E}^{x}_{N}$ the central
configuration under which all particles are at $x$;
\begin{equation}
\bm{\xi}^{x}_{N}:=A\delta^{1}_{x}+B\delta^{2}_{x}.\label{eq:xiNx}
\end{equation}
Then, define
\begin{equation}
\mathcal{E}_{N}:=\bigcup_{x\in S_{\star}}\mathcal{E}^{x}_{N}\qquad\text{and}\qquad\Delta_{N}:=\Omega_{N}\setminus\mathcal{E}_{N}.\label{eq:EN-DeltaN}
\end{equation}
Mind that the union of $\mathcal{E}^{x}_{N}$ is taken over $x\in S_{\star}$,
not the full $S$. The next theorem is proved in Section \ref{sec2}.
\begin{thm}
\label{thm:cond}The following statements hold for any $\ell_{N}$
that satisfies \eqref{eq:ellN-def}.
\begin{enumerate}
\item For $\alpha\in(1,\infty)$, we have
\begin{equation}
\lim_{N\to\infty}\nu_{N}(\mathcal{E}^{x}_{N})=\frac{1}{\kappa_{\star}}\qquad\text{for all}\quad x\in S_{\star}.\label{eq:nuN-cond}
\end{equation}
In particular, $\lim_{N\to\infty}\nu_{N}(\Delta_{N})=0$.
\item If $\alpha=1$, then \eqref{eq:nuN-cond} still holds, under an additional
assumption that
\begin{equation}
\lim_{N\to\infty}\frac{\log\ell_{N}}{\log N}=1.\label{eq:ellN-cri}
\end{equation}
\item Otherwise, if $\alpha\in(0,1)$, then
\[
\lim_{N\to\infty}\nu_{N}(\mathcal{E}^{x}_{N})=0\qquad\text{for all}\quad x\in S.
\]
\end{enumerate}
\end{thm}

Theorem \ref{thm:cond} can be interpreted as follows. If $\alpha\ge1$,
then the stationary state $\nu_{N}$ concentrates on a much smaller
set $\mathcal{E}_{N}$, in which most ($N-\ell_{N}\simeq N$) of the
particles concentrate on one of the specific sites in $S_{\star}$,
indicating the \emph{condensation} phenomenon. On the other hand,
if $\alpha<1$, then $\nu_{N}$ becomes negligible on any condensed
states and instead concentrates on the scattered configurations. This
indicates a sharp \emph{phase transition} on the parameter $\alpha$,
with a critical value $\alpha_{\star}=1$. In this regard, we refer
to the regimes $\alpha>1$, $\alpha=1$, and $\alpha<1$ as \emph{supercritical},
\emph{critical}, and \emph{subcritical}, respectively.

\subsection{\label{sec1.3}Main Result: Metastable Condensate Movements in the
Supercritical Regime}

Hereafter, we focus on the supercritical case; i.e., in this subsection,
we assume that
\[
\alpha\in(1,\infty).
\]

\subsubsection*{Order Process}

Define
\begin{equation}
\ell_{N}:=\lfloor\gamma\log N\rfloor,\label{eq:ellN-exact}
\end{equation}
where $\gamma>0$ is a small constant which will be specified later
at \eqref{eq:gamma-def}. Recall the sets $\mathcal{E}^{x}_{N}$,
$\mathcal{E}_{N}$, and $\Delta_{N}$ from \eqref{eq:ENx-def} and
\eqref{eq:EN-DeltaN}. Consider the local time $T_{N}(t)$ in $\mathcal{E}_{N}$
\[
T_{N}(t):=\int^{t}_{0}{\bf 1}\left\{ \bm{\eta}_{N}(u)\in\mathcal{E}_{N}\right\} {\rm d}u\qquad\text{for}\quad t\ge0,
\]
and its generalized inverse $S_{N}(t):=\sup\{u\ge0:T_{N}(u)\le t\}$.
The \emph{trace process} on $\mathcal{E}_{N}$ is defined as
\[
\bm{\eta}^{{\rm tr}}_{N}(t):=\bm{\eta}_{N}(S_{N}(t)).
\]
It is well known (see e.g. \cite[Section 6.1]{BL10}) that $\bm{\eta}^{{\rm tr}}_{N}$
becomes a Markov process on $\mathcal{E}_{N}$, and its unique stationary
distribution becomes $\nu_{N}$ restricted to $\mathcal{E}_{N}$.
To track only the location of the massive condensate, define a projection
$\Phi_{N}:\mathcal{E}_{N}\to S_{\star}$ as
\[
\Phi_{N}=\sum_{x\in S_{\star}}x{\bf 1}_{\mathcal{E}^{x}_{N}}.
\]
Then, the \emph{order process} on $S_{\star}$ is defined as
\[
X_{N}(t):=\Phi_{N}(\bm{\eta}^{{\rm tr}}_{N}(t)).
\]

\subsubsection*{Limiting Markov Chain}

Recall the underlying random walks induced by the two jump rates $r_{1},r_{2}:S\times S\to[0,\infty)$.
For each $i\in\{1,2\}$ and two disjoint nonempty sets $U,V\subset S$,
let $\mathfrak{h}_{i,U,V}:S\to\mathbb{R}$ be the unique solution
of the boundary-value problem 
\begin{equation}
\begin{cases}
L_{i}f\equiv0 & \text{on}\enspace S\setminus(U\cup V),\\
f\equiv1 & \text{on}\enspace U,\\
f\equiv0 & \text{on}\enspace V.
\end{cases}\label{eq:eq-pot-RW}
\end{equation}
The function $\mathfrak{h}_{i,U,V}$ is called the \emph{equilibrium
potential} between $U$ and $V$. It is easy to verify via the Markov
property that
\[
\mathfrak{h}_{i,U,V}(x)=P^{i}_{x}\left[H_{U}<H_{V}\right],
\]
where $P^{i}_{x}$ denotes the law of the Markov chain induced by
$L_{i}$ starting from $x$, and $H_{U}$ (resp. $H_{V}$) denotes
the hitting time of $U$ (resp. $V$). Then, the \emph{capacity} between
$U$ and $V$ is defined as
\begin{equation}
c_{i}(U,V):={\rm cap}_{i}(U,V):=\sum_{x\in S}m_{i}(x)\mathfrak{h}_{i,U,V}(x)(-L_{i}\mathfrak{h}_{i,U,V})(x).\label{eq:ci-def}
\end{equation}
In cases where $U$ or $V$ is a singleton, we typically omit the
brackets in the capacity notation. We know from \cite[Lemma 2.3]{GL14}
that capacity is symmetric:
\begin{equation}
c_{i}(U,V)=c_{i}(V,U).\label{eq:cap-sym}
\end{equation}

Now, we define a new Markov chain $\{\mathbb{X}(t)\}_{t\ge0}$ on
$S_{\star}$ as follows. Its infinitesimal generator acts on $f\in\mathbb{R}^{S_{\star}}$
as
\begin{equation}
\mathfrak{L}_{\mathbb{X}}f(x)=\sum_{y\in S_{\star}\setminus\{x\}}\frac{1}{\Gamma(\alpha)\mathcal{I}_{\alpha}}\left(\frac{\rho}{c_{1}(x,y)}+\frac{1-\rho}{c_{2}(x,y)}\right)^{-1}\left[f(y)-f(x)\right],\qquad x\in S_{\star}.\label{eq:lMC-gen}
\end{equation}
Here, $\Gamma(\alpha)$ is a positive constant which is defined at
\eqref{eq:Gamma-def}, and
\begin{equation}
\mathcal{I}_{\alpha}:=\int^{1}_{0}t^{\alpha}(1-t)^{\alpha}\,{\rm d}t\in(0,\infty).\label{eq:Ialpha-def}
\end{equation}
Denote by $\mathfrak{r}:S_{\star}\times S_{\star}\to[0,\infty)$ the
transition rate function of $\mathbb{X}$.

The following main theorem identifies the macroscopic limit of the
(non-Markovian) chain $X_{N}$ via acceleration. Define the acceleration
scale as
\[
\theta_{N}:=N^{1+\alpha}.
\]
In addition, let us define a larger neighborhood $\tilde{\mathcal{E}}^{x}_{N}$
around $\bm{\xi}^{x}_{N}$:
\begin{equation}
\tilde{\mathcal{E}}^{x}_{N}:=\left\{ \bm{\eta}\in\Omega_{N}:|\bm{\eta}(x)|\ge N-\tilde{\ell}_{N}\right\} \qquad\text{with}\quad\tilde{\ell}_{N}:=\lfloor N^{b}\rfloor,\label{eq:EN-tilde}
\end{equation}
where $b\in(0,1)$ is an arbitrarily small constant. It is clear that
$\mathcal{E}^{x}_{N}\subset\tilde{\mathcal{E}}^{x}_{N}$ for sufficiently
big $N$.
\begin{thm}
\label{thm:main}Suppose that $\alpha>1$. For any $x\in S_{\star}$,
the following statements hold.
\begin{enumerate}
\item (Local metastability) Each set $\mathcal{E}^{x}_{N}$ is metastable,
in the sense that from any configuration in $\mathcal{E}^{x}_{N}$,
it reaches $\bm{\xi}^{x}_{N}$ (cf. \eqref{eq:xiNx}) before exiting
$\tilde{\mathcal{E}}^{x}_{N}$ with high probability:
\[
\lim_{N\to\infty}\inf_{\bm{\eta}\in\mathcal{E}^{x}_{N}}\mathbb{P}^{N}_{\bm{\eta}}\left[H_{\bm{\xi}^{x}_{N}}<H_{\Omega_{N}\setminus\tilde{\mathcal{E}}^{x}_{N}}\right]=1.
\]
\item (Dynamical transitions) The law of the accelerated order process $\{X_{N}(\theta_{N}t)\}_{t\ge0}$
starting from any configuration in $\mathcal{E}^{x}_{N}$ converges
weakly, in the Skorokhod topology as $N\to\infty$, to the law of
the limiting Markov chain $\mathbb{X}$ starting from $x$.
\item (Negligibility) The time spent outside $\mathcal{E}_{N}$ is negligible:
\[
\lim_{N\to\infty}\sup_{\bm{\eta}\in\mathcal{E}^{x}_{N}}\mathbb{E}^{N}_{\bm{\eta}}\left[\int^{T}_{0}{\bf 1}\left\{ \bm{\eta}_{N}(\theta_{N}u)\in\Delta_{N}\right\} {\rm d}u\right]=0\qquad\text{for all}\quad T>0.
\]
\end{enumerate}
\end{thm}

Part (1) is called the \emph{attractor} condition (cf. \cite[cond. (V1)]{BL10}),
and it indicates that starting from any configuration in $\mathcal{E}^{x}_{N}$,
the process reaches a local equilibrium inside $\mathcal{E}^{x}_{N}$
before it escapes $\tilde{\mathcal{E}}^{x}_{N}$ and enters the effective
metastable transition set. This explains the exact terminology of
\emph{metastability} in our context.
\begin{rem}
The transition rates given in \eqref{eq:lMC-gen} directly generalize
the rates for the single-species analogue presented in \cite[eq. (3.3)]{Seo19}.
The two formulas differ by a factor of $M_{\star}$ in the previous
paper since, in our situation, the stationary distributions $m_{1},m_{2}$
on $S$ are normalized such that \eqref{eq:mi-normalized} holds.
Taking this into account, it is easy to see that our rates a direct
generalization of \cite[eq. (3.3)]{Seo19}.
\end{rem}

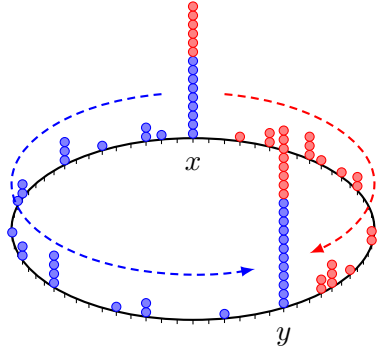
\begin{figure}
\begin{tikzpicture}[scale=1.2]
\draw[thick] (0,0) ellipse (2 and 1);
\foreach \i in {0,...,74} { \draw[] ({2*cos(\i*5)},{1*sin(\i*5)-0.05})--({2*cos(\i*5)},{1*sin(\i*5)}); }

\foreach \i in {6} { \foreach \j in {0,...,8} { \fill[blue!50!white] ({2*cos(\i*15)},{1*sin(\i*15)+0.05+0.1*\j}) circle (0.05); \draw[blue] ({2*cos(\i*15)},{1*sin(\i*15)+0.05+0.1*\j}) circle (0.05); }
\foreach \j in {9,...,14} { \fill[red!50!white] ({2*cos(\i*15)},{1*sin(\i*15)+0.05+0.1*\j}) circle (0.05); \draw[red] ({2*cos(\i*15)},{1*sin(\i*15)+0.05+0.1*\j}) circle (0.05); } }
\foreach \i in {20} { \foreach \j in {0,...,11} { \fill[blue!50!white] ({2*cos(\i*15)},{1*sin(\i*15)+0.05+0.1*\j}) circle (0.05); \draw[blue] ({2*cos(\i*15)},{1*sin(\i*15)+0.05+0.1*\j}) circle (0.05); }
\foreach \j in {12,...,19} { \fill[red!50!white] ({2*cos(\i*15)},{1*sin(\i*15)+0.05+0.1*\j}) circle (0.05); \draw[red] ({2*cos(\i*15)},{1*sin(\i*15)+0.05+0.1*\j}) circle (0.05); } }

\foreach \i in {20} { \foreach \j in {0,...,0} { \fill[blue!50!white] ({2*cos(\i*5)},{1*sin(\i*5)+0.05+0.1*\j}) circle (0.05); \draw[blue] ({2*cos(\i*5)},{1*sin(\i*5)+0.05+0.1*\j}) circle (0.05); } }
\foreach \i in {21} { \foreach \j in {0,...,1} { \fill[blue!50!white] ({2*cos(\i*5)},{1*sin(\i*5)+0.05+0.1*\j}) circle (0.05); \draw[blue] ({2*cos(\i*5)},{1*sin(\i*5)+0.05+0.1*\j}) circle (0.05); } }
\foreach \i in {24} { \foreach \j in {0,...,0} { \fill[blue!50!white] ({2*cos(\i*5)},{1*sin(\i*5)+0.05+0.1*\j}) circle (0.05); \draw[blue] ({2*cos(\i*5)},{1*sin(\i*5)+0.05+0.1*\j}) circle (0.05); } }
\foreach \i in {27} { \foreach \j in {0,...,2} { \fill[blue!50!white] ({2*cos(\i*5)},{1*sin(\i*5)+0.05+0.1*\j}) circle (0.05); \draw[blue] ({2*cos(\i*5)},{1*sin(\i*5)+0.05+0.1*\j}) circle (0.05); } }
\foreach \i in {32} { \foreach \j in {0,...,1} { \fill[blue!50!white] ({2*cos(\i*5)},{1*sin(\i*5)+0.05+0.1*\j}) circle (0.05); \draw[blue] ({2*cos(\i*5)},{1*sin(\i*5)+0.05+0.1*\j}) circle (0.05); } }
\foreach \i in {33} { \foreach \j in {0,...,0} { \fill[blue!50!white] ({2*cos(\i*5)},{1*sin(\i*5)+0.05+0.1*\j}) circle (0.05); \draw[blue] ({2*cos(\i*5)},{1*sin(\i*5)+0.05+0.1*\j}) circle (0.05); } }
\foreach \i in {37} { \foreach \j in {0,...,0} { \fill[blue!50!white] ({2*cos(\i*5)},{1*sin(\i*5)+0.05+0.1*\j}) circle (0.05); \draw[blue] ({2*cos(\i*5)},{1*sin(\i*5)+0.05+0.1*\j}) circle (0.05); } }
\foreach \i in {40} { \foreach \j in {0,...,1} { \fill[blue!50!white] ({2*cos(\i*5)},{1*sin(\i*5)+0.05+0.1*\j}) circle (0.05); \draw[blue] ({2*cos(\i*5)},{1*sin(\i*5)+0.05+0.1*\j}) circle (0.05); } }
\foreach \i in {44} { \foreach \j in {0,...,3} { \fill[blue!50!white] ({2*cos(\i*5)},{1*sin(\i*5)+0.05+0.1*\j}) circle (0.05); \draw[blue] ({2*cos(\i*5)},{1*sin(\i*5)+0.05+0.1*\j}) circle (0.05); } }
\foreach \i in {49} { \foreach \j in {0,...,0} { \fill[blue!50!white] ({2*cos(\i*5)},{1*sin(\i*5)+0.05+0.1*\j}) circle (0.05); \draw[blue] ({2*cos(\i*5)},{1*sin(\i*5)+0.05+0.1*\j}) circle (0.05); } }
\foreach \i in {51} { \foreach \j in {0,...,1} { \fill[blue!50!white] ({2*cos(\i*5)},{1*sin(\i*5)+0.05+0.1*\j}) circle (0.05); \draw[blue] ({2*cos(\i*5)},{1*sin(\i*5)+0.05+0.1*\j}) circle (0.05); } }
\foreach \i in {56} { \foreach \j in {0,...,0} { \fill[blue!50!white] ({2*cos(\i*5)},{1*sin(\i*5)+0.05+0.1*\j}) circle (0.05); \draw[blue] ({2*cos(\i*5)},{1*sin(\i*5)+0.05+0.1*\j}) circle (0.05); } }

\foreach \i in {15} { \foreach \j in {0,...,0} { \fill[red!50!white] ({2*cos(\i*5)},{1*sin(\i*5)+0.05+0.1*\j}) circle (0.05); \draw[red] ({2*cos(\i*5)},{1*sin(\i*5)+0.05+0.1*\j}) circle (0.05); } }
\foreach \i in {13} { \foreach \j in {0,...,1} { \fill[red!50!white] ({2*cos(\i*5)},{1*sin(\i*5)+0.05+0.1*\j}) circle (0.05); \draw[red] ({2*cos(\i*5)},{1*sin(\i*5)+0.05+0.1*\j}) circle (0.05); } }
\foreach \i in {10} { \foreach \j in {0,...,2} { \fill[red!50!white] ({2*cos(\i*5)},{1*sin(\i*5)+0.05+0.1*\j}) circle (0.05); \draw[red] ({2*cos(\i*5)},{1*sin(\i*5)+0.05+0.1*\j}) circle (0.05); } }
\foreach \i in {9} { \foreach \j in {0,...,0} { \fill[red!50!white] ({2*cos(\i*5)},{1*sin(\i*5)+0.05+0.1*\j}) circle (0.05); \draw[red] ({2*cos(\i*5)},{1*sin(\i*5)+0.05+0.1*\j}) circle (0.05); } }
\foreach \i in {7} { \foreach \j in {0,...,0} { \fill[red!50!white] ({2*cos(\i*5)},{1*sin(\i*5)+0.05+0.1*\j}) circle (0.05); \draw[red] ({2*cos(\i*5)},{1*sin(\i*5)+0.05+0.1*\j}) circle (0.05); } }
\foreach \i in {5} { \foreach \j in {0,...,1} { \fill[red!50!white] ({2*cos(\i*5)},{1*sin(\i*5)+0.05+0.1*\j}) circle (0.05); \draw[red] ({2*cos(\i*5)},{1*sin(\i*5)+0.05+0.1*\j}) circle (0.05); } }
\foreach \i in {-2} { \foreach \j in {0,...,1} { \fill[red!50!white] ({2*cos(\i*5)},{1*sin(\i*5)+0.05+0.1*\j}) circle (0.05); \draw[red] ({2*cos(\i*5)},{1*sin(\i*5)+0.05+0.1*\j}) circle (0.05); } }
\foreach \i in {-6} { \foreach \j in {0,...,0} { \fill[red!50!white] ({2*cos(\i*5)},{1*sin(\i*5)+0.05+0.1*\j}) circle (0.05); \draw[red] ({2*cos(\i*5)},{1*sin(\i*5)+0.05+0.1*\j}) circle (0.05); } }
\foreach \i in {-8} { \foreach \j in {0,...,2} { \fill[red!50!white] ({2*cos(\i*5)},{1*sin(\i*5)+0.05+0.1*\j}) circle (0.05); \draw[red] ({2*cos(\i*5)},{1*sin(\i*5)+0.05+0.1*\j}) circle (0.05); } }
\foreach \i in {-9} { \foreach \j in {0,...,1} { \fill[red!50!white] ({2*cos(\i*5)},{1*sin(\i*5)+0.05+0.1*\j}) circle (0.05); \draw[red] ({2*cos(\i*5)},{1*sin(\i*5)+0.05+0.1*\j}) circle (0.05); } }

\draw[blue,densely dashed,thick,-latex] ({2*cos(100)},{sin(100)+0.5}) arc (100:290:2 and 1);
\draw[red,densely dashed,thick,-latex] ({2*cos(80)},{sin(80)+0.5}) arc (80:-50:2 and 1);
\draw ({2*cos(90)},{sin(90)-0.1}) node[below]{$x$};
\draw ({2*cos(300)},{sin(300)-0.1}) node[below]{$y$};
\end{tikzpicture}\caption{\label{fig1.1}Typical particle transitions in the simple example
given at Remark \ref{rem:simple}. Blue arrow indicates small fraction
of type-$1$ particles moving in the counter-clockwise direction,
and red arrow indicates small fraction of type-$2$ particles moving
in the opposite direction.}
\end{figure}

\begin{rem}
\label{rem:simple}To illustrate the main theorem, let us consider
a simple setting where $S$ is a one-dimensional discrete torus, and
$r_{1}$ (resp. $r_{2}$) corresponds to a totally asymmetric random
walk in the counter-clockwise (resp. clockwise) direction. See Figure
\ref{fig1.1} for an illustration. In this context, suppose that a
metastable transition occurs from a condensate at a site $x\in S_{\star}$
to another site $y\in S_{\star}$. Then, the typical trajectory of
this transition proceeds as follows: small fractions of type-$1$
and type-$2$ particles successively detach from the macroscopic condensate
at $x$, moving respectively in the counter-clockwise and clockwise
directions, and ultimately coalesce at the target site $y\in S_{\star}$,
to form the new condensate. In particular, the two particle species
execute independent-like random walks along the torus in opposite
directions before recombining at the target site.

The relative abundance of the two species also effects the metastable
transition. Indeed, according to \eqref{eq:lMC-gen}, the transition
rate from $x$ to $y$ is proportional to 
\[
\left(\frac{\rho}{c_{1}(x,y)}+\frac{1-\rho}{c_{2}(x,y)}\right)^{-1},
\]
so that the macroscopic particle densities act as weights on the effective
transport resistances associated with the underlying two random walks.
In particular, when $c_{1}(x,y)\neq c_{2}(x,y)$, changing the relative
proportion of the two species accelerates or slows down the transitions,
while the qualitative transition mechanism described above remains
unchanged.
\end{rem}

\begin{rem}
\label{rem:multi-species}Although the present article focuses on
the two-species case, one can easily generalize all the results to
the $k$-species case for any fixed integer $k\ge2$. Namely, the
condensation result at Theorem \ref{thm:cond} remains true along
the same threshold $\alpha_{\star}=1$, and the dynamical metastable
transitions can also be explained, as in Theorem \ref{thm:main},
via three characterizations: local metastability, dynamical transitions,
and negligibility of the remainder set, with the same time scale $\theta_{N}=N^{1+\alpha}$
and jump rates
\[
\frac{1}{\Gamma(\alpha)\mathcal{I}_{\alpha}}\left(\sum^{k}_{j=1}\frac{\rho_{j}}{c_{j}(x,y)}\right)^{-1}\qquad\text{for}\quad x\to y,
\]
where $c_{j}(x,y)$ denotes the capacity between $x$ and $y$ according
to the $j$th random walk, and $\rho_{j}$ denotes the macroscopic
particle density of the $j$th species.
\end{rem}

\begin{rem}
We strongly believe that the same results hold in the remaining critical
case $\alpha=1$. Indeed, the condensation result is already established
in Theorem \ref{thm:cond}, and parts (2) and (3) of Theorem \ref{thm:main}
can be checked in a similar manner, albeit with significantly more
technical details. The main bottleneck in this generalization lies
in part (1) of the main theorem, namely the local metastability condition.
This presents a substantially harder challenge because, when $\alpha=1$,
each (potentially) metastable set $\mathcal{E}^{x}_{N}$, $x\in S_{\star}$,
becomes significantly larger, so that a crude capacity estimate (as
performed in Proposition \ref{prop3} for the supercritical case)
no longer suffices to ensure rapid local mixing. In \cite{LMS23},
this difficulty was overcome for the single-species zero-range process
in the symmetric random walk case through the construction of a delicate
superharmonic test function on the annulus region of $\mathcal{E}^{x}_{N}$,
combined with same the capacity estimate localized near the central
configuration $\bm{\xi}^{x}_{N}$. We do not expect this strategy
to extend, at least directly, to general non-reversible, multi-species
settings, suggesting that entirely new ideas will be required to resolve
the critical case.
\end{rem}

\subsection{\label{sec1.4}Proof Ideas}

Here, we sketch the main strategy to prove Theorem \ref{thm:main}.

\subsubsection*{Resolvent Theory}

The main tool is the recently developed \emph{resolvent approach}
to metastability \cite{LMS25}. The key idea is to interpret the metastable
behavior of the dynamics as a specific analytic property of the resolvent
equation
\begin{equation}
(\lambda-\theta_{N}\mathcal{L}_{N})F_{N}=\sum_{x\in S_{\star}}\mathfrak{g}(x){\bf 1}_{\mathcal{E}^{x}_{N}},\qquad\lambda>0,\quad\mathfrak{g}:S_{\star}\to\mathbb{R}.\label{eq:res-mic}
\end{equation}
Namely, the condition is that the microscopic solution $F_{N}$ to
\eqref{eq:res-mic} is asymptotically flat in each valley, and its
value is close to the solution $\mathfrak{f}$ to the following macroscopic
resolvent equation:
\begin{equation}
(\lambda-\mathfrak{L}_{\mathbb{X}})\mathfrak{f}=\mathfrak{g},\label{eq:res-mac}
\end{equation}
where, recall from \eqref{eq:lMC-gen}, $\mathfrak{L}_{\mathbb{X}}$
is the infinitesimal generator of the limiting Markov chain $\mathbb{X}$.
It is clear that both resolvent equations admit unique solutions since
the generators are non-positive and $\lambda>0$. The precise statement
can be summarized as:
\begin{thm}
\label{thm:res}For any fixed $\lambda>0$ and $\mathfrak{g}:S_{\star}\to\mathbb{R}$,
\[
\lim_{N\to\infty}\sup_{\bm{\eta}\in\mathcal{E}^{x}_{N}}\left|F_{N}(\bm{\eta})-\mathfrak{f}(x)\right|=0\qquad\text{for each}\quad x\in S_{\star}.
\]
\end{thm}

The previous theorem, which is called condition $\mathfrak{R}$ in
\cite{LMS25}, is equivalent to parts (2) and (3) of the main theorem,
as verified in \cite[Theorem 2.3]{LMS25}. In practice, this theorem
can be divided into the following two propositions:
\begin{prop}[Condition $\mathfrak{R}^{(1)}$]
\label{prop1}The solution $F_{N}$ to \eqref{eq:res-mic} is asymptotically
flat in each $\mathcal{E}^{x}_{N}$:
\[
\lim_{N\to\infty}\max_{\bm{\eta},\bm{\eta'}\in\mathcal{E}^{x}_{N}}|F_{N}(\bm{\eta})-F_{N}(\bm{\eta'})|=0\qquad\text{for each}\quad x\in S_{\star}.
\]
\end{prop}

Define $f_{N}:S_{\star}\to\mathbb{R}$ as
\begin{equation}
f_{N}(x):=\frac{1}{\nu_{N}(\mathcal{E}^{x}_{N})}\sum_{\bm{\eta}\in\mathcal{E}^{x}_{N}}\nu_{N}(\bm{\eta})F_{N}(\bm{\eta}),\label{eq:fN-def}
\end{equation}
which is the $\nu_{N}$-average of the resolvent solution $F_{N}$
in $\mathcal{E}^{x}_{N}$.
\begin{prop}[Condition $\mathfrak{R}^{(2)}_{\mathbb{X}}$]
\label{prop2}Recall the solution $\mathfrak{f}$ from \eqref{eq:res-mac}.
For each $x\in S_{\star}$,
\[
\lim_{N\to\infty}f_{N}(x)=\mathfrak{f}(x).
\]
\end{prop}

The condensation result (Theorem \ref{thm:cond}), condition $\mathfrak{R}^{(1)}$
(Proposition \ref{prop1}), and Theorem \ref{thm:main}-(1) follow
from the stationary analysis of $\nu_{N}$ and elementary potential
theory.

\subsubsection*{Potential Theory}

As defined in \eqref{eq:eq-pot-RW}, for nonempty disjoint subsets
$\mathcal{U},\mathcal{V}\subset\Omega_{N}$, let us denote by $\mathfrak{H}_{\mathcal{U},\mathcal{V}}(\bm{\eta}):=\mathbb{P}^{N}_{\bm{\eta}}[H_{\mathcal{U}}<H_{\mathcal{V}}]$
the equilibrium potential between $\mathcal{U},\mathcal{V}$, which
is the unique solution to
\[
\begin{cases}
\mathcal{L}_{N}F\equiv0 & \text{on}\enspace\Omega_{N}\setminus(\mathcal{U}\cup\mathcal{V}),\\
F\equiv1 & \text{on}\enspace\mathcal{U},\\
F\equiv0 & \text{on}\enspace\mathcal{V}.
\end{cases}
\]
Denote by ${\rm CAP}_{N}(\mathcal{U},\mathcal{V})$ the corresponding
capacity between $\mathcal{U},\mathcal{V}$:
\[
{\rm CAP}_{N}(\mathcal{U},\mathcal{V}):=\sum_{\bm{\eta}\in\Omega_{N}}\nu_{N}(\bm{\eta})\mathfrak{H}_{\mathcal{U},\mathcal{V}}(\bm{\eta})(-\mathcal{L}_{N}\mathfrak{H}_{\mathcal{U},\mathcal{V}})(\bm{\eta})=:\langle\mathfrak{H}_{\mathcal{U},\mathcal{V}},-\mathcal{L}_{N}\mathfrak{H}_{\mathcal{U},\mathcal{V}}\rangle_{\nu_{N}},
\]
where $\langle\cdot,\cdot\rangle_{\nu_{N}}$ denotes the inner product
with respect to $\nu_{N}$. More generally, the \emph{Dirichlet form}
of a function $F\in\mathbb{R}^{\Omega_{N}}$ is defined as
\[
\mathcal{D}_{N}(F):=\langle F,-\mathcal{L}_{N}F\rangle_{\nu_{N}}=\frac{1}{2}\sum_{\eta\in\Omega_{N}}\sum_{i\in\{1,2\}}\sum_{x,y\in S}\nu_{N}(\bm{\eta})\bm{g}_{i}(\bm{\eta}(x))r_{i}(x,y)\left(F(\bm{\eta}-\delta^{i}_{x}+\delta^{i}_{y})-F(\bm{\eta})\right)^{2}.
\]
The second equality follows easily from \eqref{eq:gen} and a change
of variables. In this sense,
\[
{\rm CAP}_{N}(\mathcal{U},\mathcal{V})=\mathcal{D}_{N}(\mathfrak{H}_{\mathcal{U},\mathcal{V}}).
\]
The two-species zero-range process satisfies the following \emph{sector
condition}, as an analogous property \cite[Proposition 4.2]{Seo19}
of the original zero-range process. We prove the next lemma in Section
\ref{sec3}.
\begin{lem}[Sector Condition]
\label{lem:sector}For any $F,G\in\mathbb{R}^{\Omega_{N}}$, we have
\[
\langle G,-\mathcal{L}_{N}F\rangle^{2}_{\nu_{N}}\leq c\mathcal{D}_{N}(F)\mathcal{D}_{N}(G).
\]
\end{lem}

The sector condition permits us to calculate the capacity of the original
process in terms of the capacity of its \emph{symmetrized} version,
which is particularly useful since the latter becomes a reversible
process. We explain this briefly below. Recall that $R_{N}(\cdot,\cdot)$
denotes the transition rate function of $\bm{\eta}_{N}$ and $\nu_{N}$
is its stationary state. Then, the adjoint jump rate function $R^{*}_{N}:\Omega_{N}\times\Omega_{N}\to[0,\infty)$
is defined as
\begin{equation}
R^{*}_{N}(\eta,\xi):=\frac{\nu_{N}(\xi)}{\nu_{N}(\eta)}R_{N}(\xi,\eta).\label{eq:adj-rate}
\end{equation}
Then, $R^{*}_{N}$ generates an irreducible Markov chain $\bm{\eta}^{*}_{N}$
on $\Omega_{N}$, with the same stationary state $\nu_{N}$, and its
infinitesimal generator $\mathcal{L}^{*}_{N}$ satisfies
\begin{equation}
\langle G,\mathcal{L}_{N}F\rangle_{\nu_{N}}=\langle\mathcal{L}^{*}_{N}G,F\rangle_{\nu_{N}}.\label{eq:gen-adj}
\end{equation}
It follows from a direct calculation that $\mathcal{L}^{*}_{N}$ also
has an explicit form:
\[
\mathcal{L}^{*}_{N}F(\bm{\eta})=\sum_{i\in\{1,2\}}\sum_{x,y\in S}\bm{g}_{i}(\bm{\eta}(x))r^{*}_{i}(x,y)\left(F(\bm{\eta}-\delta^{i}_{x}+\delta^{i}_{y})-F(\bm{\eta})\right),
\]
where $r^{*}_{i}:S\times S\to[0,\infty)$ is the adjoint jump rate
function of the original random walk $r_{i}:S\times S\to[0,\infty)$
defined as 
\[
r^{*}_{i}(x,y):=\frac{m_{i}(y)}{m_{i}(x)}r_{i}(y,x).
\]
In particular, $\bm{\eta}^{*}_{N}$ is again a two-species zero-range
process on $\Omega_{N}$. It is clear that $\mathcal{D}_{N}$ is again
the Dirichlet form of the adjoint process:
\[
\mathcal{D}_{N}(F)=\langle F,-\mathcal{L}_{N}F\rangle_{\nu_{N}}=\langle\mathcal{L}^{*}_{N}F,-F\rangle_{\nu_{N}}=\langle F,-\mathcal{L}^{*}_{N}F\rangle_{\nu_{N}}.
\]
Thus, we have an alternative representation:
\begin{equation}
\mathcal{D}_{N}(F)=\frac{1}{2}\sum_{\eta\in\Omega_{N}}\sum_{i\in\{1,2\}}\sum_{x,y\in S}\nu_{N}(\bm{\eta})\bm{g}_{i}(\bm{\eta}(x))r^{*}_{i}(x,y)\left(F(\bm{\eta}-\delta^{i}_{x}+\delta^{i}_{y})-F(\bm{\eta})\right)^{2}.\label{eq:Diri-adj}
\end{equation}
Let $\mathfrak{H}^{*}_{\mathcal{U},\mathcal{V}}$ be the equilibrium
potential of the adjoint process $\bm{\eta}^{*}_{N}$ between $\mathcal{U},\mathcal{V}$.

Define yet another jump rate function $R^{s}_{N}:\Omega_{N}\times\Omega_{N}\to[0,\infty)$
as
\[
R^{s}_{N}(\bm{\eta},\bm{\eta'}):=\frac{R_{N}(\bm{\eta},\bm{\eta'})+R^{*}_{N}(\bm{\eta},\bm{\eta'})}{2},
\]
which again defines an irreducible Markov chain $\bm{\eta}^{s}_{N}$
on $\Omega_{N}$ with the same stationary state $\nu_{N}$. We call
this the \emph{symmetrized} process of $\bm{\eta}_{N}$. Denote by
${\rm CAP}^{*}_{N}(\mathcal{U},\mathcal{V})$ and ${\rm CAP}^{s}_{N}(\mathcal{U},\mathcal{V})$
the capacities between disjoint nonempty sets $\mathcal{U},\mathcal{V}$
with respect to the adjoint process $\bm{\eta}^{*}_{N}$ and the symmetrized
process $\bm{\eta}^{s}_{N}$, respectively. It is well known (see
e.g. \cite[Section 2]{GL14}) that
\begin{equation}
{\rm CAP}^{s}_{N}(\mathcal{U},\mathcal{V})\le{\rm CAP}_{N}(\mathcal{U},\mathcal{V})={\rm CAP}^{*}_{N}(\mathcal{U},\mathcal{V})\le c{\rm CAP}^{s}_{N}(\mathcal{U},\mathcal{V}),\label{eq:CAPCAP}
\end{equation}
where $c>0$ is the sector constant given in Lemma \ref{lem:sector}
which is independent of $N$. Thus, the capacity ${\rm CAP}_{N}$
has the same scale as the symmetrized capacity ${\rm CAP}^{s}_{N}$
as $N\to\infty$.

Since the symmetrized process $\bm{\eta}^{s}_{N}$ is reversible with
respect to $\nu_{N}$, it is much easier to estimate its capacity.
The following two variational principles are typical (see \cite[Theorems 7.33 and 7.37]{BdH15}):
\begin{itemize}
\item \textbf{(Dirichlet principle)} For any $F:\Omega_{N}\to\mathbb{R}$
such that $F|_{\mathcal{U}}=1$ and $F|_{\mathcal{V}}=0$, we have
\begin{equation}
{\rm CAP}^{s}_{N}(\mathcal{U},\mathcal{V})\le\mathcal{D}_{N}(F).\label{eq:DP}
\end{equation}
\item \textbf{(Thomson principle)} For any sequence $\mathcal{U}\ni\bm{\eta}_{0},\bm{\eta}_{1},\dots,\bm{\eta}_{k}\in\mathcal{V}$
such that $R^{s}_{N}(\bm{\eta}_{m},\bm{\eta}_{m+1})>0$ for each $m\in\llbracket0,k-1\rrbracket$,
\begin{equation}
\frac{1}{{\rm CAP}^{s}_{N}(\mathcal{U},\mathcal{V})}\le\sum^{k-1}_{m=0}\frac{1}{\nu_{N}(\bm{\eta}_{m})R^{s}_{N}(\bm{\eta}_{m},\bm{\eta}_{m+1})}.\label{eq:TP}
\end{equation}
\end{itemize}
On behalf of these preliminaries, we have the following proposition,
to be proved in Section \ref{sec3}. This is the main ingredient to
prove part (1) of Theorem \ref{thm:main}.
\begin{prop}[Capacity Bound]
\label{prop3}Recall $\bm{\xi}^{x}_{N}$ from \eqref{eq:xiNx}. For
each $x\in S_{\star}$,
\[
\lim_{N\to\infty}\frac{{\rm CAP}_{N}(\mathcal{E}^{x}_{N},\Omega_{N}\setminus\tilde{\mathcal{E}}^{x}_{N})}{\inf_{\bm{\eta}\in\mathcal{E}^{x}_{N}\setminus\{\bm{\xi}^{x}_{N}\}}{\rm CAP}_{N}(\bm{\eta},\bm{\xi}^{x}_{N})}=0.
\]
\end{prop}

\subsubsection*{Idea to Prove Condition $\mathfrak{R}^{(2)}_{\mathbb{X}}$: $H^{1}$-Approximation}

Now, we heuristically explain how one can prove condition $\mathfrak{R}^{(2)}_{\mathbb{X}}$.
Let us fix $x\in S_{\star}$. The key idea is to construct an asymptotic
test function (which will later be denoted as $\mathbb{W}_{x}$) that
approximates, in the $H^{1}$-norm, the adjoint equilibrium potential
\[
\mathfrak{H}^{*}_{x}:=\mathfrak{H}^{*}_{\mathcal{E}^{x}_{N},\mathcal{E}_{N}\setminus\mathcal{E}^{x}_{N}}
\]
between $\mathcal{E}^{x}_{N}$ and $\mathcal{E}_{N}\setminus\mathcal{E}^{x}_{N}$.
Recall that $F_{N}:\Omega_{N}\to\mathbb{R}$ solves \eqref{eq:res-mic}.
Let us multiply $\mathfrak{H}^{*}_{x}$ at both sides and integrate
them with respect to the stationary profile $\nu_{N}$. Then,
\begin{equation}
\lambda\sum_{\bm{\eta}\in\Omega_{N}}F_{N}(\bm{\eta})\mathfrak{H}^{*}_{x}(\bm{\eta})\nu_{N}(\bm{\eta})-\theta_{N}\sum_{\bm{\eta}\in\Omega_{N}}\mathcal{L}_{N}F_{N}(\bm{\eta})\mathfrak{H}^{*}_{x}(\bm{\eta})\nu_{N}(\bm{\eta})=\sum_{\bm{\eta}\in\Omega_{N}}\sum_{y\in S_{\star}}\mathfrak{g}(y){\bf 1}_{\mathcal{E}^{y}_{N}}(\bm{\eta})\mathfrak{H}^{*}_{x}(\bm{\eta})\nu_{N}(\bm{\eta}).\label{eq:res-int}
\end{equation}
Since $\mathfrak{H}^{*}_{x}$ vanishes on $\mathcal{E}_{N}\setminus\mathcal{E}^{x}_{N}$,
the right-hand side becomes
\begin{equation}
\sum_{\bm{\eta}\in\Omega_{N}}\sum_{y\in S_{\star}}\mathfrak{g}(y){\bf 1}_{\mathcal{E}^{y}_{N}}(\bm{\eta})\mathfrak{H}^{*}_{x}(\bm{\eta})\nu_{N}(\bm{\eta})=\sum_{\bm{\eta}\in\mathcal{E}^{x}_{N}}\mathfrak{g}(x)\nu_{N}(\bm{\eta})=\mathfrak{g}(x)\nu_{N}(\mathcal{E}^{x}_{N}).\label{eq:int-1}
\end{equation}
Similarly, the first term in the left-hand side of \eqref{eq:res-int}
becomes (cf. \eqref{eq:fN-def})
\begin{equation}
\lambda\sum_{\bm{\eta}\in\mathcal{E}^{x}_{N}}F_{N}(\bm{\eta})\nu_{N}(\bm{\eta})+\lambda\sum_{\bm{\eta}\in\Delta_{N}}F_{N}(\bm{\eta})\mathfrak{H}^{*}_{x}(\bm{\eta})\nu_{N}(\bm{\eta})=\lambda\nu_{N}(\mathcal{E}^{x}_{N})f_{N}(x)+o(\nu_{N}(\mathcal{E}^{x}_{N})),\label{eq:int-2}
\end{equation}
where the equality holds since the functions $F_{N},\mathfrak{H}^{*}_{x}$
are uniformly bounded (see \eqref{eq:FN-bdd}) and $\nu_{N}(\Delta_{N})\ll\nu_{N}(\mathcal{E}^{x}_{N})$
by Theorem \ref{thm:cond}. The remaining second term at the left-hand
side of \eqref{eq:res-int} becomes
\[
-\theta_{N}\langle\mathcal{L}_{N}F_{N},\mathfrak{H}^{*}_{x}\rangle_{\nu_{N}}=\theta_{N}\langle F_{N},-\mathcal{L}^{*}_{N}\mathfrak{H}^{*}_{x}\rangle_{\nu_{N}},
\]
by \eqref{eq:gen-adj}. By the definition of $\mathfrak{H}^{*}_{x}$
as an equilibrium potential, $\mathcal{L}^{*}_{N}\mathfrak{H}^{*}_{x}$
vanishes outside $\mathcal{E}_{N}$. Thus,
\begin{align*}
\theta_{N}\langle F_{N},-\mathcal{L}^{*}_{N}\mathfrak{H}^{*}_{x}\rangle_{\nu_{N}} & =\theta_{N}\sum_{\bm{\eta}\in\mathcal{E}_{N}}F_{N}(\bm{\eta})(-\mathcal{L}^{*}_{N}\mathfrak{H}^{*}_{x})(\bm{\eta})\nu_{N}(\bm{\eta})\\
 & \simeq\theta_{N}\left(f_{N}(x)\sum_{\bm{\eta}\in\mathcal{E}^{x}_{N}}(-\mathcal{L}^{*}_{N}\mathfrak{H}^{*}_{x})(\bm{\eta})\nu_{N}(\bm{\eta})+\sum_{y\in S_{\star}\setminus\{x\}}f_{N}(y)\sum_{\bm{\eta}\in\mathcal{E}^{y}_{N}}(-\mathcal{L}^{*}_{N}\mathfrak{H}^{*}_{x})(\bm{\eta})\nu_{N}(\bm{\eta})\right).
\end{align*}
Here, the asymptotic holds since $F_{N}$ is flat in each metastable
valley by Proposition \ref{prop1}. By definition we have
\[
{\rm CAP}_{N}(\mathcal{E}^{x}_{N},\mathcal{E}_{N}\setminus\mathcal{E}^{x}_{N})=\langle\mathfrak{H}^{*}_{x},-\mathcal{L}^{*}_{N}\mathfrak{H}^{*}_{x}\rangle_{\nu_{N}}=\sum_{\bm{\eta}\in\mathcal{E}^{x}_{N}}(-\mathcal{L}^{*}_{N}\mathfrak{H}^{*}_{x})(\bm{\eta})\nu_{N}(\bm{\eta})=\sum_{y\in S_{\star}\setminus\{x\}}\sum_{\bm{\eta}\in\mathcal{E}^{y}_{N}}(\mathcal{L}^{*}_{N}\mathfrak{H}^{*}_{x})(\bm{\eta})\nu_{N}(\bm{\eta}).
\]
Thus, we conclude that
\begin{equation}
-\theta_{N}\langle\mathcal{L}_{N}F_{N},\mathfrak{H}^{*}_{x}\rangle_{\nu_{N}}\simeq\sum_{y\in S_{\star}\setminus\{x\}}\left(\theta_{N}\sum_{\bm{\eta}\in\mathcal{E}^{y}_{N}}(-\mathcal{L}^{*}_{N}\mathfrak{H}^{*}_{x})(\bm{\eta})\nu_{N}(\bm{\eta})\right)\left[f_{N}(y)-f_{N}(x)\right]\simeq-\nu_{N}(\mathcal{E}^{x}_{N})\mathfrak{L}_{\mathbb{X}}f_{N}(x),\label{eq:int-3}
\end{equation}
provided that
\[
\theta_{N}\sum_{\bm{\eta}\in\mathcal{E}^{y}_{N}}\mathcal{L}^{*}_{N}\mathfrak{H}^{*}_{x}(\bm{\eta})\nu_{N}(\bm{\eta})\simeq\nu_{N}(\mathcal{E}^{x}_{N})\mathfrak{r}(x,y).
\]
Indeed, the limiting rate function $\mathfrak{r}(\cdot,\cdot)$ is
determined by this asymptotic. Combining these heuristics at \eqref{eq:int-1},
\eqref{eq:int-2}, and \eqref{eq:int-3}, we obtain that
\[
\lambda f_{N}(x)-\mathfrak{L}_{\mathbb{X}}f_{N}(x)\simeq\mathfrak{g}(x)\qquad\text{for each}\quad x\in S_{\star},\qquad\text{thus}\qquad(\lambda-\mathfrak{L}_{\mathbb{X}})(f_{N}-\mathfrak{f})\simeq0.
\]
The inverted operator $(\lambda-\mathfrak{L}_{\mathbb{X}})^{-1}$
is bounded, thus this would give us $f_{N}-\mathfrak{f}\simeq0$ as
desired. 

Therefore, our objective is to find a nice test object $\mathbb{W}_{x}$
which approximates $\mathfrak{H}^{*}_{x}$ well enough such that the
above heuristics are justified and we have the desired asymptotic
equation as $N\to\infty$. One can easily see that this approximation
should be in the $H^{1}$ sense, since the value of $\langle F,\mathcal{L}^{*}_{N}\mathfrak{H}^{*}_{x}\rangle_{\nu_{N}}$
must be close to that of $\langle F,\mathcal{L}^{*}_{N}\mathbb{W}_{x}\rangle_{\nu_{N}}$
for all functions $F\in L^{2}(\nu_{N})$.

\subsubsection*{Organization of the Article}

Following the proof ideas presented previously in this subsection,
the rest of the article is organized as follows. In Section \ref{sec2},
we conduct an analysis on the stationary state $\nu_{N}$ and prove
Theorem \ref{thm:cond}. In Section \ref{sec3}, we analyze the local
metastable state $\mathcal{E}^{x}_{N}$, $x\in S_{\star}$, and prove
Propositions \ref{prop1} and \ref{prop3}. As a consequence, at the
end of that section, we prove part (1) of Theorem \ref{thm:main}.
In Section \ref{sec4}, we construct the test function $\mathbb{W}_{x}$
explained above. Finally, in Section \ref{sec5}, we prove Proposition
\ref{prop2} and thus prove the remaining parts (2) and (3) of Theorem
\ref{thm:main}. In the appendix, we gather several lemmas regarding
summation asymptotics that are useful in the main context.

\section{\label{sec2}Stationary State Analysis}

In this section, we analyze the stationary state $\nu_{N}$ and prove
Theorem \ref{thm:cond}. To achieve this, we first present a series
of lemmas regarding the normalizing constant $Z_{N}$ (cf. \eqref{eq:ZN-def}).

Given a threshold $\ell<N/2$, we may divide the set $\Omega_{N}$
into three subsets as
\begin{equation}
\Omega_{N}=\bigcup_{x\in S_{\star}}\Omega^{x,\ell}_{N}\cup\bigcup_{y\in S\setminus S_{\star}}\Omega^{y,\ell}_{N}\cup\Xi^{\ell}_{N},\label{eq:Omega-dec}
\end{equation}
where
\begin{equation}
\Xi^{\ell}_{N}:=\left\{ \bm{\eta}\in\Omega_{N}:|\bm{\eta}(x)|<N-\ell\quad\text{for all}\enspace x\in S\right\} ,\label{eq:XiNell}
\end{equation}
and
\begin{equation}
\Omega^{x,\ell}_{N}:=\left\{ \bm{\eta}\in\Omega_{N}:|\bm{\eta}(x)|\ge N-\ell\right\} \qquad\text{for each}\quad x\in S.\label{eq:OmegaNxell}
\end{equation}
Since $N>2\ell$, the display \eqref{eq:Omega-dec} serves as a partition,
i.e., a disjoint decomposition.

\subsection{\label{sec2.1}Supercritical Case: $\alpha>1$}

In this subsection, let us assume that $\alpha>1$ and prove the first
part of Theorem \ref{thm:cond}.
\begin{lem}
\label{lem:s1}There exists $c>0$ such that
\[
1\le Z_{N}\le c.
\]
\end{lem}

\begin{proof}
Fix one $x\in S_{\star}$ and consider the configuration $\bm{\eta}:=A\delta^{1}_{x}+B\delta^{2}_{x}\in\Omega_{N}$
(cf. \eqref{eq:delta-zi-def}). Then,
\[
Z_{N}\geq\frac{N^{\alpha}}{\binom{N}{A}}\frac{\bm{m}^{\bm{\eta}}}{\bm{g!}(\bm{\eta})}=1
\]
leads to the lower bound. For the upper bound, for simplicity let
us denote by $S=\llbracket1,\kappa\rrbracket$. From \eqref{eq:nuN-def},
\[
Z_{N}=\frac{N^{\alpha}}{{N \choose A}}\sum_{\bm{\eta}\in\Omega_{N}}\frac{\bm{m}^{\bm{\eta}}}{\bm{g!}(\bm{\eta})}=\frac{N^{\alpha}}{{N \choose A}}\sum_{\substack{p_{1},\dots,p_{\kappa}\ge0:\\
p_{1}+\cdots+p_{\kappa}=N
}
}\sum_{\substack{q_{1},\dots,q_{\kappa}\ge0:\\
q_{1}+\cdots+q_{\kappa}=A
}
}\prod^{\kappa}_{j=1}\frac{{p_{j} \choose q_{j}}m_{1}(j)^{q_{j}}m_{2}(j)^{p_{j}-q_{j}}}{\mathfrak{a}(p_{j})},
\]
where $p_{j}$ denotes the number of particles at $j\in\llbracket1,\kappa\rrbracket$,
and $q_{j}$ denotes the number of type-$1$ particles at $j$. Then,
Lemma \ref{lem:comb-bound} implies that
\[
Z_{N}\le\sum_{\substack{p_{1},\dots,p_{\kappa}\ge0:\\
p_{1}+\cdots+p_{\kappa}=N
}
}\frac{cN^{\alpha}}{\mathfrak{a}(p_{1})\mathfrak{a}(p_{2})\cdots\mathfrak{a}(p_{\kappa})}\prod^{\kappa}_{j=1}\left(\frac{Am_{1}(j)+Bm_{2}(j)}{N}\right)^{p_{j}}\le\sum_{\substack{p_{1},\dots,p_{\kappa}\ge0:\\
p_{1}+\cdots+p_{\kappa}=N
}
}\frac{cN^{\alpha}}{\mathfrak{a}(p_{1})\mathfrak{a}(p_{2})\cdots\mathfrak{a}(p_{\kappa})},
\]
where the second inequality follows from trivial bounds $m_{1}(j),m_{2}(j)\le1$.
Thus, Lemma \ref{lem:asymp2} implies the desired result.
\end{proof}

\begin{rem}
\label{rem:s1}One can easily notice that the previous lemma holds
without the asymptotic condition \eqref{eq:ANBN}, i.e., it holds
for any $A,B\ge0$ such that $A+B=N$.
\end{rem}

Next, we analyze the set $\Xi^{\ell}_{N}$ defined at \eqref{eq:XiNell}.
\begin{lem}
\label{lem:s2}For $\ell<N/2$ we have
\[
\frac{N^{\alpha}}{{N \choose A}}\sum_{\bm{\eta}\in\Xi^{\ell}_{N}}\frac{\bm{m}^{\bm{\eta}}}{\bm{g!}(\bm{\eta})}\leq\frac{c}{(\ell+1)^{\alpha-1}}.
\]
\end{lem}

\begin{proof}
Let us adopt the same notation as in the proof of Lemma \ref{lem:s1}.
By definition and Lemma \ref{lem:comb-bound}, the left-hand side
becomes
\[
\frac{N^{\alpha}}{{N \choose A}}\sum_{\substack{0\le p_{1},\dots,p_{\kappa}<N-\ell:\\
p_{1}+\cdots+p_{\kappa}=N
}
}\sum_{\substack{q_{1},\dots,q_{\kappa}\ge0:\\
q_{1}+\cdots+q_{\kappa}=A
}
}\prod^{\kappa}_{j=1}\frac{{p_{j} \choose q_{j}}m_{1}(j)^{q_{j}}m_{2}(j)^{p_{j}-q_{j}}}{\mathfrak{a}(p_{j})}\le\sum_{\substack{0\le p_{1},\dots,p_{\kappa}<N-\ell:\\
p_{1}+\cdots+p_{\kappa}=N
}
}\frac{cN^{\alpha}}{\mathfrak{a}(p_{1})\mathfrak{a}(p_{2})\cdots\mathfrak{a}(p_{\kappa})}.
\]
By Lemma \ref{lem:asymp3}, we conclude that
\[
\frac{N^{\alpha}}{{N \choose A}}\sum_{\bm{\eta}\in\Xi^{\ell}_{N}}\frac{\bm{m}^{\bm{\eta}}}{\bm{g!}(\bm{\eta})}\le\frac{c}{(\ell+1)^{\alpha-1}},
\]
for a constant $c>0$ independent of $N$.
\end{proof}

Before moving on, we go one step further. The next two lemmas will
be exploited in Section \ref{sec5.2}. Define
\begin{equation}
\Xi^{\ell,\ell'}_{N}:=\left\{ \bm{\eta}\in\Omega_{N}:|\bm{\eta}(x)|<N-\ell\enspace\forall x\in S,\quad|\bm{\eta}(x)|+|\bm{\eta}(y)|<N-\ell'\enspace\forall x\ne y\in S\right\} .\label{eq:XiNellell'}
\end{equation}

\begin{lem}
\label{lem:s3}For $\ell'<\ell/2<N/4$ we have
\[
\frac{N^{\alpha}}{{N \choose A}}\sum_{\bm{\eta}\in\Xi^{\ell,\ell'}_{N}}\frac{\bm{m}^{\bm{\eta}}}{\bm{g!}(\bm{\eta})}\leq\frac{c}{(\ell+1)^{\alpha-1}(\ell'+1)^{\alpha-1}}.
\]
\end{lem}

\begin{proof}
Similarly as before, the left-hand side of the lemma becomes
\begin{align*}
\frac{N^{\alpha}}{{N \choose A}}\sum_{\substack{0\le p_{1},\dots,p_{\kappa}<N-\ell:\\
p_{1}+\cdots+p_{\kappa}=N,\\
p_{j}+p_{j'}<N-\ell',\ \forall j\ne j'
}
}\sum_{\substack{q_{1},\dots,q_{\kappa}\ge0:\\
q_{1}+\cdots+q_{\kappa}=A
}
}\prod^{\kappa}_{j=1}\frac{{p_{j} \choose q_{j}}m_{1}(j)^{q_{j}}m_{2}(j)^{p_{j}-q_{j}}}{\mathfrak{a}(p_{j})} & \le\sum_{\substack{0\le p_{1},\dots,p_{\kappa}<N-\ell:\\
p_{1}+\cdots+p_{\kappa}=N,\\
p_{j}+p_{j'}<N-\ell',\ \forall j\ne j'
}
}\frac{cN^{\alpha}}{\mathfrak{a}(p_{1})\mathfrak{a}(p_{2})\cdots\mathfrak{a}(p_{\kappa})}\\
 & \le\frac{c}{(\ell+1)^{\alpha-1}(\ell'+1)^{\alpha-1}},
\end{align*}
now via Lemma \ref{lem:asymp4} at the last inequality.
\end{proof}

\begin{lem}
\label{lem:s4}For $\ell'<\ell/2<N/4$ and fixed $x\ne y\in S$,
\[
\nu_{N}\left(\bm{\eta}\in\Omega_{N}:|\bm{\eta}(x)|,|\bm{\eta}(y)|<N-\ell,\quad|\bm{\eta}(x)|+|\bm{\eta}(y)|=N-\ell'\right)\leq\frac{c}{(\ell+1)^{\alpha-1}(\ell'+1)^{\alpha}}.
\]
\end{lem}

\begin{proof}
Write $x=1$, $y=2$, and $S=\llbracket1,\kappa\rrbracket$. Applying
$Z_{N}\ge1$ from Lemma \ref{lem:s1} and using the same representation,
the left-hand side is bounded by
\[
cN^{\alpha}\sum^{N-\ell-1}_{p_{1}=\ell-\ell'+1}\frac{1}{\mathfrak{a}(p_{1})\mathfrak{a}(N-\ell'-p_{1})}\sum_{\substack{p_{3},\dots,p_{\kappa}\ge0:\\
p_{3}+\cdots+p_{\kappa}=\ell'
}
}\frac{1}{\mathfrak{a}(p_{3})\cdots\mathfrak{a}(p_{\kappa})}\le\frac{c'}{(\ell'+1)^{\alpha}}\sum^{N-\ell-1}_{p_{1}=\ell-\ell'+1}\frac{N^{\alpha}}{p^{\alpha}_{1}(N-\ell'-p_{1})^{\alpha}}.
\]
Above, Lemma \ref{lem:asymp2} was also used. Since $(\ell-\ell'+1)+(N-\ell-1)=N-\ell'\ge N/2$,
we may exploit the symmetry to bound as
\[
\sum^{N-\ell-1}_{p_{1}=\ell-\ell'+1}\frac{N^{\alpha}}{p^{\alpha}_{1}(N-\ell'-p_{1})^{\alpha}}\le c\sum^{\frac{N-\ell'}{2}}_{p_{1}=\ell-\ell'+1}\frac{1}{p^{\alpha}_{1}}\le\frac{c'}{(\ell-\ell'+1)^{\alpha-1}}\le\frac{c''}{(\ell+1)^{\alpha-1}},
\]
via $2(\ell-\ell'+1)>\ell+1$. Combining the last two displays finishes
the proof.
\end{proof}

We return to the main storyline of this subsection. Our next step
is to study $\Omega^{x,\ell}_{N}$ (cf. \eqref{eq:OmegaNxell}).
\begin{lem}
\label{lem:s5}For each $x\in S\setminus S_{\star}$ and $\ell\in\llbracket0,N\rrbracket$,
\[
\frac{N^{\alpha}}{{N \choose A}}\sum_{\bm{\eta}\in\Omega^{x,\ell}_{N}}\frac{\bm{m}^{\bm{\eta}}}{\bm{g!}(\bm{\eta})}\leq c\upsilon^{N-\ell},
\]
where $\upsilon\in(0,1)$ does not depend on $N$ and $\ell$.
\end{lem}

\begin{proof}
Let $S=\llbracket1,\kappa\rrbracket$ and $x=1$. Denote by $p_{1}\in\llbracket N-\ell,N\rrbracket$
the number of particles at $x$. Then, we may write
\[
\frac{N^{\alpha}}{{N \choose A}}\sum_{\bm{\eta}\in\Omega^{x,\ell}_{N}}\frac{\bm{m}^{\bm{\eta}}}{\bm{g!}(\bm{\eta})}=\frac{N^{\alpha}}{{N \choose A}}\sum_{\substack{p_{1},\dots,p_{\kappa}\ge0:\\
p_{1}+\cdots+p_{\kappa}=N,\\
p_{1}\ge N-\ell
}
}\sum_{\substack{q_{1},\dots,q_{\kappa}\ge0:\\
q_{1}+\cdots+q_{\kappa}=A
}
}\prod^{\kappa}_{j=1}\frac{{p_{j} \choose q_{j}}m_{1}(j)^{q_{j}}m_{2}(j)^{p_{j}-q_{j}}}{\mathfrak{a}(p_{j})}.
\]
By Lemma \ref{lem:comb-bound}, the right-hand side is bounded above
by
\[
cN^{\alpha}\sum_{\substack{p_{1},\dots,p_{\kappa}\ge0:\\
p_{1}+\cdots+p_{\kappa}=N,\\
p_{1}\ge N-\ell
}
}\frac{1}{\mathfrak{a}(p_{1})\mathfrak{a}(p_{2})\cdots\mathfrak{a}(p_{\kappa})}\left(\frac{Am_{1}(x)+Bm_{2}(x)}{N}\right)^{p_{1}}\le c\left(\frac{Am_{1}(x)+Bm_{2}(x)}{N}\right)^{N-\ell},
\]
where Lemma \ref{lem:asymp2} was used at the inequality. Since $x\in S\setminus S_{\star}$,
we have $m_{1}(x)\wedge m_{2}(x)<1$. Thus by \eqref{eq:ANBN},
\[
\frac{Am_{1}(x)+Bm_{2}(x)}{N}\xrightarrow{N\to\infty}\rho m_{1}(x)+(1-\rho)m_{2}(x)<1.
\]
This guarantees the existence of a positive constant $\upsilon<1$
such that $\max_{y\in S\setminus S_{\star}}\frac{Am_{1}(y)+Bm_{2}(y)}{N}\le\upsilon$
for sufficiently large $N$. This proves the lemma.
\end{proof}

Now, we are ready to prove the first part of Theorem \ref{thm:cond}.
\begin{proof}[Proof of Theorem \ref{thm:cond}-(1)]
 By the symmetry of $\nu_{N}$ in \eqref{eq:nuN-def} between the
sites in $S_{\star}$, it suffices to prove that
\begin{equation}
\lim_{N\to\infty}\nu_{N}(\Delta_{N})=0.\label{eq:cond-1-wts}
\end{equation}
The decomposition \eqref{eq:Omega-dec} implies that
\[
\Delta_{N}=\Omega_{N}\setminus\mathcal{E}_{N}=\bigcup_{x\in S\setminus S_{\star}}\Omega^{x,\ell_{N}}_{N}\cup\Xi^{\ell_{N}}_{N}.
\]
First, by Lemma \ref{lem:s2},
\[
\nu_{N}(\Xi^{\ell_{N}}_{N})=\frac{N^{\alpha}}{Z_{N}{N \choose A}}\sum_{\bm{\eta}\in\Xi^{\ell_{N}}_{N}}\frac{\bm{m}^{\bm{\eta}}}{\bm{g!}(\bm{\eta})}\le\frac{c}{Z_{N}(\ell_{N}+1)^{\alpha-1}}\ll1.
\]
The final bound follows from Lemma \ref{lem:s1} and the fact that
$\ell_{N}\gg1$. Next, by Lemma \ref{lem:s5}, for each $x\in S\setminus S_{\star}$,
\[
\nu_{N}(\Omega^{x,\ell_{N}}_{N})=\frac{N^{\alpha}}{Z_{N}{N \choose A}}\sum_{\bm{\eta}\in\Omega^{x,\ell_{N}}_{N}}\frac{\bm{m}^{\bm{\eta}}}{\bm{g!}(\bm{\eta})}\le\frac{c\upsilon^{N-\ell_{N}}}{Z_{N}}\ll1.
\]
Again, the final bound follows from Lemma \ref{lem:s1} and $N-\ell_{N}\gg1$.
The last two displayed estimates prove \eqref{eq:cond-1-wts} as desired.
\end{proof}

Before moving on, we provide a sharp asymptotic of the normalizing
constant $Z_{N}$ as $N\to\infty$. Recall that $\rho\in(0,1)$ denotes
the fixed macroscopic density value. Let
\begin{equation}
\Gamma_{x}:=\sum_{p\ge0}\frac{(\rho m_{1}(x)+(1-\rho)m_{2}(x))^{p}}{\mathfrak{a}(p)}\qquad\text{for}\quad x\in S,\qquad\Gamma(\alpha):=\sum_{p\ge0}\frac{1}{\mathfrak{a}(p)}.\label{eq:Gamma-def}
\end{equation}
Notice that $\Gamma_{x}=\Gamma(\alpha)$ if and only if $x\in S_{\star}$.
For each nonempty $S_{0}\subset S$, let
\[
\Omega_{S_{0}}(p_{1},p_{2}):=\left\{ \bm{\sigma}=(\sigma_{1},\sigma_{2})\in(\mathbb{N}^{2}_{0})^{S_{0}}:|\sigma_{1}|=p_{1},\enspace|\sigma_{2}|=p_{2}\right\} .
\]
In addition, let
\[
\Omega_{S_{0}}(L):=\left\{ \bm{\sigma}\in(\mathbb{N}^{2}_{0})^{S_{0}}:|\bm{\sigma}|=L\right\} =\bigcup_{\substack{p_{1},p_{2}\ge0:\\
p_{1}+p_{2}=L
}
}\Omega_{S_{0}}(p_{1},p_{2}).
\]

\begin{prop}
\label{prop:ZN-limit}Recall \eqref{eq:kappa-star}. We have
\[
\lim_{N\to\infty}Z_{N}=\kappa_{\star}\Gamma(\alpha)^{\kappa_{\star}-1}\prod_{z\in S\setminus S_{\star}}\Gamma_{z}\in(0,\infty).
\]
\end{prop}

\begin{proof}
By the decomposition \eqref{eq:Omega-dec},
\begin{equation}
Z_{N}=\sum_{x\in S_{\star}}\frac{N^{\alpha}}{{N \choose A}}\sum_{\bm{\eta}\in\Omega^{x,\ell_{N}}_{N}}\frac{\bm{m}^{\bm{\eta}}}{\bm{g!}(\bm{\eta})}+\sum_{x\in S\setminus S_{\star}}\frac{N^{\alpha}}{{N \choose A}}\sum_{\bm{\eta}\in\Omega^{x,\ell_{N}}_{N}}\frac{\bm{m}^{\bm{\eta}}}{\bm{g!}(\bm{\eta})}+\frac{N^{\alpha}}{{N \choose A}}\sum_{\bm{\eta}\in\Xi^{\ell_{N}}_{N}}\frac{\bm{m}^{\bm{\eta}}}{\bm{g!}(\bm{\eta})}.\label{eq:ZN-1}
\end{equation}
As done in the proof of Theorem \ref{thm:cond}-(1), we obtain via
Lemmas \ref{lem:s2} and \ref{lem:s5} that
\begin{equation}
\frac{N^{\alpha}}{{N \choose A}}\sum_{\bm{\eta}\in\Xi^{\ell}_{N}}\frac{\bm{m}^{\bm{\eta}}}{\bm{g!}(\bm{\eta})}\le\frac{c}{(\ell_{N}+1)^{\alpha-1}}\ll1,\qquad\sum_{x\in S\setminus S_{\star}}\frac{N^{\alpha}}{{N \choose A}}\sum_{\bm{\eta}\in\Omega^{x,\ell}_{N}}\frac{\bm{m}^{\bm{\eta}}}{\bm{g!}(\bm{\eta})}\le c\upsilon^{N-\ell_{N}}\ll1.\label{eq:ZN-2}
\end{equation}
For the first term in the right-hand side of \eqref{eq:ZN-1}, by
symmetry,
\begin{equation}
\sum_{x\in S_{\star}}\frac{N^{\alpha}}{{N \choose A}}\sum_{\bm{\eta}\in\Omega^{x,\ell_{N}}_{N}}\frac{\bm{m}^{\bm{\eta}}}{\bm{g!}(\bm{\eta})}=\kappa_{\star}\sum^{\ell_{N}}_{p=0}\sum^{p}_{q=0}\frac{N^{\alpha}\binom{N-p}{A-q}}{(N-p)^{\alpha}\binom{N}{A}}\sum_{\bm{\sigma}\in\Omega_{S\setminus\{x\}}(q,p-q)}\frac{\bm{m}^{\bm{\sigma}}}{\bm{g!}(\bm{\sigma})}.\label{eq:ZN-3}
\end{equation}
At the right-hand side, a site $x\in S_{\star}$ was arbitrarily fixed,
$p$ denotes the number of particles not at $x$, and $q$ denotes
the number of type-$1$ particles not at $x$. First, note that
\begin{equation}
\frac{(A-\ell_{N})^{q}(B-\ell_{N})^{p-q}}{N^{p}}\le\frac{N^{\alpha}\binom{N-p}{A-q}}{(N-p)^{\alpha}\binom{N}{A}}\le\left(\frac{N}{N-\ell_{N}}\right)^{\alpha}\frac{A^{q}B^{p-q}}{(N-\ell_{N})^{p}}.\label{eq:ZN-4}
\end{equation}
Thus, the right-hand side of \eqref{eq:ZN-3} is bounded above by
\[
\kappa_{\star}\left(\frac{N}{N-\ell_{N}}\right)^{\alpha}\sum^{\ell_{N}}_{p=0}\left(\frac{N}{N-\ell_{N}}\right)^{p}\sum^{p}_{q=0}\left(\frac{A}{N}\right)^{q}\left(\frac{B}{N}\right)^{p-q}\sum_{\bm{\sigma}\in\Omega_{S\setminus\{x\}}(q,p-q)}\frac{\bm{m}^{\bm{\sigma}}}{\bm{g!}(\bm{\sigma})}\simeq\kappa_{\star}\sum^{\ell_{N}}_{p=0}\sum_{\bm{\sigma}\in\Omega_{S\setminus\{x\}}(p)}\frac{(\bm{m}_{\rho_{N}})^{\bm{\sigma}}}{\bm{g!}(\bm{\sigma})},
\]
where
\begin{equation}
\rho_{N}:=\frac{A}{N},\label{eq:rhoN-def}
\end{equation}
$(\bm{m}_{\rho_{N}})_{1}:=\rho_{N}m_{1}$, and $(\bm{m}_{\rho_{N}})_{2}:=(1-\rho_{N})m_{2}$.
Taking the limit $N\to\infty$, the right-hand side converges to
\begin{align*}
\kappa_{\star}\sum^{\infty}_{p=0}\sum_{\bm{\sigma}\in\Omega_{S\setminus\{x\}}(p)}\frac{(\bm{m}_{\rho})^{\bm{\sigma}}}{\bm{g!}(\bm{\sigma})} & =\kappa_{\star}\left(\prod_{y\in S_{\star}\setminus\{x\}}\sum^{\infty}_{p=0}\frac{1}{\mathfrak{a}(p)}\right)\left(\prod_{z\in S\setminus S_{\star}}\sum^{\infty}_{p=0}\frac{(\rho m_{1}(z)+(1-\rho)m_{2}(z))^{p}}{\mathfrak{a}(p)}\right)\\
 & =\kappa_{\star}\Gamma(\alpha)^{\kappa_{\star}-1}\prod_{z\in S\setminus S_{\star}}\Gamma_{z}.
\end{align*}
Thus, we have proved that
\[
\limsup_{N\to\infty}\kappa_{\star}\sum^{\ell_{N}}_{p=0}\sum^{p}_{q=0}\frac{N^{\alpha}\binom{N-p}{A-q}}{(N-p)^{\alpha}\binom{N}{A}}\sum_{\bm{\sigma}\in\Omega_{S\setminus\{x\}}(q,p-q)}\frac{\bm{m}^{\bm{\sigma}}}{\bm{g!}(\bm{\sigma})}\le\kappa_{\star}\Gamma(\alpha)^{\kappa_{\star}-1}\prod_{z\in S\setminus S_{\star}}\Gamma_{z}.
\]
Next, applying the lower bound part of \eqref{eq:ZN-4} to \eqref{eq:ZN-3}
yields similarly that
\[
\liminf_{N\to\infty}\kappa_{\star}\sum^{\ell_{N}}_{p=0}\sum^{p}_{q=0}\frac{N^{\alpha}\binom{N-p}{A-q}}{(N-p)^{\alpha}\binom{N}{A}}\sum_{\bm{\sigma}\in\Omega_{S\setminus\{x\}}(q,p-q)}\frac{\bm{m}^{\bm{\sigma}}}{\bm{g!}(\bm{\sigma})}\ge\kappa_{\star}\Gamma(\alpha)^{\kappa_{\star}-1}\prod_{z\in S\setminus S_{\star}}\Gamma_{z}.
\]
Combining the last two displays, along with \eqref{eq:ZN-1} and \eqref{eq:ZN-2},
completes the proof.
\end{proof}

The rest of Section \ref{sec2} is devoted to the analysis of the
critical ($\alpha=1$) and subcritical cases ($\alpha<1$), and may
be safely skipped by readers interested solely in the supercritical
regime, which is the primary focus of this article.

\subsection{\label{sec2.2}Critical Case: $\alpha=1$}

Next, we consider the critical case of $\alpha=1$, which we assume
throughout this subsection. The overall logic here is similar to Section
\ref{sec2.1}, but now the number of sites in $S_{\star}$ starts
to play a crucial role, so we need a more delicate analysis.
\begin{lem}
\label{lem:s6}We have
\[
Z_{N}\asymp(\log N)^{\kappa_{\star}-1}.
\]
\end{lem}

\begin{proof}
We want to prove that there exist two constants $c,c'>0$ independent
of $N$ such that
\begin{equation}
c(\log N)^{\kappa_{\star}-1}\le Z_{N}\le c'(\log N)^{\kappa_{\star}-1}.\label{eq:s4-wts}
\end{equation}
The lower bound part is easy. Recall that
\[
Z_{N}=\frac{N}{{N \choose A}}\sum_{\bm{\eta}\in\Omega_{N}}\frac{\bm{m}^{\bm{\eta}}}{\bm{g!}(\bm{\eta})}.
\]
For simplicity write $S=\llbracket1,\kappa\rrbracket$ and $S_{\star}=\llbracket1,\kappa_{\star}\rrbracket$.
Let us restrict the summation on the configurations in which all particles
are on $S_{\star}$. Then, by writing $p_{j}:=|\bm{\eta}(j)|$ and
$q_{j}:=\eta_{1}(j)$ for each $\bm{\eta}\in\Omega_{N}$, we have
\[
Z_{N}\ge\frac{N}{{N \choose A}}\sum_{\substack{p_{1},\dots,p_{\kappa_{\star}}\ge0:\\
p_{1}+\cdots+p_{\kappa_{\star}}=N
}
}\sum_{\substack{q_{1},\dots,q_{\kappa_{\star}}\ge0:\\
q_{1}+\cdots+q_{\kappa_{\star}}=A
}
}\prod^{\kappa_{\star}}_{j=1}\frac{{p_{j} \choose q_{j}}}{\mathfrak{a}(p_{j})}=N\sum_{\substack{p_{1},\dots,p_{\kappa_{\star}}\ge0:\\
p_{1}+\cdots+p_{\kappa_{\star}}=N
}
}\prod^{\kappa_{\star}}_{j=1}\frac{1}{\mathfrak{a}(p_{j})},
\]
where the equality follows from Vandermonde's identity. By Lemma \ref{lem:asymp2},
the right-hand side has an asymptotic scale of $(\log N)^{\kappa_{\star}-1}$
as $N\to\infty$, thus we have the lower bound.

For the upper bound part of \eqref{eq:s4-wts}, let us write $x_{j}:=m_{1}(j)$
and $y_{j}:=m_{2}(j)$. As done above, we may represent $Z_{N}$ as
\[
Z_{N}=\frac{N}{{N \choose A}}\sum_{\substack{p_{1},\dots,p_{\kappa}\ge0:\\
p_{1}+\cdots+p_{\kappa}=N
}
}\sum_{\substack{q_{1},\dots,q_{\kappa}\ge0:\\
q_{1}+\cdots+q_{\kappa}=A
}
}\prod^{\kappa}_{j=1}\frac{{p_{j} \choose q_{j}}x^{q_{j}}_{j}y^{p_{j}-q_{j}}_{j}}{\mathfrak{a}(p_{j})}.
\]
Applying Lemma \ref{lem:comb-bound} to the right-hand side, we obtain
\[
Z_{N}\le cN\sum_{\substack{p_{1},\dots,p_{\kappa}\ge0:\\
p_{1}+\cdots+p_{\kappa}=N
}
}\left(\frac{1}{\mathfrak{a}(p_{1})\mathfrak{a}(p_{2})\cdots\mathfrak{a}(p_{\kappa})}\prod^{\kappa}_{j=1}\left(\frac{Ax_{j}+By_{j}}{N}\right)^{p_{j}}\right),
\]
where $c>0$ is a global constant since it depends only on the graph
information from $x_{j},y_{j}$. Divide the summation into $(p_{1},\dots,p_{\kappa_{\star}})$
and $(p_{\kappa_{\star}+1},\dots,p_{\kappa})$ to bound the right-hand
side with
\begin{equation}
cN\sum^{N}_{L=0}\left(\sum_{\substack{p_{1},\dots,p_{\kappa_{\star}}\ge0:\\
p_{1}+\cdots+p_{\kappa_{\star}}=N-L
}
}\frac{1}{\mathfrak{a}(p_{1})\mathfrak{a}(p_{2})\cdots\mathfrak{a}(p_{\kappa_{\star}})}\right)\left(\sum_{\substack{p_{\kappa_{\star}+1},\dots,p_{\kappa}\ge0:\\
p_{\kappa_{\star}+1}+\cdots+p_{\kappa}=L
}
}\prod^{\kappa}_{j=\kappa_{\star}+1}\left(\frac{Ax_{j}+By_{j}}{N}\right)^{p_{j}}\right),\label{eq:s4-3}
\end{equation}
where $L$ stands for the number of particles at $S\setminus S_{\star}$.
Above, in the first parenthesis we used $\frac{Ax_{j}+By_{j}}{N}=1$
which holds since $x_{j}=y_{j}=1$, and in the second parenthesis
we used $\mathfrak{a}(p_{\kappa_{\star}+1}),\dots,\mathfrak{a}(p_{\kappa})\ge1$.
For $j\in\llbracket\kappa_{\star}+1,\kappa\rrbracket$, we have $x_{j}\wedge y_{j}<1$
and $\frac{A}{N}\simeq\rho\in(0,1)$, thus
\[
\lim_{N\to\infty}\frac{Ax_{j}+By_{j}}{N}=\rho x_{j}+(1-\rho)y_{j}<1.
\]
This implies that there exists a constant $\upsilon\in(0,1)$ such
that $\frac{Ax_{j}+By_{j}}{N}\le\upsilon$ for all $j\in\llbracket\kappa_{\star}+1,\kappa\rrbracket$,
for all sufficiently big $N$. Therefore, the term at \eqref{eq:s4-3}
is bounded by
\begin{equation}
cN\sum^{N}_{L=0}(L+1)^{\kappa}\upsilon^{L}\left(\sum_{\substack{p_{1},\dots,p_{\kappa_{\star}}\ge0:\\
p_{1}+\cdots+p_{\kappa_{\star}}=N-L
}
}\frac{1}{\mathfrak{a}(p_{1})\mathfrak{a}(p_{2})\cdots\mathfrak{a}(p_{\kappa_{\star}})}\right),\label{eq:s4-4}
\end{equation}
where we bounded the number of all $(p_{\kappa_{\star}+1},\dots,p_{\kappa})$
simply by $c(L+1)^{\kappa}$. By Lemma \ref{lem:asymp2}, this is
bounded above by
\[
c\sum^{N}_{L=0}(L+1)^{\kappa}\upsilon^{L}(\log N)^{\kappa_{\star}-1}.
\]
It is straightforward that this has an asymptotic scale of $(\log N)^{\kappa_{\star}-1}$
as $N\to\infty$. This completes the upper bound part.
\end{proof}

The following is an analogue of Lemma \ref{lem:s2}.
\begin{lem}
\label{lem:s7}We have
\[
\frac{N}{{N \choose A}}\sum_{\bm{\eta}\in\Xi^{\ell_{N}}_{N}}\frac{\bm{m}^{\bm{\eta}}}{\bm{g!}(\bm{\eta})}\leq c(\log N)^{\kappa_{\star}-2}\log\frac{N}{\ell_{N}}.
\]
\end{lem}

\begin{proof}
Again, write $S=\llbracket1,\kappa\rrbracket$ and $S_{\star}=\llbracket1,\kappa_{\star}\rrbracket$.
From the proof of Lemma \ref{lem:s6}, especially from \eqref{eq:s4-3}
and \eqref{eq:s4-4}, we infer that
\[
\frac{N}{{N \choose A}}\sum_{\bm{\eta}\in\Xi^{\ell_{N}}_{N}}\frac{\bm{m}^{\bm{\eta}}}{\bm{g!}(\bm{\eta})}\le cN\sum^{N}_{L=0}(L+1)^{\kappa}\upsilon^{L}\left(\sum_{\substack{0\le p_{1},\dots,p_{\kappa_{\star}}<N-\ell_{N}:\\
p_{1}+\cdots+p_{\kappa_{\star}}=N-L
}
}\frac{1}{\mathfrak{a}(p_{1})\mathfrak{a}(p_{2})\cdots\mathfrak{a}(p_{\kappa_{\star}})}\right).
\]
By Lemma \ref{lem:asymp5}, the right-hand side is bounded above by
\[
c\sum^{N}_{L=0}(L+1)^{\kappa}\upsilon^{L}(\log N)^{\kappa_{\star}-2}(\log N-\log(\ell_{N}+1))\asymp(\log N)^{\kappa_{\star}-2}(\log N-\log\ell_{N}),
\]
as desired.
\end{proof}

The next lemma is an analogue of Lemma \ref{lem:s5}.
\begin{lem}
\label{lem:s8}For each $x\in S\setminus S_{\star}$,
\[
\frac{N}{\binom{N}{A}}\sum_{\bm{\eta}\in\Omega^{x,\ell_{N}}_{N}}\frac{\bm{m}^{\bm{\eta}}}{\bm{g!}(\bm{\eta})}\leq c(\log N)^{\kappa-1}\upsilon^{N-\ell_{N}},
\]
where $\upsilon\in(0,1)$ does not depend on $N$.
\end{lem}

\begin{proof}
For simplicity write $x=1$ and let $S=\llbracket1,\kappa\rrbracket$.
Via Lemmas \ref{lem:comb-bound} and \ref{lem:asymp2}, the left-hand
side becomes
\begin{align*}
 & \frac{N}{\binom{N}{A}}\sum_{\substack{p_{1},\dots,p_{\kappa}\ge0:\\
p_{1}+\cdots+p_{\kappa}=N,\\
p_{1}\ge N-\ell_{N}
}
}\sum_{\substack{q_{1},\dots,q_{\kappa}\ge0:\\
q_{1}+\cdots+q_{\kappa}=A
}
}\prod^{\kappa}_{j=1}\frac{{p_{j} \choose q_{j}}m_{1}(j)^{q_{j}}m_{2}(j)^{p_{j}-q_{j}}}{\mathfrak{a}(p_{j})}\\
 & \le cN\sum_{\substack{p_{1},\dots,p_{\kappa}\ge0:\\
p_{1}+\cdots+p_{\kappa}=N,\\
p_{1}\ge N-\ell_{N}
}
}\left(\prod^{\kappa}_{j=1}\frac{1}{\mathfrak{a}(p_{j})}\right)\left(\frac{Am_{1}(x)+Bm_{2}(x)}{N}\right)^{p_{1}}\le c'(\log N)^{\kappa-1}\left(\frac{Am_{1}(x)+Bm_{2}(x)}{N}\right)^{N-\ell_{N}}.
\end{align*}
As claimed in the proof of Lemma \ref{lem:s5}, there exists $\upsilon\in(0,1)$
such that $\max_{y\in S\setminus S_{\star}}\frac{Am_{1}(y)+Bm_{2}(y)}{N}\le\upsilon$
for all big $N$. This finishes the proof.
\end{proof}

\begin{proof}[Proof of Theorem \ref{thm:cond}-(2)]
 As done in part (1), it suffices to prove that
\begin{equation}
\lim_{N\to\infty}\nu_{N}(\Delta_{N})=0.\label{eq:cond-2-wts}
\end{equation}
Lemmas \ref{lem:s6} and \ref{lem:s7} imply
\[
\nu_{N}(\Xi^{\ell_{N}}_{N})=\frac{N}{Z_{N}{N \choose A}}\sum_{\bm{\eta}\in\Xi^{\ell_{N}}_{N}}\frac{\bm{m}^{\bm{\eta}}}{\bm{g!}(\bm{\eta})}\le\frac{c(\log N)^{\kappa_{\star}-2}\log\frac{N}{\ell_{N}}}{(\log N)^{\kappa_{\star}-1}}\ll1,
\]
since \eqref{eq:ellN-cri} implies $\log\frac{N}{\ell_{N}}\ll\log N$.
Moreover, Lemmas \ref{lem:s6} and \ref{lem:s8} give, for $x\notin S_{\star}$,
\[
\nu_{N}(\Omega^{x,\ell_{N}}_{N})=\frac{1}{Z_{N}}\frac{N}{\binom{N}{A}}\sum_{\bm{\eta}\in\Omega^{x,\ell_{N}}_{N}}\frac{\bm{m}^{\bm{\eta}}}{\bm{g!}(\bm{\eta})}\le\frac{c(\log N)^{\kappa-1}\upsilon^{N-\ell_{N}}}{(\log N)^{\kappa_{\star}-1}}\ll1.
\]
The last two displays prove \eqref{eq:cond-2-wts} as desired, via
\eqref{eq:Omega-dec}.
\end{proof}

\subsection{\label{sec2.3}Subcritical Case: $\alpha<1$}

Finally, in this last subsection we assume that $\alpha<1$. Here,
we want to prove the exact opposite; i.e., we want to prove that the
sets $\mathcal{E}^{x}_{N}$, $x\in S$, are negligible in the limit.
\begin{lem}
\label{lem:s9}We have
\[
Z_{N}\ge cN^{(\kappa_{\star}-1)(1-\alpha)}.
\]
\end{lem}

\begin{proof}
Recall from \eqref{eq:kappa-star} that $\kappa_{\star}\ge2$. Consider
only the configurations in which all particles are on $S_{\star}$.
Then, via Vandermonde's identity,
\[
Z_{N}\ge N^{\alpha}\sum_{\substack{p_{1},\dots,p_{\kappa_{\star}}\ge0:\\
p_{1}+\cdots+p_{\kappa_{\star}}=N
}
}\frac{1}{\mathfrak{a}(p_{1})\mathfrak{a}(p_{2})\cdots\mathfrak{a}(p_{\kappa_{\star}})}\ge cN^{\alpha}N^{\kappa_{\star}(1-\alpha)-1}=cN^{(\kappa_{\star}-1)(1-\alpha)},
\]
where the second inequality follows from Lemma \ref{lem:asymp2}.
\end{proof}

\begin{lem}
\label{lem:s10}For each $x\in S$, we have
\[
\frac{N^{\alpha}}{\binom{N}{A}}\sum_{\bm{\eta}\in\Omega^{x,\ell_{N}}_{N}}\frac{\bm{m}^{\bm{\eta}}}{\bm{g!}(\bm{\eta})}\leq c\ell^{(\kappa_{\star}-1)(1-\alpha)}_{N}.
\]
\end{lem}

\begin{proof}
The facts $A-\ell_{N}\gg\ell_{N}$ and $B-\ell_{N}\gg\ell_{N}$ imply
that the left-hand side of the lemma is maximized if $x\in S_{\star}$.
Thus, we may assume so. Let $S=\llbracket1,\kappa\rrbracket$, $S_{\star}=\llbracket1,\kappa_{\star}\rrbracket$,
and $x=1$. As done in the proof of Lemma \ref{lem:s8}, the left-hand
side is bounded by
\begin{align*}
 & \frac{N^{\alpha}}{{N \choose A}}\sum_{\substack{p_{1},\dots,p_{\kappa}\ge0:\\
p_{1}+\cdots+p_{\kappa}=N,\\
p_{1}\ge N-\ell_{N}
}
}\sum_{\substack{q_{1},\dots,q_{\kappa}\ge0:\\
q_{1}+\cdots+q_{\kappa}=A
}
}\prod^{\kappa}_{j=1}\frac{{p_{j} \choose q_{j}}m_{1}(j)^{q_{j}}m_{2}(j)^{p_{j}-q_{j}}}{\mathfrak{a}(p_{j})}\\
 & \le cN^{\alpha}\sum_{\substack{p_{1},\dots,p_{\kappa}\ge0:\\
p_{1}+\cdots+p_{\kappa}=N,\\
p_{1}\ge N-\ell_{N}
}
}\left(\prod^{\kappa}_{j=1}\frac{1}{\mathfrak{a}(p_{j})}\right)\prod^{\kappa}_{j=\kappa_{\star}+1}\left(\frac{Am_{1}(j)+Bm_{2}(j)}{N}\right)^{p_{j}}.
\end{align*}
As claimed in the proof of Lemma \ref{lem:s5}, there exists $\upsilon\in(0,1)$
such that $\max_{y\in S\setminus S_{\star}}\frac{Am_{1}(y)+Bm_{2}(y)}{N}\le\upsilon$
for sufficiently large $N$. Thus, denoting by $L$ the number of
particles on $S\setminus S_{\star}$, the right-hand side equals
\[
cN^{\alpha}\sum^{\ell_{N}}_{L=0}\sum^{N-L}_{p_{1}=N-\ell_{N}}\frac{1}{p^{\alpha}_{1}}\sum_{\substack{p_{2},\dots,p_{\kappa_{\star}}\ge0:\\
p_{2}+\cdots+p_{\kappa_{\star}}=N-L-p_{1}
}
}\sum_{\substack{p_{\kappa_{\star}+1},\dots,p_{\kappa}\ge0:\\
p_{\kappa_{\star}+1}+\cdots+p_{\kappa}=L
}
}\left(\prod^{\kappa}_{j=2}\frac{1}{\mathfrak{a}(p_{j})}\right)\upsilon^{L}.
\]
By Lemma \ref{lem:asymp2} and the fact that $p^{\alpha}_{1}\ge(N-\ell_{N})^{\alpha}\asymp N^{\alpha}$,
this is bounded by
\[
c\sum^{\ell_{N}}_{L=0}(L+1)^{(\kappa-\kappa_{\star})(1-\alpha)-1}\upsilon^{L}\sum^{N-L}_{p_{1}=N-\ell_{N}}(N-L-p_{1}+1)^{(\kappa_{\star}-1)(1-\alpha)-1}.
\]
Summing up the summation in $p_{1}$, this is bounded above by
\begin{align*}
c\sum^{\ell_{N}}_{L=0}(\ell_{N}-L+1)^{(\kappa_{\star}-1)(1-\alpha)}(L+1)^{(\kappa-\kappa_{\star})(1-\alpha)-1}\upsilon^{L} & \le c(\ell_{N}+1)^{(\kappa_{\star}-1)(1-\alpha)}\sum^{\ell_{N}}_{L=0}(L+1)^{(\kappa-\kappa_{\star})(1-\alpha)-1}\upsilon^{L}\\
 & \le c'(\ell_{N}+1)^{(\kappa_{\star}-1)(1-\alpha)}.
\end{align*}
This finishes the proof.
\end{proof}

\begin{proof}[Proof of Theorem \ref{thm:cond}-(3)]
 Our objective is to prove that
\[
\nu_{N}(\mathcal{E}^{x}_{N})\ll1\qquad\text{for all}\quad x\in S.
\]
This follows from Lemmas \ref{lem:s9} and \ref{lem:s10}:
\[
\nu_{N}(\mathcal{E}^{x}_{N})=\frac{N^{\alpha}}{Z_{N}{N \choose A}}\sum_{\bm{\eta}\in\Omega^{x,\ell_{N}}_{N}}\frac{\bm{m}^{\bm{\eta}}}{\bm{g!}(\bm{\eta})}\le\frac{c\ell^{(\kappa_{\star}-1)(1-\alpha)}_{N}}{N^{(\kappa_{\star}-1)(1-\alpha)}}\ll1,
\]
which is valid since $\ell_{N}\ll N$, $\kappa_{\star}-1\ge1$, and
$1-\alpha>0$.
\end{proof}

\section{\label{sec3}Local Analysis in Metastable Valleys}

In this section, we prove two local mixing results in Propositions
\ref{prop1} and \ref{prop3}. Then, we present a proof of Theorem
\ref{thm:main}-(1).

We start with the proof of the sector condition, Lemma \ref{lem:sector}.
The following notation is useful throughout. Define $(N-1)$-particle
configuration spaces $\Omega^{i,-}_{N-1}$ for each $i\in\{1,2\}$
as 
\begin{equation}
\begin{aligned}\Omega^{1,-}_{N-1} & :=\left\{ \bm{\xi}\in\mathbb{N}^{S}_{0}:\sum_{x\in S}\xi_{1}(x)=A-1,\enspace\sum_{x\in S}\xi_{2}(x)=B\right\} ,\\
\Omega^{2,-}_{N-1} & :=\left\{ \bm{\xi}\in\mathbb{N}^{S}_{0}:\sum_{x\in S}\xi_{1}(x)=A,\enspace\sum_{x\in S}\xi_{2}(x)=B-1\right\} .
\end{aligned}
\label{eq:Omega-minus}
\end{equation}
In words, $\Omega^{i,-}_{N-1}$ is obtained from $\Omega_{N}$ by
removing one type-$i$ particle.
\begin{proof}[Proof of Lemma \ref{lem:sector}]
 Notice that 
\[
\langle G,-\mathcal{L}_{N}F\rangle_{\nu_{N}}=\sum_{\bm{\eta}\in\Omega_{N}}\sum_{i\in\{1,2\}}\sum_{u,w\in S}\nu_{N}(\bm{\eta})\bm{g}_{i}(\bm{\eta}(u))r_{i}(u,w)\left(F(\bm{\eta})-F(\bm{\eta}-\delta^{i}_{u}+\delta^{i}_{w})\right)G(\bm{\eta}).
\]
By the change of variables $\bm{\eta}-\delta^{i}_{u}=\bm{\xi}\in\Omega^{i,-}_{N-1}$,
one can notice that the above expression is equal to
\begin{equation}
\begin{aligned} & \sum_{i\in\{1,2\}}\sum_{u,w\in S}\sum_{\bm{\xi}\in\Omega^{i,-}_{N-1}}\nu_{N}(\bm{\xi}+\delta^{i}_{u})\bm{g}_{i}((\bm{\xi}+\delta^{i}_{u})(u))r_{i}(u,w)\left(F(\bm{\xi}+\delta^{i}_{u})-F(\bm{\xi}+\delta^{i}_{w})\right)G(\bm{\xi}+\delta^{i}_{u})\\
 & =\sum_{i\in\{1,2\}}\sum_{\bm{\xi}\in\Omega^{i,-}_{N-1}}\frac{N^{\alpha}}{Z_{N}\binom{N}{A}}\frac{\bm{m}^{\bm{\xi}}}{\bm{g!}(\bm{\xi})}\sum_{u,w\in S}m_{i}(u)r_{i}(u,w)G(\bm{\xi}+\delta^{i}_{u})\left(F(\bm{\xi}+\delta^{i}_{u})-F(\bm{\xi}+\delta^{i}_{w})\right).
\end{aligned}
\label{eq:sector-1}
\end{equation}
From the stationarity of $m_{i}$, we have $\sum_{u,w\in S}m_{i}(u)r_{i}(u,w)(F(\bm{\xi}+\delta^{i}_{u})-F(\bm{\xi}+\delta^{i}_{w}))=0$.
Thus by letting 
\begin{equation}
\overline{G}^{i}(\bm{\xi}):=\frac{1}{\kappa}\sum_{z\in S}G(\bm{\xi}+\delta^{i}_{z}),\label{eq:Gibar-def}
\end{equation}
we have
\begin{align*}
 & \left|\sum_{u,w\in S}m_{i}(u)r_{i}(u,w)G(\bm{\xi}+\delta^{i}_{u})\left(F(\bm{\xi}+\delta^{i}_{u})-F(\bm{\xi}+\delta^{i}_{w})\right)\right|\\
 & =\left|\sum_{u,w\in S}m_{i}(u)r_{i}(u,w)\left(G(\bm{\xi}+\delta^{i}_{u})-\overline{G}^{i}(\bm{\xi})\right)\left(F(\bm{\xi}+\delta^{i}_{u})-F(\bm{\xi}+\delta^{i}_{w})\right)\right|\\
 & \leq\frac{\theta}{2}\sum_{u,w\in S}m_{i}(u)r_{i}(u,w)\left(F(\bm{\xi}+\delta^{i}_{u})-F(\bm{\xi}+\delta^{i}_{w})\right)^{2}+\frac{1}{2\theta}\sum_{u,w\in S}m_{i}(u)r_{i}(u,w)\left(G(\bm{\xi}+\delta^{i}_{u})-\overline{G}^{i}(\bm{\xi})\right)^{2},
\end{align*}
for any $\theta>0$. Substituting this to \eqref{eq:sector-1}, we
obtain
\begin{equation}
\left|\langle G,-\mathcal{L}_{N}F\rangle_{\nu_{N}}\right|\le\theta\mathcal{D}_{N}(F)+\frac{1}{2\theta}\sum_{i\in\{1,2\}}\sum_{\bm{\xi}\in\Omega^{i,-}_{N-1}}\frac{N^{\alpha}}{Z_{N}\binom{N}{A}}\frac{\bm{m}^{\bm{\xi}}}{\bm{g!}(\bm{\xi})}\sum_{u,w\in S}m_{i}(u)r_{i}(u,w)\left(G(\bm{\xi}+\delta^{i}_{u})-\overline{G}^{i}(\bm{\xi})\right)^{2}.\label{eq:sector-2}
\end{equation}
Note that by the definition of $\overline{G}^{i}(\bm{\xi})$,
\begin{equation}
\left(G(\bm{\xi}+\delta^{i}_{u})-\overline{G}^{i}(\bm{\xi})\right)^{2}=\left(G(\bm{\xi}+\delta^{i}_{u})-\frac{1}{\kappa}\sum_{z\in S}G(\bm{\xi}+\delta^{i}_{z})\right)^{2}\le\frac{1}{\kappa}\sum_{z\in S}\left(G(\bm{\xi}+\delta^{i}_{u})-G(\bm{\xi}+\delta^{i}_{z})\right)^{2}.\label{eq:sector-3}
\end{equation}
Let $E_{i}:=\{(x,y)\in S\times S:r_{i}(x,y)>0\}$ and note that
\begin{equation}
c\le\min_{i\in\{1,2\}}\min_{(x,y)\in E_{i}}m_{i}(x)r_{i}(x,y)\le\max_{i\in\{1,2\}}\max_{(x,y)\in E_{i}}m_{i}(x)r_{i}(x,y)\le c',\label{eq:sector-4}
\end{equation}
By the irreducibility of the underlying random walk, for each pair
$u,z\in S$ there exists a path
\[
u=u^{i}_{0}\to u^{i}_{1}\to\cdots\to u^{i}_{k}=z
\]
such that $k<\kappa$ and $(u^{i}_{m},u^{i}_{m+1})\in E_{i}$ for
each $m\in\llbracket0,k-1\rrbracket$. Then, the summation at the
right-hand line of \eqref{eq:sector-3} can be bounded above, via
the Cauchy--Schwarz inequality, by 
\[
\frac{1}{\kappa}\sum_{z\in S}k\sum^{k-1}_{m=0}\left[G(\bm{\xi}+\delta^{i}_{u^{i}_{m+1}})-G(\bm{\xi}+\delta^{i}_{u^{i}_{m}})\right]^{2}\leq\sum_{z\in S}\sum_{(x,y)\in E_{i}}\left(G(\bm{\xi}+\delta^{i}_{x})-G(\bm{\xi}+\delta^{i}_{y})\right)^{2}.
\]
Thus, via \eqref{eq:sector-4}, the summation in $u,w\in S$ at \eqref{eq:sector-2}
is bounded above by
\begin{equation}
\begin{aligned}\sum_{u,w\in S}m_{i}(u)r_{i}(u,w)\sum_{z\in S}\sum_{(x,y)\in E_{i}} & \left(G(\bm{\xi}+\delta^{i}_{x})-G(\bm{\xi}+\delta^{i}_{y})\right)^{2}\\
 & \le c\sum_{(x,y)\in E_{i}}m_{i}(x)r_{i}(x,y)\left(G(\bm{\xi}+\delta^{i}_{x})-G(\bm{\xi}+\delta^{i}_{y})\right)^{2}.
\end{aligned}
\label{eq:sector-5}
\end{equation}
Substituting this to \eqref{eq:sector-2} and returning to $\Omega_{N}$
with $\bm{\xi}+\delta^{i}_{x}=\bm{\eta}$, we have
\begin{align*}
\left|\langle G,-\mathcal{L}_{N}F\rangle_{\nu_{N}}\right| & \le\theta\mathcal{D}_{N}(F)+\frac{c}{\theta}\sum_{i\in\{1,2\}}\sum_{\bm{\eta}\in\Omega_{N}}\sum_{(x,y)\in E_{i}}\nu_{N}(\bm{\eta})\bm{g}_{i}(\bm{\eta}(x))r_{i}(x,y)\left(G(\bm{\eta})-G(\bm{\eta}-\delta^{i}_{x}+\delta^{i}_{y})\right)^{2}\\
 & =\theta\mathcal{D}_{N}(F)+\frac{2c}{\theta}\mathcal{D}_{N}(G).
\end{align*}
The proof is completed by taking $\theta=\sqrt{\mathcal{D}_{N}(G)/\mathcal{D}_{N}(F)}$.
\end{proof}

Next, we state and prove two capacity estimates which immediately
verify Proposition \ref{prop3}.
\begin{lem}
\label{lem:cap1}Recall the definition of $\tilde{\mathcal{E}}^{x}_{N}$
from \eqref{eq:EN-tilde}. For each $x\in S_{\star}$,
\[
{\rm CAP}_{N}(\mathcal{E}^{x}_{N},\Omega_{N}\setminus\tilde{\mathcal{E}}^{x}_{N})\le cN^{-b\alpha}.
\]
\end{lem}

\begin{proof}
By \eqref{eq:CAPCAP}, it suffices to prove the inequality with ${\rm CAP}^{s}_{N}$
in place of ${\rm CAP}_{N}$. Using a test object $F={\bf 1}_{\tilde{\mathcal{E}}^{x}_{N}}$,
which clearly satisfies $F|_{\mathcal{E}^{x}_{N}}\equiv1$ and $F|_{\Omega_{N}\setminus\tilde{\mathcal{E}}^{x}_{N}}\equiv0$
since $\mathcal{E}^{x}_{N}\subset\tilde{\mathcal{E}}^{x}_{N}$, the
Dirichlet principle \eqref{eq:DP} implies that
\begin{align*}
 & {\rm CAP}^{s}_{N}(\mathcal{E}^{x}_{N},\Omega_{N}\setminus\tilde{\mathcal{E}}^{x}_{N})\le\mathcal{D}_{N}(F)=\langle F,-\mathcal{L}_{N}F\rangle_{\nu_{N}}\\
 & =\sum_{\bm{\eta}\in\tilde{\mathcal{E}}^{x}_{N}}\nu_{N}(\bm{\eta})\sum_{i\in\{1,2\}}\sum_{u,w\in S}\bm{g}_{i}(\bm{\eta}(u))r_{i}(u,w)\left[1-{\bf 1}_{\tilde{\mathcal{E}}^{x}_{N}}(\bm{\eta}-\delta^{i}_{u}+\delta^{i}_{w})\right].
\end{align*}
By \eqref{eq:EN-tilde}, the new configuration $\bm{\eta}-\delta^{i}_{u}+\delta^{i}_{w}$
exits $\tilde{\mathcal{E}}^{x}_{N}$ if and only if $|\bm{\eta}(x)|=N-\tilde{\ell}_{N}$,
$u=x$, and $w\ne x$. Thus, the right-hand side equals
\[
\sum_{\substack{\bm{\eta}\in\tilde{\mathcal{E}}^{x}_{N}:\\
|\bm{\eta}(x)|=N-\tilde{\ell}_{N}
}
}\nu_{N}(\bm{\eta})\sum_{i\in\{1,2\}}\sum_{w\in S\setminus\{x\}}\bm{g}_{i}(\bm{\eta}(x))r_{i}(x,w)\le c\nu_{N}\left(\bm{\eta}\in\tilde{\mathcal{E}}^{x}_{N}:|\bm{\eta}(x)|=N-\tilde{\ell}_{N}\right).
\]
The second inequality holds since $\bm{g}_{i}$ and $r_{i}$ are uniformly
bounded. Finally, we write the measure at the right-hand side as
\[
\frac{N^{\alpha}}{Z_{N}{N \choose A}}\sum_{\substack{p_{1},\dots,p_{\kappa}\ge0:\\
p_{1}+\cdots+p_{\kappa}=N,\\
p_{1}=N-\tilde{\ell}_{N}
}
}\sum_{\substack{q_{1},\dots,q_{\kappa}\ge0:\\
q_{1}+\cdots+q_{\kappa}=A
}
}\prod^{\kappa}_{j=1}\frac{{p_{j} \choose q_{j}}m_{1}(j)^{q_{j}}m_{2}(j)^{p_{j}-q_{j}}}{\mathfrak{a}(p_{j})}\le cN^{\alpha}\sum_{\substack{p_{1},\dots,p_{\kappa}\ge0:\\
p_{1}+\cdots+p_{\kappa}=N,\\
p_{1}=N-\tilde{\ell}_{N}
}
}\frac{1}{\mathfrak{a}(p_{1})\mathfrak{a}(p_{2})\cdots\mathfrak{a}(p_{\kappa})},
\]
via Lemmas \ref{lem:s1}, \ref{lem:comb-bound}, and Vandermonde's
identity. Then by Lemma \ref{lem:asymp2}, the right-hand side is
bounded by
\[
c\sum_{\substack{p_{2},\dots,p_{\kappa}\ge0:\\
p_{2}+\cdots+p_{\kappa}=\tilde{\ell}_{N}
}
}\frac{1}{\mathfrak{a}(p_{2})\cdots\mathfrak{a}(p_{\kappa})}\le\frac{c'}{\tilde{\ell}^{\alpha}_{N}}\asymp N^{-b\alpha}.
\]
Combining all the estimates completes the proof.
\end{proof}

Recall from \eqref{eq:ellN-exact} that we selected $\ell_{N}=\lfloor\gamma\log N\rfloor$.
The following lemma explains why we need to take $\gamma>0$ small,
which is clearly different from the original zero-range process considered
in \cite{BL12a,Lan14,Seo19}.
\begin{lem}
\label{lem:cap2}Recall $\bm{\xi}^{x}_{N}\in\mathcal{E}^{x}_{N}$
from \eqref{eq:xiNx}. For each $x\in S_{\star}$,
\[
\inf_{\bm{\eta}\in\mathcal{E}^{x}_{N}\setminus\{\bm{\xi}^{x}_{N}\}}{\rm CAP}_{N}(\bm{\eta},\bm{\xi}^{x}_{N})\ge\frac{cN^{-\gamma c_{0}}}{(\log N)^{1+\alpha(\kappa-1)}},
\]
where $c_{0}>0$ is a global constant defined at \eqref{eq:c0-def}.
\end{lem}

\begin{proof}
Fix $x\in S_{\star}$. For any $\bm{\eta}\in\mathcal{E}^{x}_{N}$,
\[
\nu_{N}(\bm{\eta})=\frac{N^{\alpha}}{Z_{N}{N \choose A}}\frac{\bm{m}^{\bm{\eta}}}{\bm{g!}(\bm{\eta})}=\frac{N^{\alpha}}{Z_{N}{N \choose A}}\frac{1}{\bm{g!}(\eta_{1}(x),\eta_{2}(x))}\prod_{y\ne x}\frac{m_{1}(y)^{\eta_{1}(y)}m_{2}(y)^{\eta_{2}(y)}}{\bm{g!}(\eta_{1}(y),\eta_{2}(y))}.
\]
First,
\begin{align*}
\frac{N^{\alpha}}{{N \choose A}}\frac{1}{\bm{g!}(\eta_{1}(x),\eta_{2}(x))}=\frac{N^{\alpha}}{{N \choose A}}\frac{{\eta_{1}(x)+\eta_{2}(x) \choose \eta_{1}(x)}}{(\eta_{1}(x)+\eta_{2}(x))^{\alpha}} & \ge\frac{\left[A(A-1)\cdots(\eta_{1}(x)+1)\right]\left[B(B-1)\cdots(\eta_{2}(x)+1)\right]}{N(N-1)\cdots(\eta_{1}(x)+\eta_{2}(x)+1)}\\
 & \ge\left(\frac{A-\ell_{N}}{N}\right)^{\ell_{N}}\left(\frac{B-\ell_{N}}{N}\right)^{\ell_{N}}\ge\left[c\rho(1-\rho)\right]^{\ell_{N}},
\end{align*}
for sufficiently big $N$. At the first inequality we used the fact
that $\eta_{1}(x)+\eta_{2}(x)\le N$. In addition, we have $\sum_{y\in S\setminus\{x\}}|\bm{\eta}(y)|\le\ell_{N}$,
thus
\[
\prod_{y\ne x}\frac{m_{1}(y)^{\eta_{1}(y)}m_{2}(y)^{\eta_{2}(y)}}{\bm{g!}(\eta_{1}(y),\eta_{2}(y))}\ge m^{\ell_{N}}_{0}\prod_{y\ne x}\frac{{\eta_{1}(y)+\eta_{2}(y) \choose \eta_{1}(y)}}{\mathfrak{a}(\eta_{1}(y)+\eta_{2}(y))}\ge\frac{m^{\ell_{N}}_{0}}{\ell^{\alpha(\kappa-1)}_{N}},
\]
where $m_{0}:=\min_{y\in S}(m_{1}(y)\wedge m_{2}(y))>0$. Therefore,
gathering the last two displays and Proposition \ref{prop:ZN-limit},
we obtain that
\begin{equation}
\nu_{N}(\bm{\eta})\ge\frac{c[c\rho(1-\rho)]^{\ell_{N}}m^{\ell_{N}}_{0}}{\ell^{\alpha(\kappa-1)}_{N}}\asymp\frac{[c\rho(1-\rho)]^{\gamma\log N}m^{\gamma\log N}_{0}}{(\log N)^{\alpha(\kappa-1)}}\ge\frac{N^{-\gamma c_{0}}}{(\log N)^{\alpha(\kappa-1)}},\label{eq:nuN-LB}
\end{equation}
where 
\begin{equation}
c_{0}:=|\log(c(\rho(1-\rho)))|-\log m_{0}>0\label{eq:c0-def}
\end{equation}
 is a constant independent of $N$.

Now, take any $\bm{\eta}\in\mathcal{E}^{x}_{N}$ such that $\bm{\eta}\ne\bm{\xi}^{x}_{N}$.
By the definition of $\mathcal{E}^{x}_{N}$ and the irreducibility,
there exists a sequence
\[
\bm{\eta}=\bm{\eta}_{0},\bm{\eta}_{1},\dots,\bm{\eta}_{k}=\bm{\xi}^{x}_{N}
\]
such that $\bm{\eta}_{m}\in\mathcal{E}^{x}_{N}$, $R^{s}_{N}(\bm{\eta}_{m},\bm{\eta}_{m+1})>0$
for all $m$, and $k<\kappa\ell_{N}$, and that in each jump $\bm{\eta}_{m}\to\bm{\eta}_{m+1}$,
only the particle type with a greater or equal count can jump. In
this way, according to \eqref{eq:g-def}, $\{R^{s}_{N}(\bm{\eta}_{m},\bm{\eta}_{m+1})\}_{m\in\llbracket0,k-1\rrbracket}$
is bounded uniformly away from zero. Thus, by \eqref{eq:CAPCAP} and
the Thomson principle \eqref{eq:TP},
\[
{\rm CAP}_{N}(\bm{\eta},\bm{\xi}^{x}_{N})\ge{\rm CAP}^{s}_{N}(\bm{\eta},\bm{\xi}^{x}_{N})\ge\left(\sum^{k-1}_{m=0}\frac{1}{\nu_{N}(\bm{\eta}_{m})R^{s}_{N}(\bm{\eta}_{m},\bm{\eta}_{m+1})}\right)^{-1}\ge c\left(\sum^{k-1}_{m=0}\frac{1}{\nu_{N}(\bm{\eta}_{m})}\right)^{-1}.
\]
By \eqref{eq:nuN-LB} and the fact that $k<\kappa\ell_{N}$, we conclude
that
\[
{\rm CAP}_{N}(\bm{\eta},\bm{\xi}^{x}_{N})\ge c\left(kN^{\gamma c_{0}}(\log N)^{\alpha(\kappa-1)}\right)^{-1}\ge c'\frac{N^{-\gamma c_{0}}}{(\log N)^{1+\alpha(\kappa-1)}},
\]
which is a bound uniform for all $\bm{\eta}\in\mathcal{E}^{x}_{N}\setminus\{\bm{\xi}^{x}_{N}\}$,
as desired.
\end{proof}

Now, we are ready to prove the two propositions. Let us take
\begin{equation}
\gamma c_{0}<b\alpha,\label{eq:gamma-def}
\end{equation}
where $b\in(0,1)$ is from \eqref{eq:EN-tilde}.
\begin{proof}[Proof of Proposition \ref{prop1}]
 To prove the proposition, it suffices to prove that
\[
\lim_{N\to\infty}\max_{\bm{\eta}\in\mathcal{E}^{x}_{N}}|F_{N}(\bm{\eta})-F_{N}(\bm{\xi}^{x}_{N})|=0\qquad\text{for each}\quad x\in S_{\star}.
\]
Let us write
\begin{equation}
G_{N}:=\sum_{x\in S_{\star}}\mathfrak{g}(x){\bf 1}_{\mathcal{E}^{x}_{N}}.\label{eq:GN-def}
\end{equation}
It is well known that the solution $F_{N}$ to \eqref{eq:res-mic}
is represented as
\[
F_{N}(\bm{\eta})=\mathbb{E}^{N}_{\bm{\eta}}\left[\int^{\infty}_{0}e^{-\lambda t}G_{N}(\bm{\eta}_{N}(\theta_{N}t))\,{\rm d}t\right]=\theta^{-1}_{N}\mathbb{E}^{N}_{\bm{\eta}}\left[\int^{\infty}_{0}e^{-\lambda\theta^{-1}_{N}t}G_{N}(\bm{\eta}_{N}(t))\,{\rm d}t\right].
\]
In particular, 
\begin{equation}
\|F_{N}\|_{\infty}\le\frac{\|\mathfrak{g}\|_{\infty}}{\lambda}.\label{eq:FN-bdd}
\end{equation}
Dividing the integration at time $t=H_{\bm{\xi}^{x}_{N}}$ and applying
the strong Markov property, we obtain
\[
F_{N}(\bm{\eta})=\theta^{-1}_{N}\mathbb{E}^{N}_{\bm{\eta}}\left[\int^{H_{\bm{\xi}^{x}_{N}}}_{0}e^{-\lambda\theta^{-1}_{N}t}G_{N}(\bm{\eta}_{N}(t))\,{\rm d}t\right]+\mathbb{E}^{N}_{\bm{\eta}}\left[e^{-\lambda\theta^{-1}_{N}H_{\bm{\xi}^{x}_{N}}}\right]F_{N}(\bm{\xi}^{x}_{N}).
\]
Thus, we have
\begin{align*}
|F_{N}(\bm{\eta})-F_{N}(\bm{\xi}^{x}_{N})| & \le\mathbb{E}^{N}_{\bm{\eta}}\left[1-e^{-\lambda\theta^{-1}_{N}H_{\bm{\xi}^{x}_{N}}}\right]|F_{N}(\bm{\xi}^{x}_{N})|+\theta^{-1}_{N}\|\mathfrak{g}\|_{\infty}\mathbb{E}^{N}_{\bm{\eta}}\left[\int^{H_{\bm{\xi}^{x}_{N}}}_{0}e^{-\lambda\theta^{-1}_{N}t}\,{\rm d}t\right]\\
 & \le\frac{2\|\mathfrak{g}\|_{\infty}}{\lambda}\mathbb{E}^{N}_{\bm{\eta}}\left[1-e^{-\lambda\theta^{-1}_{N}H_{\bm{\xi}^{x}_{N}}}\right]\le2\theta^{-1}_{N}\|\mathfrak{g}\|_{\infty}\mathbb{E}^{N}_{\bm{\eta}}\left[H_{\bm{\xi}^{x}_{N}}\right].
\end{align*}
Therefore, it suffices to prove that
\[
\max_{\bm{\eta}\in\mathcal{E}^{x}_{N}}\mathbb{E}^{N}_{\bm{\eta}}\left[H_{\bm{\xi}^{x}_{N}}\right]\ll\theta_{N}=N^{\alpha+1}.
\]
A well-known formula (cf. \cite[Proposition A.2]{BL12b}) gives that
\[
\mathbb{E}^{N}_{\bm{\eta}}\left[H_{\bm{\xi}^{x}_{N}}\right]=\frac{\sum_{\bm{\eta'}\in\Omega_{N}}\nu_{N}(\bm{\eta'})\mathfrak{H}^{*}_{\bm{\eta},\bm{\xi}^{x}_{N}}(\bm{\eta'})}{{\rm CAP}_{N}(\bm{\eta},\bm{\xi}^{x}_{N})}\le\frac{1}{{\rm CAP}_{N}(\bm{\eta},\bm{\xi}^{x}_{N})},
\]
where the inequality is simply from $\mathfrak{H}^{*}_{\bm{\eta},\bm{\xi}^{x}_{N}}\le1$.
Since $\gamma c_{0}<\alpha+1$ by \eqref{eq:gamma-def}, Lemma \ref{lem:cap2}
concludes the proof.
\end{proof}

\begin{proof}[Proof of Proposition \ref{prop3}]
 By Lemmas \ref{lem:cap1} and \ref{lem:cap2},
\[
\frac{{\rm CAP}_{N}(\mathcal{E}^{x}_{N},\Omega_{N}\setminus\tilde{\mathcal{E}}^{x}_{N})}{\inf_{\bm{\eta}\in\mathcal{E}^{x}_{N}\setminus\{\bm{\xi}^{x}_{N}\}}{\rm CAP}_{N}(\bm{\eta},\bm{\xi}^{x}_{N})}\le cN^{\gamma c_{0}-b\alpha}(\log N)^{1+\alpha(\kappa-1)}\ll1,
\]
since $\gamma$ satisfies \eqref{eq:gamma-def}.
\end{proof}

Finally, we are ready to prove Theorem \ref{thm:main}-(1).
\begin{proof}[Proof of Theorem \ref{thm:main}-(1)]
 A simple renewal estimate (cf. \cite[Lemma 8.4]{BdH15}) implies
that for each $\bm{\eta}\in\mathcal{E}^{x}_{N}$,
\[
\mathbb{P}^{N}_{\bm{\eta}}\left[H_{\bm{\xi}^{x}_{N}}>H_{\Omega_{N}\setminus\tilde{\mathcal{E}}^{x}_{N}}\right]\le\frac{{\rm CAP}_{N}(\bm{\eta},\Omega_{N}\setminus\tilde{\mathcal{E}}^{x}_{N})}{{\rm CAP}_{N}(\bm{\eta},\bm{\xi}^{x}_{N})}\le\frac{{\rm CAP}_{N}(\mathcal{E}^{x}_{N},\Omega_{N}\setminus\tilde{\mathcal{E}}^{x}_{N})}{{\rm CAP}_{N}(\bm{\eta},\bm{\xi}^{x}_{N})}.
\]
The second inequality follows from the fact that capacity is monotone
(cf. \cite[Lemma 2.2]{GL14}). Therefore, Proposition \ref{prop3}
concludes the proof.
\end{proof}

\section{\label{sec4}Test Function Construction}

In this section, we fix $x\in S_{\star}$ and construct a test object
$\mathbb{W}_{x}:\Omega_{N}\to\mathbb{R}$ which approximates the equilibrium
potential $\mathfrak{H}^{*}_{x}=\mathfrak{H}^{*}_{\mathcal{E}^{x}_{N},\mathcal{E}_{N}\setminus\mathcal{E}^{x}_{N}}$
between $\mathcal{E}^{x}_{N}$ and $\mathcal{E}_{N}\setminus\mathcal{E}^{x}_{N}$
of the adjoint process $\bm{\eta}^{*}_{N}$ (cf. \eqref{eq:adj-rate}),
as explained in Section \ref{sec1.4}.

Let us fix a sufficiently small positive constant $\epsilon>0$. We
always assume that $N\epsilon$ is an integer, since otherwise we
may consider $\lfloor N\epsilon\rfloor$ instead. See Figure \ref{fig4.1}
for illustrations of macroscopic and microscopic decompositions of
the space.

\subsection{\label{sec4.1}Macroscopic Tube Construction}

\begin{figure}
\begin{tikzpicture}[scale=0.7]
\fill[red!30!white] (0,0)--(3,0)--(1.5,{-1.5*sqrt(3)})--cycle;
\fill[blue!30!white] (3,0)--(2.5,{-0.5*sqrt(3)})--(9.5,{-0.5*sqrt(3)})--(10,0)--cycle;
\fill[blue!30!white] (2,{-sqrt(3)})--(1.5,{-1.5*sqrt(3)})--(5,{-5*sqrt(3)})--(5.5,{-4.5*sqrt(3)})--cycle;

\draw[] (0,0)--(10,0)--(5,{-5*sqrt(3)})--cycle;
\draw[densely dotted] (0.5,{-0.5*sqrt(3)})--(9.5,{-0.5*sqrt(3)}); \draw[densely dotted] (1,0)--(5.5,{-4.5*sqrt(3)});
\draw (3,0)--(1.5,{-1.5*sqrt(3)});
\draw[very thick] (3,0)--(2.5,{-0.5*sqrt(3)})--(9.5,{-0.5*sqrt(3)})--(10,0)--cycle;
\draw[very thick] (2,{-sqrt(3)})--(1.5,{-1.5*sqrt(3)})--(5,{-5*sqrt(3)})--(5.5,{-4.5*sqrt(3)})--cycle;

\draw (-0.2,0) node[left]{$x$}; \draw (10.2,0) node[right]{$y$}; \draw (5,{-5*sqrt(3)-0.2}) node[below]{$z$};
\draw (0.75,{-0.75*sqrt(3)}) node[below left]{\color{red} $\mathcal D_{3\epsilon}^x$};
\draw (6,{-0.5*sqrt(3)-0.2}) node[below]{\color{blue} $\mathcal K_{\epsilon}^{x,y}$};
\draw (3.25,{-3.25*sqrt(3)}) node[below left]{\color{blue} $\mathcal K_{\epsilon}^{x,z}$};
\end{tikzpicture}\hspace{10mm}\begin{tikzpicture}[scale=0.7]
\fill[red!30!white] (0,0)--(2,0)--(1,{-sqrt(3)})--cycle; \fill[red] (0,0)--(0.5,0)--(0.25,{0.25*-sqrt(3)})--cycle;
\fill[red!30!white] (10,0)--(8,0)--(9,{-sqrt(3)})--cycle; \fill[red] (10,0)--(9.5,0)--(9.75,{0.25*-sqrt(3)})--cycle;
\fill[red!30!white] (5,{-5*sqrt(3)})--(6,{-4*sqrt(3)})--(4,{-4*sqrt(3)})--cycle; \fill[red] (5,{-5*sqrt(3)})--(5.25,{-4.75*sqrt(3)})--(4.75,{-4.75*sqrt(3)})--cycle;
\fill[blue!30!white] (2.2,0)--(1.9,{-0.3*sqrt(3)})--(8.1,{-0.3*sqrt(3)})--(7.8,0)--cycle;
\fill[blue!30!white] (1.4,{-0.8*sqrt(3)})--(1.1,{-1.1*sqrt(3)})--(3.9,{-3.9*sqrt(3)})--(4.5,{-3.9*sqrt(3)})--cycle;

\draw[] (0,0)--(10,0)--(5,{-5*sqrt(3)})--cycle;
\draw[dotted] (0.5,{-0.5*sqrt(3)})--(9.5,{-0.5*sqrt(3)}); \draw[dotted] (1,0)--(5.5,{-4.5*sqrt(3)});
\draw[dotted] (3,0)--(1.5,{-1.5*sqrt(3)});
\draw[densely dotted] (0.3,{-0.3*sqrt(3)})--(9.7,{-0.3*sqrt(3)}); \draw[densely dotted] (0.6,0)--(5.3,{-4.7*sqrt(3)});
\draw[] (2,0)--(1,{-sqrt(3)});
\draw[very thick] (0,0)--(2,0)--(1,{-sqrt(3)})--cycle;  \draw[thick] (0,0)--(0.5,0)--(0.25,{0.25*-sqrt(3)})--cycle;
\draw[very thick] (10,0)--(8,0)--(9,{-sqrt(3)})--cycle; \draw[thick] (10,0)--(9.5,0)--(9.75,{0.25*-sqrt(3)})--cycle;
\draw[very thick] (5,{-5*sqrt(3)})--(6,{-4*sqrt(3)})--(4,{-4*sqrt(3)})--cycle; \draw[thick] (5,{-5*sqrt(3)})--(5.25,{-4.75*sqrt(3)})--(4.75,{-4.75*sqrt(3)})--cycle;
\draw[very thick] (2.2,0)--(1.9,{-0.3*sqrt(3)})--(8.1,{-0.3*sqrt(3)})--(7.8,0)--cycle;
\draw[very thick] (1.4,{-0.8*sqrt(3)})--(1.1,{-1.1*sqrt(3)})--(3.9,{-3.9*sqrt(3)})--(4.5,{-3.9*sqrt(3)})--cycle;

\draw (-0.2,0) node[left]{$x$}; \draw (10.2,0) node[right]{$y$}; \draw (5,{-5*sqrt(3)-0.2}) node[below]{$z$};
\draw (0.5,{-0.5*sqrt(3)}) node[below left]{\color{red} $\mathcal D_N^x$};
\draw (9.5,{-0.5*sqrt(3)}) node[below right]{\color{red} $\mathcal D_N^y$};
\draw (5.5+0.2,{-4.5*sqrt(3)}) node[right]{\color{red} $\mathcal D_N^z$};
\draw (5,{-0.3*sqrt(3)-0.2}) node[below]{\color{blue} $\mathcal J_N^{x,y}$};
\draw (2.5,{-2.5*sqrt(3)}) node[below left]{\color{blue} $\mathcal J_N^{x,z}$};
\draw (0.25,0.1) node[above]{\color{red} $\mathcal E_N^x$};
\draw (9.75,0.1) node[above]{\color{red} $\mathcal E_N^y$};
\draw (4.875-0.1,{-4.875*sqrt(3)}) node[left]{\color{red} $\mathcal E_N^z$};
\end{tikzpicture}\caption{\label{fig4.1}Macroscopic (left) and microscopic (right) decompositions
of the state space. By taking thicker macroscopic valleys and tubes,
our computations in the microscopic world are allowed to be conducted
mainly on the bulk tube $\mathcal{J}^{x,y}_{N}$.}
\end{figure}
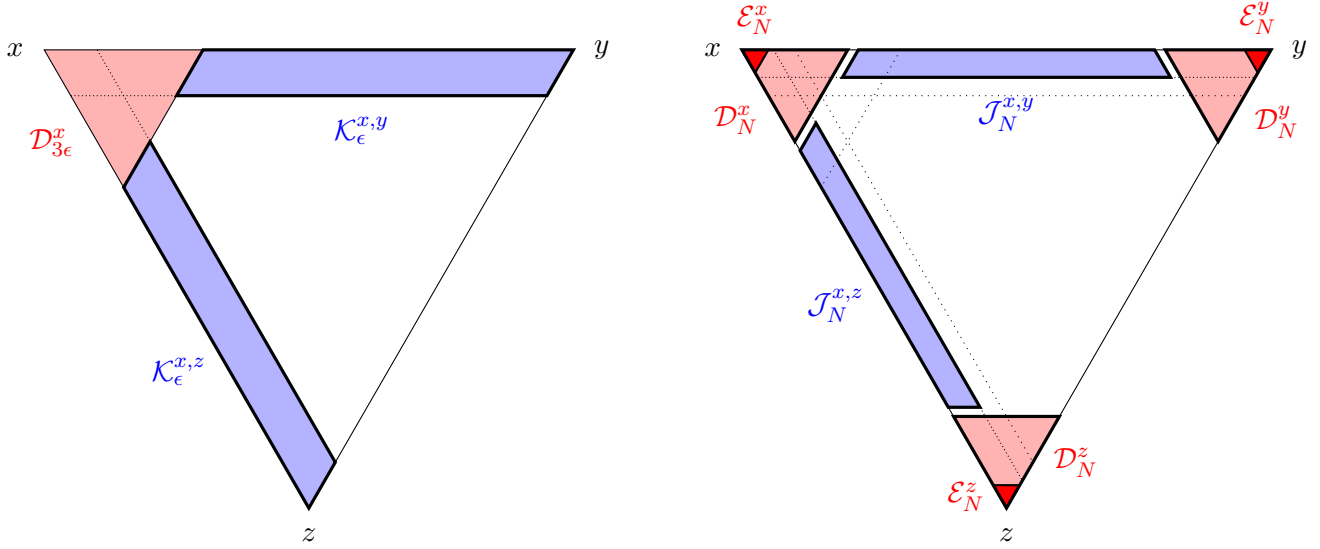

First, we present a macroscopic decomposition of the state space.
This allows us in our computations to neglect all edge terms via truncation
and focus on the calculation on the bulk part. Let $\Omega$ be a
compact subset of $\mathbb{R}^{S}_{+}\times\mathbb{R}^{S}_{+}$ defined
as 
\[
\Omega:=\left\{ \bm{u}=(u_{1},u_{2})\in\mathbb{R}^{S}_{+}\times\mathbb{R}^{S}_{+}:|\bm{u}|:=\sum_{x\in S}(u_{1}(x)+u_{2}(x))=1\right\} .
\]
There is a natural embedding of the configuration space $\Omega_{N}$
into $\Omega$ as
\[
\Omega_{N}\ni\bm{\eta}\longmapsto\frac{\bm{\eta}}{N}\in\Omega.
\]
For each pair of different sites $x,y\in S_{\star}$, consider the
following subsets of $\Omega$:
\[
\mathcal{D}^{x}_{3\epsilon}:=\{\bm{u}\in\Omega:|\bm{u}(x)|:=u_{1}(x)+u_{2}(x)>1-3\epsilon\},\qquad\mathcal{T}^{x,y}_{\epsilon}:=\{\bm{u}\in\Omega:|\bm{u}(x)|+|\bm{u}(y)|\geq1-\epsilon\}.
\]
Note that $\mathcal{T}^{x,y}_{\epsilon}=\mathcal{T}^{y,x}_{\epsilon}$.
Let $\mathcal{K}^{x,y}_{\epsilon}=\mathcal{T}^{x,y}_{\epsilon}\setminus\mathcal{D}^{x}_{3\epsilon}$.
Since $\{\mathcal{K}^{x,y}_{\epsilon}\}_{y\in S_{\star}\setminus\{x\}}$
is a collection of pairwise-disjoint compact subsets of $\Omega$
(since $\epsilon$ is small enough), we can construct a smooth partition
of unity
\begin{equation}
\Theta^{x,y}_{\epsilon}:\Omega\to[0,1],\qquad y\in S_{\star}\setminus\{x\},\label{eq:PoU}
\end{equation}
such that $\sum_{y\in S_{\star}\setminus\{x\}}\Theta^{x,y}_{\epsilon}(\bm{u})=1$
for all $\bm{u}\in\Omega$ and $\Theta^{x,y}_{\epsilon}|_{\mathcal{K}^{x,y}_{\epsilon}}\equiv1$
for each $y\in S_{\star}\setminus\{x\}$.

\subsection{\label{sec4.2}Microscopic Tube Construction}

Next, we define microscopic sets that correspond to the macroscopic
sets in the previous subsection. For some fixed constant $\beta\in(\frac{1}{\alpha},1)$
define
\begin{equation}
\pi_{N}:=\lfloor N^{\beta}\rfloor\qquad\text{such that}\quad\pi_{N}\ll N\ll\pi^{\alpha}_{N}.\label{eq:piN}
\end{equation}
For each $x,y\in S_{\star}$ with $x\ne y$, define
\begin{equation}
\mathcal{D}^{x}_{N}:=\{\bm{\eta}\in\Omega_{N}:|\bm{\eta}(x)|\geq N(1-2\epsilon)\},\qquad\mathcal{T}^{x,y}_{N}:=\{\bm{\eta}\in\Omega_{N}:|\bm{\eta}(x)|+|\bm{\eta}(y)|\geq N-\pi_{N}\},\label{eq:DT}
\end{equation}
and
\begin{equation}
\mathcal{J}^{x,y}_{N}:=\mathcal{T}^{x,y}_{N}\setminus(\mathcal{D}^{x}_{N}\cup\mathcal{D}^{y}_{N})=\left\{ \bm{\eta}\in\Omega_{N}:|\bm{\eta}(x)|+|\bm{\eta}(y)|\geq N-\pi_{N},\enspace|\bm{\eta}(x)|,|\bm{\eta}(y)|<N(1-2\epsilon)\right\} ,\label{eq:J-def}
\end{equation}
In addition, define their one-unit closures as
\begin{equation}
\begin{aligned}\overline{\mathcal{D}}^{x}_{N} & :=\{\bm{\eta}\in\Omega_{N}:|\bm{\eta}(x)|\geq N(1-2\epsilon)-1\},\\
\overline{\mathcal{J}}^{x,y}_{N} & :=\{\bm{\eta}\in\Omega_{N}:|\bm{\eta}(x)|+|\bm{\eta}(y)|\geq N-\pi_{N}-1,\quad|\bm{\eta}(x)|,|\bm{\eta}(y)|\le N(1-2\epsilon)\}.
\end{aligned}
\label{eq:DJ-cl}
\end{equation}
Note that $\overline{\mathcal{J}}^{x,y}_{N}\cap\overline{\mathcal{J}}^{x',y'}_{N}\neq\emptyset$
if and only if $\{x,y\}=\{x',y'\}$. Finally, we define a set 
\begin{equation}
\mathcal{G}^{x}_{N}:=\bigcup_{y\in S_{\star}}\mathcal{D}^{y}_{N}\cup\bigcup_{y\in S_{\star}\setminus\{x\}}\mathcal{J}^{x,y}_{N}.\label{eq:G-def}
\end{equation}
It will be demonstrated in the sequel by direct computations that
the set $\mathcal{G}^{x}_{N}$ serves as the collection of typical
trajectories of the metastable transition from $x\in S_{\star}$ to
another site in $S_{\star}\setminus\{x\}$.

\subsection{\label{sec4.3}Test Function Construction}

Recall from \eqref{eq:eq-pot-RW} that $\mathfrak{h}_{i,U,V}\in\mathbb{R}^{S}$
denotes the equilibrium potential between $U$ and $V$ for the random
walk generated by $r_{i}:S\times S\to[0,\infty)$. Denote by $\mathfrak{h}^{*}_{i,U,V}\in\mathbb{R}^{S}$
the equilibrium potential between $U$ and $V$ for the adjoint random
walk generated by $r^{*}_{i}:S\times S\to[0,\infty)$, given by $m_{i}(x)r^{*}_{i}(x,y):=m_{i}(y)r_{i}(y,x)$.
Then, define macroscopic coordinates $\mathcal{X}^{1}_{x,y},\mathcal{X}^{2}_{x,y}:(\mathbb{N}^{2}_{0})^{S}\to\mathbb{R}$
for each $x\ne y\in S_{\star}$ as
\begin{equation}
\mathcal{X}^{i}_{x,y}(\bm{\eta}):=\frac{1}{N}\sum_{u\in\{x,y\}}\eta_{i}(u)\mathfrak{h}^{*}_{i,x,y}(u)+\min\left\{ \frac{1}{N}\sum_{u\in S\setminus\{x,y\}}\eta_{i}(u)\mathfrak{h}^{*}_{i,x,y}(u),\epsilon\right\} ,\label{eq:Xxyi-def}
\end{equation}
and $\mathcal{X}_{x,y}(\bm{\eta}):=\mathcal{X}^{1}_{x,y}(\bm{\eta})+\mathcal{X}^{2}_{x,y}(\bm{\eta})$.
One can easily see that
\[
0\le\mathcal{X}^{i}_{x,y}(\bm{\eta})\le\frac{|\eta_{i}|}{N},\qquad\text{thus}\qquad0\le\mathcal{X}_{x,y}(\bm{\eta})\le\frac{|\bm{\eta}|}{N}.
\]
Define some constants as (cf. \eqref{eq:ci-def} and \eqref{eq:rhoN-def})
\[
\tau_{N}(x,y):=\frac{\rho_{N}}{c_{1}(x,y)}+\frac{1-\rho_{N}}{c_{2}(x,y)},\qquad\omega_{N}(x,y):=\frac{\frac{1}{c_{1}(x,y)}-\frac{1}{c_{2}(x,y)}}{\tau_{N}(x,y)}.
\]
It is clear that $\lim_{N\to\infty}\tau_{N}(x,y)$ and $\lim_{N\to\infty}\omega_{N}(x,y)$
exist, and
\begin{equation}
1+\omega_{N}(x,y)(1-\rho_{N})=\frac{1}{\tau_{N}(x,y)c_{1}(x,y)},\qquad1-\omega_{N}(x,y)\rho_{N}=\frac{1}{\tau_{N}(x,y)c_{2}(x,y)}.\label{eq:lambda-theta}
\end{equation}
Before starting the construction of the test function $\mathbb{W}_{x}$,
we first introduce a truncation lemma without proof. A similar version
can be found at \cite[Lemma 7.6]{Seo19}.
\begin{lem}
\label{lem:gamma-eps}For all sufficiently small $\epsilon>0$, there
exists a smooth non-decreasing function $\gamma_{\epsilon}:\mathbb{R}\to\mathbb{R}$
such that
\[
\gamma_{\epsilon}(t)=0\qquad\text{for}\quad t\le3\epsilon,\qquad\gamma_{\epsilon}(t)=1\qquad\text{for}\quad t\ge1-3\epsilon,
\]
and 
\[
\lim_{\epsilon\to0}\sup_{t\in[0,1]}|\gamma_{\epsilon}(t)-t|=0.
\]
\end{lem}

Recall \eqref{eq:Ialpha-def}. For $t\in[0,1]$, define
\[
U(t):=\frac{t^{\alpha}(1-t)^{\alpha}}{\mathcal{I}_{\alpha}},\qquad U_{\epsilon}(t):=\frac{\gamma_{\epsilon}(t)^{\alpha}(1-\gamma_{\epsilon}(t))^{\alpha}}{\mathcal{I}_{\alpha,\epsilon}},
\]
where
\[
\mathcal{I}_{\alpha,\epsilon}:=\int^{1}_{0}\gamma_{\epsilon}(t)^{\alpha}(1-\gamma_{\epsilon}(t))^{\alpha}\,{\rm d}t.
\]
For future use, define an error term
\begin{equation}
\Psi_{\epsilon}(t):=U_{\epsilon}(t)-U(t).\label{eq:Psi-eps-def}
\end{equation}
It follows easily from Lemma \ref{lem:gamma-eps} that
\begin{equation}
\lim_{\epsilon\to0}\mathcal{I}_{\alpha,\epsilon}=\mathcal{I}_{\alpha}\qquad\text{and}\qquad\lim_{\epsilon\to0}\sup_{t\in[0,1]}|\Psi_{\epsilon}(t)|=0.\label{eq:Ialpha-Psi-eps}
\end{equation}

\begin{defn}[Test Function $\mathbb{W}_{x}$]
\label{def:Wx-def}Fix $x\in S_{\star}$. For $u,w\ge0$ and $y\in S_{\star}\setminus\{x\}$,
define
\begin{equation}
W_{N}(u,w):=\int^{u+w}_{0}U_{\epsilon}(s)\,{\rm d}s+\omega_{N}(x,y)U_{\epsilon}(u+w)(u-(u+w)\rho_{N}).\label{eq:WN-def}
\end{equation}
Then, define a function $\mathcal{W}_{x,y}:\Omega_{N}\to\mathbb{R}$
for each $y\in S_{\star}\setminus\{x\}$ as 
\begin{equation}
\mathcal{W}_{x,y}(\bm{\eta}):=W_{N}(\mathcal{X}^{1}_{x,y}(\bm{\eta}),\mathcal{X}^{2}_{x,y}(\bm{\eta})).\label{eq:Wxy-def}
\end{equation}
Finally, via the partition of unity \eqref{eq:PoU}, the global test
function $\mathbb{W}_{x}:\Omega_{N}\to\mathbb{R}$ is defined as
\begin{equation}
\mathbb{W}_{x}(\bm{\eta}):=\sum_{y\in S_{\star}\setminus\{x\}}\Theta^{x,y}_{\epsilon}\left(\frac{\bm{\eta}}{N}\right)\mathcal{W}_{x,y}(\bm{\eta}).\label{eq:Wx-def}
\end{equation}
\end{defn}

\begin{rem}
If we neglect the correction term introduced in \eqref{eq:WN-def}
and consider only the leading term
\[
\widehat{W}_{N}(u,w)=\int^{u+w}_{0}U_{\epsilon}(s)\,{\rm d}s,
\]
this matches the test function for the single-species zero-range process
utilized in \cite{BL12a,LMS25,Seo19,Oh19}. However, this simple test
function does not suffice in our setting: the presence of multiple
particle species generates a non-negligible current within the metastable
transition set, preventing the test function from closing the microscopic
identity \eqref{eq:res-int} into the desired macroscopic limit. To
resolve this issue, the correction term $\omega_{N}(x,y)U_{\epsilon}(u+w)(u-(u+w)\rho_{N})$
is indispensable. This term cancels the interior current and adjusts
the leading order coefficient so as to produce the correct effective
transition rate between metastable wells. The remaining fluctuation
is negligible, as established via a variance estimate for hypergeometric
random variables in \eqref{eq:dom-3-1.5}.
\end{rem}

\subsection{\label{sec4.4}Properties of the Test Function $\mathbb{W}_{x}$}

We start with a lemma that $\mathbb{W}_{x}$ is easy to track on $\mathcal{G}^{x}_{N}$.
\begin{lem}
\label{lem:Wx-prop}For each $x\in S_{\star}$, the function $\mathbb{W}_{x}$
is uniformly bounded and
\begin{equation}
\mathbb{W}_{x}(\bm{\eta})=\begin{cases}
1 & \text{if}\quad\bm{\eta}\in\overline{\mathcal{D}}^{x}_{N},\\
\mathcal{W}_{x,y}(\bm{\eta}) & \text{if}\quad\bm{\eta}\in\overline{\mathcal{J}}^{x,y}_{N},\enspace y\in S_{\star}\setminus\{x\},\\
0 & \text{if}\quad\bm{\eta}\in\overline{\mathcal{D}}^{y}_{N},\enspace y\in S_{\star}\setminus\{x\}.
\end{cases}\label{eq:Wx-prop}
\end{equation}
\end{lem}

\begin{proof}
Uniform boundedness follows directly from \eqref{eq:WN-def}, \eqref{eq:Wxy-def},
and \eqref{eq:Wx-def}. For the second statement, we first consider
the value of $\mathbb{W}_{x}$ on $\overline{\mathcal{D}}^{x}_{N}$.
By definition, $|\bm{\eta}(x)|\geq N(1-2\epsilon)-1$ for $\bm{\eta}\in\overline{\mathcal{D}}^{x}_{N}$
which implies $\mathcal{X}_{x,y}(\bm{\eta})\geq1-3\epsilon$. Thus,
by the cutoff property of $\gamma_{\epsilon}$ from Lemma \ref{lem:gamma-eps},
we can notice that $\mathcal{W}_{x,y}(\bm{\eta})=1$ for all $y\in S_{\star}\setminus\{x\}$.
Therefore, since $\{\Theta^{x,y}_{\epsilon}\}_{y\in S_{\star}\setminus\{x\}}$
partitions the unity, we obtain that $\mathbb{W}_{x}\equiv1$ on $\overline{\mathcal{D}}^{x}_{N}$.

Next, consider $\overline{\mathcal{D}}^{y}_{N}$ for $y\in S_{\star}\setminus\{x\}$.
For each $\bm{\eta}\in\overline{\mathcal{D}}^{y}_{N}$ we have $|\bm{\eta}(y)|\geq N(1-2\epsilon)-1$
and thus $|\bm{\eta}(x)|\le2N\epsilon+1$. We separate the cases according
to the value of $|\bm{\eta}(x)|$.
\begin{itemize}
\item Assume that $|\bm{\eta}(x)|\leq N\epsilon$. Then, $\mathcal{X}_{x,z}(\bm{\eta})\leq3\epsilon$
for all $z\in S_{\star}\setminus\{x\}$. Therefore, by the definition
of $\mathcal{W}_{x,z}$ and Lemma \ref{lem:gamma-eps}, we have $\mathcal{W}_{x,z}(\bm{\eta})=0$
for all $z\in S_{\star}\setminus\{x\}$, and thus $\mathbb{W}_{x}(\bm{\eta})=0$.
\item Assume that $|\bm{\eta}(x)|\in(N\epsilon,2N\epsilon+1]$. Then, one
can notice that
\[
|\bm{\eta}(x)|+|\bm{\eta}(y)|\geq N(1-\epsilon),
\]
which leads to $\frac{\bm{\eta}}{N}\in\mathcal{K}^{x,y}_{\epsilon}$.
Therefore, since $\{\Theta^{x,y}_{\epsilon}\}$ partitions the unity,
$\mathbb{W}_{x}(\bm{\eta})=\mathcal{W}_{x,y}(\bm{\eta})$. Since 
\[
\mathcal{X}_{x,y}(\bm{\eta})\leq1-\frac{|\bm{\eta}(y)|}{N}\leq3\epsilon,
\]
Lemma \ref{lem:gamma-eps} implies $\mathcal{W}_{x,y}(\bm{\eta})=0$,
and in consequence, $\mathbb{W}_{x}(\bm{\eta})=0$.
\end{itemize}
Now, consider $\bm{\eta}\in\overline{\mathcal{J}}^{x,y}_{N}$ for
each $y\in S_{\star}\setminus\{x\}$, such that $|\bm{\eta}(x)|\leq N(1-2\epsilon)$.
As before, we divide into two cases.
\begin{itemize}
\item Assume that $N(1-3\epsilon)\leq|\bm{\eta}(x)|\leq N(1-2\epsilon)$.
Directly from the assumption, for each $z\in S_{\star}\setminus\{x\}$
we have $\mathcal{X}_{x,z}(\bm{\eta})\geq1-3\epsilon$ which gives
$\mathcal{W}_{x,z}(\bm{\eta})=1$. Since this holds for all $z\in S_{\star}\setminus\{x\}$,
we conclude that $\mathbb{W}_{x}(\bm{\eta})=1=\mathcal{W}_{x,y}(\bm{\eta})$.
\item Finally, assume that $|\bm{\eta}(x)|<N(1-3\epsilon)$. Since $|\bm{\eta}(x)|+|\bm{\eta}(y)|\ge N-\pi_{N}-1\ge N(1-\epsilon)$
and $|\bm{\eta}(x)|<N(1-3\epsilon)$, we have $\frac{\bm{\eta}}{N}\in\mathcal{K}^{x,y}_{\epsilon}$.
Therefore, similarly, $\mathbb{W}_{x}(\bm{\eta})=\mathcal{W}_{x,y}(\bm{\eta})$.
\end{itemize}
\end{proof}

Next, we analyze the increments of the function $\mathcal{W}_{x,y}$
on $\mathcal{J}^{x,y}_{N}$ for $x\in S_{\star}$ and $y\in S_{\star}\setminus\{x\}$.
Consider the particle movement $\bm{\eta}\to\bm{\eta}-\delta^{i}_{u}+\delta^{i}_{w}$
for $\bm{\eta}\in\mathcal{J}^{x,y}_{N}$ and the corresponding increment.
First, by the smoothness of $\gamma_{\epsilon}$ and $\Theta^{x,y}_{\epsilon}$,
one easily obtains rough bounds as
\begin{align}
\max_{x\in S_{\star}}\max_{y\in S_{\star}\setminus\{x\}}\max_{i\in\{1,2\}}\max_{u,w\in S}\max_{\bm{\eta}\in\Omega_{N}}|U_{\epsilon}(\mathcal{X}_{x,y}(\bm{\eta}-\delta^{i}_{u}+\delta^{i}_{w}))-U_{\epsilon}(\mathcal{X}_{x,y}(\bm{\eta}))| & \leq\frac{c}{N},\label{eq:Lip1}\\
\max_{x\in S_{\star}}\max_{i\in\{1,2\}}\max_{u,w\in S}\max_{\bm{\eta}\in\Omega_{N}}|\mathbb{W}_{x}(\bm{\eta}-\delta^{i}_{u}+\delta^{i}_{w})-\mathbb{W}_{x}(\bm{\eta})| & \leq\frac{c}{N}.\label{eq:Lip2}
\end{align}
A more thorough increment analysis is given below, which is needed
in Section \ref{sec5.3}.
\begin{lem}
\label{lem:Wx-incr}Let $x\in S_{\star}$, $y\in S_{\star}\setminus\{x\}$,
and $u,w\in S$ such that $r_{i}(u,w)>0$. For each $\bm{\eta}\in\mathcal{J}^{x,y}_{N}$
and $\bm{\xi}:=\bm{\eta}-\delta^{i}_{u}\in\Omega^{i,-}_{N-1}$ (cf.
\eqref{eq:Omega-minus}) where $\eta_{i}(u)\ge1$,
\[
\mathcal{W}_{x,y}(\bm{\xi}+\delta^{i}_{w})-\mathcal{W}_{x,y}(\bm{\xi}+\delta^{i}_{u})=\left(\frac{\mathfrak{h}^{*}_{i,x,y}(w)-\mathfrak{h}^{*}_{i,x,y}(u)}{N\tau_{N}(x,y)c_{i}(x,y)}\right)\frac{|\bm{\xi}(x)|^{\alpha}|\bm{\xi}(y)|^{\alpha}}{N^{2\alpha}\mathcal{I}_{\alpha}}+\mathcal{U}^{i}_{x,y}(\bm{\xi};u,w)+\mathcal{V}^{i}_{x,y}(\bm{\xi};u,w),
\]
 where
\[
\mathcal{U}^{i}_{x,y}(\bm{\xi};u,w)=O\left(\frac{1}{N}\right)o_{\epsilon}(1)+O\left(\frac{\pi_{N}}{N^{2}}\right)\qquad\text{and}\qquad\mathcal{V}^{i}_{x,y}(\bm{\xi};u,w)=O\left(\frac{\xi_{1}(x)}{N^{2}}-\frac{|\bm{\xi}(x)|}{N^{2}}\rho_{N}\right).
\]
Above, $o_{\epsilon}(1)$ denotes a term which is independent of $N$
and vanishes as $\epsilon\to0$.
\end{lem}

\begin{proof}
Without loss of generality, we only consider the case of $i=1$. By
the definition of $\mathcal{W}_{x,y}$ \eqref{eq:Wxy-def}, we can
separate the increment in the left-hand side into
\[
\mathcal{W}_{x,y}(\bm{\xi}+\delta^{1}_{w})-\mathcal{W}_{x,y}(\bm{\xi}+\delta^{1}_{u})=:{\rm I}+{\rm II}
\]
where
\[
{\rm I}:=\int^{\mathcal{X}_{x,y}(\bm{\xi}+\delta^{1}_{w})}_{\mathcal{X}_{x,y}(\bm{\xi}+\delta^{1}_{u})}U_{\epsilon}(s)\,{\rm d}s,
\]
and
\begin{align*}
{\rm II}:=\omega_{N}(x,y)U_{\epsilon}(\mathcal{X}_{x,y}(\bm{\xi}+\delta^{1}_{w})) & \left(\mathcal{X}^{1}_{x,y}(\bm{\xi}+\delta^{1}_{w})-\mathcal{X}_{x,y}(\bm{\xi}+\delta^{1}_{w})\rho_{N}\right)\\
 & -\omega_{N}(x,y)U_{\epsilon}(\mathcal{X}_{x,y}(\bm{\xi}+\delta^{1}_{u}))\left(\mathcal{X}^{1}_{x,y}(\bm{\xi}+\delta^{1}_{u})-\mathcal{X}_{x,y}(\bm{\xi}+\delta^{1}_{u})\rho_{N}\right).
\end{align*}
Since $\bm{\eta}\in\mathcal{J}^{x,y}_{N}$, we have $\sum_{u\neq x,y}|\bm{\eta}(u)|\le\pi_{N}$.
This implies via \eqref{eq:Xxyi-def} that
\[
\mathcal{X}^{i}_{x,y}(\bm{\eta})=\frac{1}{N}\sum_{u\in S}\eta_{i}(u)\mathfrak{h}^{*}_{i,x,y}(u).
\]
From this and the facts that $\mathfrak{h}^{*}_{i,x,y}(x)=1$ and
$\mathfrak{h}^{*}_{i,x,y}(y)=0$, we note that
\begin{equation}
\mathcal{X}^{1}_{x,y}(\bm{\xi}+\delta^{1}_{w})=\mathcal{X}^{1}_{x,y}(\bm{\xi})+\frac{\mathfrak{h}^{*}_{1,x,y}(w)}{N},\qquad\mathcal{X}_{x,y}(\bm{\xi}+\delta^{1}_{w})=\mathcal{X}_{x,y}(\bm{\xi})+\frac{\mathfrak{h}^{*}_{1,x,y}(w)}{N},\label{eq:eq1}
\end{equation}
and the same for $u$ in place of $w$.

Recall \eqref{eq:Psi-eps-def}. We claim that, for any $s$ between
$\mathcal{X}_{x,y}(\bm{\xi}+\delta^{1}_{u})$ and $\mathcal{X}_{x,y}(\bm{\xi}+\delta^{1}_{w})$,
\begin{equation}
U_{\epsilon}(s)=\frac{|\bm{\xi}(x)|^{\alpha}|\bm{\xi}(y)|^{\alpha}}{N^{2\alpha}\mathcal{I}_{\alpha}}+o_{\epsilon}(1)+O\left(\frac{\pi_{N}}{N}\right).\label{eq:eq2}
\end{equation}
Indeed, since $|s-\frac{|\bm{\xi}(x)|}{N}|\leq\frac{\pi_{N}+1}{N}$
and $U_{\epsilon}'$ is uniformly bounded, via Taylor's theorem we
have 
\[
U_{\epsilon}(s)=U_{\epsilon}\left(\frac{|\bm{\xi}(x)|}{N}\right)+O\left(\frac{\pi_{N}}{N}\right)=\frac{\left(\frac{|\bm{\xi}(x)|}{N}\right)^{\alpha}\left(1-\frac{|\bm{\xi}(x)|}{N}\right)^{\alpha}}{\mathcal{I}_{\alpha}}+\Psi_{\epsilon}\left(\frac{|\bm{\xi}(x)|}{N}\right)+O\left(\frac{\pi_{N}}{N}\right).
\]
Combining this with $(1-\frac{|\bm{\xi}(x)|}{N})^{\alpha}=(\frac{|\bm{\xi}(y)|}{N})^{\alpha}+O(\frac{\pi_{N}}{N})$
and \eqref{eq:Ialpha-Psi-eps} leads to the proof of the claim.

By \eqref{eq:eq1}, \eqref{eq:eq2}, and the mean value theorem, we
deduce that
\begin{equation}
{\rm I}=\left(\frac{\mathfrak{h}^{*}_{1,x,y}(w)-\mathfrak{h}^{*}_{1,x,y}(u)}{N}\right)\frac{|\bm{\xi}(x)|^{\alpha}|\bm{\xi}(y)|^{\alpha}}{N^{2\alpha}\mathcal{I}_{\alpha}}+O\left(\frac{1}{N}\right)o_{\epsilon}(1)+O\left(\frac{\pi_{N}}{N^{2}}\right).\label{eq:eq3}
\end{equation}
Moreover, note from \eqref{eq:eq1} that
\[
\mathcal{X}_{x,y}(\bm{\xi}+\delta^{1}_{w})-\mathcal{X}_{x,y}(\bm{\xi}+\delta^{1}_{u})=\mathcal{X}^{1}_{x,y}(\bm{\xi}+\delta^{1}_{w})-\mathcal{X}^{1}_{x,y}(\bm{\xi}+\delta^{1}_{u})=\frac{\mathfrak{h}^{*}_{1,x,y}(w)-\mathfrak{h}^{*}_{1,x,y}(u)}{N}.
\]
Thus, we can rewrite the term ${\rm II}$ as
\begin{align*}
{\rm II}=\omega_{N}(x,y) & \left(U_{\epsilon}(\mathcal{X}_{x,y}(\bm{\xi}+\delta^{1}_{w}))-U_{\epsilon}(\mathcal{X}_{x,y}(\bm{\xi}+\delta^{1}_{u}))\right)\left(\mathcal{X}^{1}_{x,y}(\bm{\xi}+\delta^{1}_{u})-\mathcal{X}_{x,y}(\bm{\xi}+\delta^{1}_{u})\rho_{N}\right)\\
 & +\omega_{N}(x,y)(1-\rho_{N})\left(\frac{\mathfrak{h}^{*}_{1,x,y}(w)-\mathfrak{h}^{*}_{1,x,y}(u)}{N}\right)U_{\epsilon}(\mathcal{X}_{x,y}(\bm{\xi}+\delta^{1}_{w}))=:{\rm II}_{1}+{\rm II}_{2}.
\end{align*}
Using the facts
\[
\left|\mathcal{X}^{1}_{x,y}(\bm{\xi}+\delta^{1}_{u})-\frac{\xi_{1}(x)}{N}\right|,\left|\mathcal{X}_{x,y}(\bm{\xi}+\delta^{1}_{u})-\frac{|\bm{\xi}(x)|}{N}\right|\leq\frac{\pi_{N}+1}{N},
\]
along with \eqref{eq:Lip1}, we can rewrite ${\rm II}_{1}$ as 
\begin{equation}
{\rm II}_{1}=O\left(\frac{1}{N}\right)\left(\frac{\xi_{1}(x)}{N}-\frac{|\bm{\xi}(x)|}{N}\rho_{N}\right)+O\left(\frac{\pi_{N}}{N^{2}}\right).\label{eq:eq4}
\end{equation}
Finally, by the claim at \eqref{eq:eq2},
\begin{equation}
{\rm II}_{2}=\omega_{N}(x,y)(1-\rho_{N})\left(\frac{\mathfrak{h}^{*}_{1,x,y}(w)-\mathfrak{h}^{*}_{1,x,y}(u)}{N}\right)\left[\frac{|\bm{\xi}(x)|^{\alpha}|\bm{\xi}(y)|^{\alpha}}{N^{2\alpha}\mathcal{I}_{\alpha}}+o_{\epsilon}(1)+O\left(\frac{\pi_{N}}{N}\right)\right]+O\left(\frac{\pi_{N}}{N^{2}}\right).\label{eq:eq5}
\end{equation}
Combining \eqref{eq:eq3}, \eqref{eq:eq4}, and \eqref{eq:eq5}, along
with \eqref{eq:lambda-theta}, concludes the proof.
\end{proof}

\section{\label{sec5}Resolvent--$H^{1}$ Analysis on the Tube of Typical
Trajectories}

In this section, we prove Proposition \ref{prop2} by following the
idea presented in Section \ref{sec1.4}. Namely, we consider the (accelerated)
microscopic resolvent equation
\begin{equation}
(\lambda-\theta_{N}\mathcal{L}_{N})F_{N}=G_{N},\label{eq:resol-equ}
\end{equation}
where the function $G_{N}:\Omega_{N}\to\mathbb{R}$ is from \eqref{eq:GN-def}.
As done at \eqref{eq:res-int}, we prove that the test function $\mathbb{W}_{x}$
defined in Section \ref{sec4} approximates the equilibrium potential
of the adjoint process $\mathfrak{H}^{*}_{x}$ between $\mathcal{E}^{x}_{N}$
and 
\begin{equation}
\breve{\mathcal{E}}^{x}_{N}:=\mathcal{E}_{N}\setminus\mathcal{E}^{x}_{N},\label{eq:ENx-breve}
\end{equation}
well enough so that the equation 
\begin{equation}
\lambda\langle F_{N},\mathbb{\mathbb{W}}_{x}\rangle_{\nu_{N}}+\theta_{N}\langle F_{N},-\mathcal{L}^{*}_{N}\mathbb{W}_{x}\rangle_{\nu_{N}}=\langle G_{N},\mathbb{W}_{x}\rangle_{\nu_{N}}\label{eq:inner-product}
\end{equation}
converges to the macroscopic resolvent equation
\[
(\lambda-\mathfrak{L}_{\mathbb{X}})\mathfrak{f}(x)=\mathfrak{g}(x),
\]
as $N\to\infty$ and then $\epsilon\downarrow0$. The first and third
terms in \eqref{eq:inner-product} are straightforward:
\begin{lem}
\label{lem:res-1}We have
\[
\langle G_{N},\mathbb{W}_{x}\rangle_{\nu_{N}}=\mathfrak{g}(x)\nu_{N}(\mathcal{E}^{x}_{N})
\]
and
\[
\lambda\langle F_{N},\mathbb{\mathbb{W}}_{x}\rangle_{\nu_{N}}=\lambda\nu_{N}(\mathcal{E}^{x}_{N})f_{N}(x)+o(1).
\]
\end{lem}

\begin{proof}
By the definition of $G_{N}$, the first inner product becomes
\[
\langle G_{N},\mathbb{W}_{x}\rangle_{\nu_{N}}=\sum_{y\in S_{\star}}\mathfrak{g}(y)\sum_{\bm{\eta}\in\mathcal{E}^{y}_{N}}\nu_{N}(\bm{\eta})\mathbb{W}_{x}(\bm{\eta}).
\]
Since $\mathcal{E}^{y}_{N}\subset\mathcal{D}^{y}_{N}$, Lemma \ref{lem:Wx-prop}
implies that
\[
\langle G_{N},\mathbb{W}_{x}\rangle_{\nu_{N}}=\mathfrak{g}(x)\sum_{\bm{\eta}\in\mathcal{E}^{x}_{N}}\nu_{N}(\bm{\eta})=\mathfrak{g}(x)\nu_{N}(\mathcal{E}^{x}_{N}).
\]
Next, the second inner product is decomposed into (cf. \eqref{eq:ENx-breve})
\[
\lambda\langle F_{N},\mathbb{\mathbb{W}}_{x}\rangle_{\nu_{N}}=\lambda\sum_{\bm{\eta}\in\mathcal{E}^{x}_{N}}F_{N}(\bm{\eta})\mathbb{W}_{x}(\bm{\eta})\nu_{N}(\bm{\eta})+\lambda\sum_{\bm{\eta}\in\breve{\mathcal{E}}^{x}_{N}}F_{N}(\bm{\eta})\mathbb{W}_{x}(\bm{\eta})\nu_{N}(\bm{\eta})+\lambda\sum_{\bm{\eta}\in\Delta_{N}}F_{N}(\bm{\eta})\mathbb{W}_{x}(\bm{\eta})\nu_{N}(\bm{\eta}).
\]
From Lemma \ref{lem:Wx-prop}, the second term at the right-hand side
is zero. The third term at the right-hand side is negligible as $N\to\infty$
by Theorem \ref{thm:cond}-(1) and the boundedness of $F_{N}$ and
$\mathbb{W}_{x}$. Thus, according to the definition of $f_{N}$ at
\eqref{eq:fN-def} and Lemma \ref{lem:Wx-prop},
\[
\lambda\langle F_{N},\mathbb{\mathbb{W}}_{x}\rangle_{\nu_{N}}=\lambda\nu_{N}(\mathcal{E}^{x}_{N})f_{N}(x)+o(1),
\]
as desired.
\end{proof}

The rest of the section is devoted to the estimation of the second
inner product in \eqref{eq:inner-product}, which is $\langle F_{N},-\mathcal{L}^{*}_{N}\mathbb{W}_{x}\rangle_{\nu_{N}}$.

\subsection{\label{sec5.1}A Priori Estimates}

Before we get down to the estimation, let us first introduce two lemmas.
\begin{lem}
\label{lem:PE1}We have
\[
\sup_{N\ge1}\theta_{N}\mathcal{D}_{N}(F_{N})<\infty.
\]
\end{lem}

\begin{proof}
By the definition of the Dirichlet form, we may express $\mathcal{D}_{N}(F_{N})$
as 
\[
\mathcal{D}_{N}(F_{N})=-\sum_{\bm{\eta}\in\Omega_{N}}F_{N}(\bm{\eta})\mathcal{L}_{N}F_{N}(\bm{\eta})\nu_{N}(\bm{\eta}).
\]
Since $F_{N}$ is the solution of \eqref{eq:resol-equ}, we have 
\[
\theta_{N}\mathcal{D}_{N}(F_{N})=\sum_{\bm{\eta}\in\Omega_{N}}F_{N}(\bm{\eta})(G_{N}-\lambda F_{N})(\bm{\eta})\nu_{N}(\bm{\eta})\leq\|F_{N}\|_{\infty}\|G_{N}-\lambda F_{N}\|_{\infty}.
\]
The uniform boundedness of $F_{N}$ and $G_{N}$ thus completes the
proof.
\end{proof}

\begin{lem}
\label{lem:PE2}As in \eqref{eq:Gibar-def}, define $\overline{F}^{i}_{N}:\Omega^{i,-}_{N-1}\to\mathbb{R}$
as
\[
\overline{F}^{i}_{N}(\bm{\xi}):=\frac{1}{\kappa}\sum_{z\in S}F_{N}(\bm{\xi}+\delta^{i}_{z}).
\]
Then,
\[
\sum_{i\in\{1,2\}}\sum_{u,w\in S}\sum_{\substack{\bm{\eta}\in\Omega_{N}:\\
\eta_{i}(u)\geq1
}
}\nu_{N}(\bm{\eta})\bm{g}_{i}(\bm{\eta}(u))r^{*}_{i}(u,w)\left(F_{N}(\bm{\eta})-\overline{F}^{i}_{N}(\bm{\eta}-\delta^{i}_{u})\right)^{2}\leq c\mathcal{D}_{N}(F_{N}).
\]
\end{lem}

\begin{proof}
By a change of variables $\bm{\eta}-\delta^{i}_{u}=\bm{\xi}$, the
left-hand side can be expressed as
\[
\sum_{i\in\{1,2\}}\sum_{\bm{\xi}\in\Omega^{i,-}_{N-1}}\frac{N^{\alpha}}{Z_{N}\binom{N}{A}}\frac{\bm{m}^{\bm{\xi}}}{\bm{g!}(\bm{\xi})}\sum_{u,w\in S}m_{i}(u)r^{*}_{i}(u,w)\left(F_{N}(\bm{\xi}+\delta^{i}_{u})-\overline{F}^{i}_{N}(\bm{\xi})\right)^{2}.
\]
By the same computation conducted at \eqref{eq:sector-5}, we obtain
that 
\[
\sum_{u,w\in S}m_{i}(u)r^{*}_{i}(u,w)\left(F_{N}(\bm{\xi}+\delta^{i}_{u})-\overline{F}^{i}_{N}(\bm{\xi})\right)^{2}\le c\sum_{x,y\in S}m_{i}(x)r^{*}_{i}(x,y)\left(F_{N}(\bm{\xi}+\delta^{i}_{x})-F_{N}(\bm{\xi}+\delta^{i}_{y})\right)^{2}.
\]
Thus, returning to $\bm{\xi}+\delta^{i}_{x}=\bm{\eta}$ completes
the proof, via \eqref{eq:Diri-adj}.
\end{proof}

\subsection{\label{sec5.2}Negligible Part}

Recall that our remaining goal is to estimate $\langle F_{N},-\mathcal{L}^{*}_{N}\mathbb{W}_{x}\rangle_{\nu_{N}}$.
For $F,G\in\mathbb{R}^{\Omega_{N}}$and $\mathcal{E}\subset\Omega_{N}$,
define
\[
\langle F,G\rangle_{\nu_{N},\mathcal{E}}:=\sum_{\bm{\eta}\in\mathcal{E}}F(\bm{\eta})G(\bm{\eta})\nu_{N}(\bm{\eta}).
\]
In this sense (cf. \eqref{eq:G-def}), we have
\[
\langle F_{N},-\mathcal{L}^{*}_{N}\mathbb{W}_{x}\rangle_{\nu_{N}}=\langle F_{N},-\mathcal{L}^{*}_{N}\mathbb{W}_{x}\rangle_{\nu_{N},\mathcal{G}^{x}_{N}}+\langle F_{N},-\mathcal{L}^{*}_{N}\mathbb{W}_{x}\rangle_{\nu_{N},(\mathcal{G}^{x}_{N})^{c}}.
\]
In this subsection, we prove that the latter part is negligible:
\begin{lem}
\label{lem:res-2}For each $x\in S_{\star}$ we have 
\[
\langle F_{N},-\mathcal{L}^{*}_{N}\mathbb{W}_{x}\rangle_{\nu_{N},(\mathcal{G}^{x}_{N})^{c}}=o(N^{-1-\alpha}).
\]
\end{lem}

Recall that we may represent the increment $\bm{\eta}\to\bm{\eta}-\delta^{i}_{u}+\delta^{i}_{w}$
as $\bm{\xi}+\delta^{i}_{u}\to\bm{\xi}+\delta^{i}_{w}$ with $\bm{\xi}:=\bm{\eta}-\delta^{i}_{u}$.
In this regard, we introduce an admissible departure site collection
$\Xi^{i}_{x}(\bm{\xi})$, corresponding to each $\bm{\xi}\in\Omega^{i,-}_{N-1}$,
as\textbf{ }
\begin{equation}
\Xi^{i}_{x}(\bm{\xi}):=\left\{ u\in S:\bm{\xi}+\delta^{i}_{u}\in(\mathcal{G}^{x}_{N})^{c}\right\} .\label{eq:Xix-def}
\end{equation}
Then, write
\[
\mathcal{A}^{i,-}_{N-1}:=\left\{ \bm{\xi}\in\Omega^{i,-}_{N-1}:\Xi^{i}_{x}(\bm{\xi})\neq\emptyset\right\} =\bigcup_{u\in S}\left\{ \bm{\xi}\in\Omega^{i,-}_{N-1}:\bm{\xi}+\delta^{i}_{u}\in(\mathcal{G}^{x}_{N})^{c}\right\} .
\]
In this sense, we may write
\begin{equation}
\langle F_{N},-\mathcal{L}^{*}_{N}\mathbb{W}_{x}\rangle_{\nu_{N},(\mathcal{G}^{x}_{N})^{c}}=\langle\!\langle F_{N}-\mathcal{L}^{*}_{N}\mathbb{W}_{x}\rangle\!\rangle_{\mathcal{A}^{-}_{N-1}},\label{eq:res-2-1}
\end{equation}
for $\mathcal{A}^{-}_{N-1}=\mathcal{A}^{1,-}_{N-1}\times\mathcal{A}^{2,-}_{N-1}$
where 
\begin{equation}
\begin{aligned} & \langle\!\langle F,-\mathcal{L}^{*}_{N}G\rangle\!\rangle_{\mathcal{F}}\\
 & :=\sum_{i\in\{1,2\}}\sum_{\bm{\xi}\in\mathcal{F}^{i}}\sum_{u\in\Xi^{i}_{x}(\bm{\xi})}\sum_{w\in S}\nu_{N}(\bm{\xi}+\delta^{i}_{u})\bm{g}_{i}((\bm{\xi}+\delta^{i}_{u})(u))r^{*}_{i}(u,w)F(\bm{\xi}+\delta^{i}_{u})\left[G(\bm{\xi}+\delta^{i}_{u})-G(\bm{\xi}+\delta^{i}_{w})\right]\\
 & =\sum_{i\in\{1,2\}}\sum_{\bm{\xi}\in\mathcal{F}^{i}}\sum_{u\in\Xi^{i}_{x}(\bm{\xi})}\sum_{w\in S}\frac{N^{\alpha}}{Z_{N}\binom{N}{A}}\frac{\bm{m}^{\bm{\xi}}}{\bm{g!}(\bm{\xi})}m_{i}(u)r^{*}_{i}(u,w)F(\bm{\xi}+\delta^{i}_{u})\left[G(\bm{\xi}+\delta^{i}_{u})-G(\bm{\xi}+\delta^{i}_{w})\right],
\end{aligned}
\label{eq:inner-prod-sim}
\end{equation}
for $\mathcal{F}=\mathcal{F}^{1}\times\mathcal{F}^{2}\subset\Omega^{1,-}_{N-1}\times\Omega^{2,-}_{N-1}$.
In addition, define the bulk part $\mathcal{A}^{i,{\rm bulk}}_{N-1}$
as
\begin{align*}
 & \mathcal{A}^{i,{\rm bulk}}_{N-1}:=\left\{ \bm{\xi}\in\Omega^{i,-}_{N-1}:\Xi^{i}_{x}(\bm{\xi})=S\right\} \\
 & =\left\{ \bm{\xi}\in\Omega^{i,-}_{N-1}:|\bm{\xi}(y)|<N(1-2\epsilon)-1,\enskip\forall y\in S_{\star},\quad|\bm{\xi}(x)|+|\bm{\xi}(z)|<N-\pi_{N}-1,\enspace\forall z\in S_{\star}\setminus\{x\}\right\} ,
\end{align*}
and let $\mathcal{A}^{i,{\rm edge}}_{N-1}:=\mathcal{A}^{i,-}_{N-1}\setminus\mathcal{A}^{i,{\rm bulk}}_{N-1}$
such that
\begin{equation}
\langle\!\langle F_{N},-\mathcal{L}^{*}_{N}\mathbb{W}_{x}\rangle\!\rangle_{\mathcal{A}^{-}_{N-1}}=\langle\!\langle F_{N},-\mathcal{L}^{*}_{N}\mathbb{W}_{x}\rangle\!\rangle_{\mathcal{A}^{{\rm bulk}}_{N-1}}+\langle\!\langle F_{N},-\mathcal{L}^{*}_{N}\mathbb{W}_{x}\rangle\!\rangle_{\mathcal{A}^{{\rm edge}}_{N-1}},\label{eq:res-2-2}
\end{equation}
with $\mathcal{A}^{{\rm bulk}}_{N-1}:=\mathcal{A}^{1,{\rm bulk}}_{N-1}\times\mathcal{A}^{2,{\rm bulk}}_{N-1}$
and $\mathcal{A}^{{\rm edge}}_{N-1}:=\mathcal{A}^{1,{\rm edge}}_{N-1}\times\mathcal{A}^{2,{\rm edge}}_{N-1}$.

\begin{figure}
\begin{tikzpicture}[scale=0.8]
\fill[red!30!white] (0,0)--(3,0)--(1.5,{-1.5*sqrt(3)})--cycle;
\fill[red!30!white] (10,0)--(7,0)--(8.5,{-1.5*sqrt(3)})--cycle;
\fill[red!30!white] (5,{-5*sqrt(3)})--(6.5,{-3.5*sqrt(3)})--(3.5,{-3.5*sqrt(3)})--cycle;
\fill[blue!30!white] (3.2,0)--(2.7,{-0.5*sqrt(3)})--(7.3,{-0.5*sqrt(3)})--(6.8,0)--cycle;
\fill[blue!30!white] (2.1,{-1.1*sqrt(3)})--(1.6,{-1.6*sqrt(3)})--(3.4,{-3.4*sqrt(3)})--(4.4,{-3.4*sqrt(3)})--cycle;
\fill[teal!30!white] (2.6,{-0.6*sqrt(3)})--(2.2,{-1*sqrt(3)})--(4.6,{-3.4*sqrt(3)})--(6.4,{-3.4*sqrt(3)})--(8.3,{-1.5*sqrt(3)})--(7.4,{-0.6*sqrt(3)})--cycle;

\draw[] (0,0)--(10,0)--(5,{-5*sqrt(3)})--cycle;
\draw[dotted] (0,0)--(3,0)--(1.5,{-1.5*sqrt(3)})--cycle;
\draw[dotted] (10,0)--(7,0)--(8.5,{-1.5*sqrt(3)})--cycle;
\draw[dotted] (5,{-5*sqrt(3)})--(6.5,{-3.5*sqrt(3)})--(3.5,{-3.5*sqrt(3)})--cycle;
\draw[] (3.2,0)--(2.7,{-0.5*sqrt(3)})--(7.3,{-0.5*sqrt(3)})--(6.8,0)--cycle;
\draw[] (2.1,{-1.1*sqrt(3)})--(1.6,{-1.6*sqrt(3)})--(3.4,{-3.4*sqrt(3)})--(4.4,{-3.4*sqrt(3)})--cycle;
\draw[] (2.6,{-0.6*sqrt(3)})--(2.2,{-1*sqrt(3)})--(4.6,{-3.4*sqrt(3)})--(6.4,{-3.4*sqrt(3)})--(8.3,{-1.5*sqrt(3)})--(7.4,{-0.6*sqrt(3)})--cycle;

\draw[very thick] (2.6,{-0.6*sqrt(3)})--(2.2,{-1*sqrt(3)}); \draw (2.4,{-0.8*sqrt(3)+0.2}) node[left]{\tiny$\mathcal A_{N-1}^{i,x}$};
\draw[very thick] (2.6,{-0.6*sqrt(3)})--(7.4,{-0.6*sqrt(3)}); \draw (5,{-0.6*sqrt(3)}) node[below]{\tiny$\mathcal A_{N-1}^{i,x,y}$};
\draw[very thick] (8.3,{-1.5*sqrt(3)})--(7.4,{-0.6*sqrt(3)}); \draw (7.85-0.3,{-1.05*sqrt(3)}) node[below]{\tiny$\mathcal A_{N-1}^{i,y}$};
\draw[very thick] (2.2,{-1*sqrt(3)})--(4.6,{-3.4*sqrt(3)}); \draw (3.4+0.1,{-2.2*sqrt(3)}) node[right]{\tiny$\mathcal A_{N-1}^{i,x,z}$};
\draw[very thick] (4.6,{-3.4*sqrt(3)})--(6.4,{-3.4*sqrt(3)}); \draw (5.5,{-3.4*sqrt(3)}) node[above]{\tiny$\mathcal A_{N-1}^{i,z}$};
\fill[violet] (2.6,{-0.6*sqrt(3)}) circle (0.1); \draw (2.6,{-0.6*sqrt(3)}) node[below right]{\tiny\color{violet}$\mathcal A_{N-1}^{i,x|y}$};
\fill[violet] (2.2,{-1*sqrt(3)}) circle (0.1); \draw (2.2,{-1*sqrt(3)}) node[below right]{\tiny\color{violet}$\mathcal A_{N-1}^{i,x|z}$};
\fill[violet] (7.4,{-0.6*sqrt(3)}) circle (0.1); \draw (7.4,{-0.6*sqrt(3)}) node[below left]{\tiny\color{violet}$\mathcal A_{N-1}^{i,y|x}$};
\fill[violet] (4.6,{-3.4*sqrt(3)}) circle (0.1); \draw (4.6,{-3.4*sqrt(3)}) node[below]{\tiny\color{violet}$\mathcal A_{N-1}^{i,z|x}$};

\draw (-0.2,0) node[left]{$x$}; \draw (10.2,0) node[right]{$y$}; \draw (5,{-5*sqrt(3)-0.2}) node[below]{$z$};
\draw (5,0.2) node[above]{\color{blue} $\mathcal B_{N-1}^{i,y,-}$};
\draw (2.5,{-2.5*sqrt(3)}) node[below left]{\color{blue} $\mathcal B_{N-1}^{i,z,-}$};
\draw (7.35,{-2.45*sqrt(3)}) node[below right]{\color{teal} $\mathcal A_{N-1}^{i,-}$};
\draw (5,{-5/sqrt(3)}) node{$\mathcal A_{N-1}^{i,\rm{bulk}}$};
\end{tikzpicture}\hspace{5mm}\begin{tikzpicture}[scale=1.2]
\draw[white] (5,0)--(5,-3.75); 
\fill[blue!30!white] (3.2,0)--(2.7,{-0.5*sqrt(3)})--(7.3,{-0.5*sqrt(3)})--(6.8,0)--cycle;

\draw[] (3.2,0)--(2.7,{-0.5*sqrt(3)})--(7.3,{-0.5*sqrt(3)})--(6.8,0)--cycle;

\draw[very thick] (3.2,0)--(2.7,{-0.5*sqrt(3)}); \draw (2.95,{-0.25*sqrt(3)+0.2}) node[left]{$\mathcal C_{N-1}^{i,x}$};
\draw[very thick] (2.7,{-0.5*sqrt(3)})--(7.3,{-0.5*sqrt(3)}); \draw (5,{-0.5*sqrt(3)-0.1}) node[below]{$\mathcal B_{N-1}^{i,x,y}$};
\draw[very thick] (7.3,{-0.5*sqrt(3)})--(6.8,0); \draw (7.05,{-0.25*sqrt(3)+0.2}) node[right]{$\mathcal C_{N-1}^{i,y}$};
\fill[violet] (2.7,{-0.5*sqrt(3)}) circle (0.1); \draw (2.7,{-0.5*sqrt(3)-0.1}) node[below]{\color{violet} $\mathcal B_{N-1}^{i,x}$};
\fill[violet] (7.3,{-0.5*sqrt(3)}) circle (0.1); \draw (7.3,{-0.5*sqrt(3)-0.1}) node[below]{\color{violet} $\mathcal B_{N-1}^{i,y}$};

\draw (5,0.2) node[above]{\color{blue} $\mathcal B_{N-1}^{i,y,-}$};
\draw (5,{-0.25*sqrt(3)}) node{$\mathcal B_{N-1}^{i,y,\rm{bulk}}$};
\end{tikzpicture}\caption{\label{fig5.1}Decomposition of the set $\Omega^{i,-}_{N-1}$. The
negligible part $\mathcal{A}^{i,-}_{N-1}$, which corresponds to $\mathcal{G}^{c}_{N}$
in the $N$-particle space, is decomposed as (left) $\mathcal{A}^{i,-}_{N-1}=\mathcal{A}^{i,{\rm bulk}}_{N-1}\cup\mathcal{A}^{i,{\rm edge}}_{N-1}$
with
\[
\mathcal{A}^{i,{\rm edge}}_{N-1}=\mathcal{A}^{i,x}_{N-1}\cup\bigcup_{y\in S_{\star}\setminus\{x\}}\left(\mathcal{A}^{i,y}_{N-1}\cup\mathcal{A}^{i,x\mid y}_{N-1}\cup\mathcal{A}^{i,y\mid x}_{N-1}\cup\mathcal{A}^{i,x,y}_{N-1}\right).
\]
The dominant part $\mathcal{B}^{i,y,-}_{N-1}$, $y\in S_{\star}\setminus\{x\}$,
which corresponds to $\mathcal{J}^{x,y}_{N}$ in the $N$-particle
space, is decomposed as (right) $\mathcal{B}^{i,y,-}_{N-1}=\mathcal{B}^{i,y,{\rm bulk}}_{N-1}\cup\mathcal{B}^{i,y,{\rm edge}}_{N-1}$
with
\[
\mathcal{B}^{i,y,{\rm edge}}_{N-1}=\mathcal{B}^{i,x}_{N-1}\cup\mathcal{B}^{i,y}_{N-1}\cup\mathcal{B}^{i,x,y}_{N-1}\cup\mathcal{C}^{i,x}_{N-1}\cup\mathcal{C}^{i,y}_{N-1}.
\]
}
\end{figure}

First, we show that the bulk part is negligible. For each $\mathcal{E}\subset\Omega_{N}$
and $F\in\mathbb{R}^{\Omega_{N}}$, define
\[
\mathcal{D}_{N}(F;\mathcal{E}):=\frac{1}{2}\sum_{\bm{\eta}\in\mathcal{E}}\sum_{i\in\{1,2\}}\sum_{u,w\in S}\nu_{N}(\bm{\eta})\bm{g}_{i}(\bm{\eta}(u))r^{*}_{i}(u,w)\left[F(\bm{\eta}-\delta^{i}_{u}+\delta^{i}_{w})-F(\bm{\eta})\right]^{2}.
\]
Clearly, $\mathcal{D}_{N}(F;\Omega_{N})=\mathcal{D}_{N}(F)$.
\begin{lem}
\label{lem:neg-1}For each $x\in S_{\star}$, we have 
\[
\mathcal{D}_{N}(\mathbb{W}_{x};(\mathcal{G}^{x}_{N})^{c})=o(N^{-1-\alpha}).
\]
\end{lem}

\begin{proof}
Fix $x\in S_{\star}$. The left-hand side can be explicitly written
as 
\begin{equation}
\mathcal{D}_{N}(\mathbb{W}_{x};(\mathcal{G}^{x}_{N})^{c})=\frac{1}{2}\sum_{i\in\{1,2\}}\sum_{\bm{\eta}\in(\mathcal{G}^{x}_{N})^{c}}\sum_{u,w\in S}\nu_{N}(\bm{\eta})\bm{g}_{i}(\bm{\eta}(u))r^{*}_{i}(u,w)\left[\mathbb{W}_{x}(\bm{\eta}-\delta^{i}_{u}+\delta^{i}_{w})-\mathbb{W}_{x}(\bm{\eta})\right]^{2}.\label{eq:neg-1-1}
\end{equation}
First, each configuration $\bm{\eta}\in(\mathcal{G}^{x}_{N})^{c}$
satisfies $|\bm{\eta}(x)|\leq N(1-\epsilon)$. Moreover, for any $\bm{\eta}\in\Omega_{N}$
that satisfies $|\bm{\eta}(x)|<N\epsilon$, for all $y\in S_{\star}\setminus\{x\}$,
$i\in\{1,2\}$, and $u,w\in S$ (cf. \eqref{eq:Xxyi-def}),
\[
\mathcal{X}_{x,y}(\bm{\eta}-\delta^{i}_{u}+\delta^{i}_{w}),\mathcal{X}_{x,y}(\bm{\eta})\leq3\epsilon.
\]
Thus, by the definitions of $\gamma_{\epsilon}$ and $\mathbb{W}_{x}$
(recall \eqref{eq:Wx-def}), one can notice that
\[
\mathbb{W}_{x}(\bm{\eta}-\delta^{i}_{u}+\delta^{i}_{w})=\mathbb{W}_{x}(\bm{\eta})=0.
\]
Therefore, we can restrict the summation in \eqref{eq:neg-1-1} to
the configurations $\bm{\eta}\in(\mathcal{G}^{x}_{N})^{c}$ that satisfy
$N\epsilon\leq|\bm{\eta}(x)|\leq N(1-\epsilon)$. By \eqref{eq:Lip2}
and the boundedness of $\bm{g}_{i}$ and $r^{*}_{i}$, the right-hand
side of \eqref{eq:neg-1-1} can be bounded above by 
\begin{equation}
\frac{c}{N^{2}}\sum_{\substack{\bm{\eta}\in(\mathcal{G}^{x}_{N})^{c}:\\
N\epsilon\leq|\bm{\eta}(x)|\leq N(1-\epsilon)
}
}\nu_{N}(\bm{\eta}).\label{eq:neg-1-2}
\end{equation}
Note that $\bm{\eta}\in(\mathcal{G}^{x}_{N})^{c}$ and $N\epsilon\le|\bm{\eta}(x)|\le N(1-\epsilon)$
implies that $|\bm{\eta}(y)|\le N(1-\epsilon)$ for all $y\in S$,
and 
\[
|\bm{\eta}(y)|+|\bm{\eta}(z)|<N-\pi_{N},\quad\forall y\neq z\in S,\qquad\text{or}\qquad|\bm{\eta}(z)|\ge N\epsilon-\pi_{N},\quad\exists z\in S\setminus S_{\star}.
\]
Via Lemmas \ref{lem:s1}, \ref{lem:s3}, and \ref{lem:s5}, this implies
that
\begin{align*}
\sum_{\substack{\bm{\eta}\in(\mathcal{G}^{x}_{N})^{c}:\\
N\epsilon\leq|\bm{\eta}(x)|\leq N(1-\epsilon)
}
}\nu_{N}(\bm{\eta}) & \le\nu_{N}(\Xi^{N\epsilon-1,\pi_{N}}_{N})+\sum_{z\in S\setminus S_{\star}}\nu_{N}(\Omega^{z,N(1-\epsilon)+\pi_{N}}_{N})\\
 & \le\frac{c}{(N\epsilon)^{\alpha-1}(\pi_{N}+1)^{\alpha-1}}+c\upsilon^{N\epsilon-\pi_{N}}\ll\frac{1}{N^{\alpha-1}},
\end{align*}
where the final asymptotic holds since $\pi_{N}\gg1$ and $N\epsilon-\pi_{N}\asymp N$.
Combining this with \eqref{eq:neg-1-2} concludes the proof of the
lemma.
\end{proof}

\begin{lem}
\label{lem:neg-2}For each $x\in S_{\star}$,
\begin{equation}
\langle\!\langle F_{N},-\mathcal{L}^{*}_{N}\mathbb{W}_{x}\rangle\!\rangle_{\mathcal{A}^{{\rm bulk}}_{N-1}}=o(\theta^{-1}_{N}).\label{eq:neg-2-1}
\end{equation}
 
\end{lem}

\begin{proof}
The idea of proof is similar to that of Lemma \ref{lem:sector}. First,
the left-hand side of \eqref{eq:neg-2-1} can be expressed as 
\[
\sum_{i\in\{1,2\}}\sum_{\bm{\xi}\in\mathcal{A}^{i,{\rm bulk}}_{N-1}}\sum_{u,w\in S}\frac{N^{\alpha}}{Z_{N}\binom{N}{A}}\frac{\bm{m}^{\bm{\xi}}}{\bm{g!}(\bm{\xi})}m_{i}(u)r^{*}_{i}(u,w)F_{N}(\bm{\xi}+\delta^{i}_{u})\left\{ \mathbb{W}_{x}(\bm{\xi}+\delta^{i}_{u})-\mathbb{W}_{x}(\bm{\xi}+\delta^{i}_{w})\right\} .
\]
By the stationarity of $m_{i}$, we have 
\[
\sum_{u,w\in S}m_{i}(u)r^{*}_{i}(u,w)\left\{ \mathbb{W}_{x}(\bm{\xi}+\delta^{i}_{u})-\mathbb{W}_{x}(\bm{\xi}+\delta^{i}_{w})\right\} =0.
\]
Thus, the left hand side of \eqref{eq:neg-2-1} is equal to 
\begin{equation}
\sum_{i\in\{1,2\}}\sum_{\bm{\xi}\in\mathcal{A}^{i,{\rm bulk}}_{N-1}}\sum_{u,w\in S}\frac{N^{\alpha}}{Z_{N}\binom{N}{A}}\frac{\bm{m}^{\bm{\xi}}}{\bm{g!}(\bm{\xi})}m_{i}(u)r^{*}_{i}(u,w)\left\{ F_{N}(\bm{\xi}+\delta^{i}_{u})-\overline{F}^{i}_{N}(\bm{\xi})\right\} \left\{ \mathbb{W}_{x}(\bm{\xi}+\delta^{i}_{u})-\mathbb{W}_{x}(\bm{\xi}+\delta^{i}_{w})\right\} ,\label{eq:neg-2-2}
\end{equation}
where $\overline{F}^{i}_{N}:\Omega^{i,-}_{N-1}\to\mathbb{R}$ is defined
as in \eqref{eq:Gibar-def}. By the Cauchy--Schwarz inequality, \eqref{eq:neg-2-2}
is bounded above by 
\begin{align*}
 & \left(\sum_{i\in\{1,2\}}\sum_{\bm{\xi}\in\mathcal{A}^{i,{\rm bulk}}_{N-1}}\frac{N^{\alpha}}{Z_{N}\binom{N}{A}}\frac{\bm{m}^{\bm{\xi}}}{\bm{g!}(\bm{\xi})}\sum_{u,w\in S}m_{i}(u)r^{*}_{i}(u,w)\left\{ F_{N}(\bm{\xi}+\delta^{i}_{u})-\overline{F}^{i}_{N}(\bm{\xi})\right\} ^{2}\right)^{\frac{1}{2}}\\
 & \times\left(\sum_{i\in\{1,2\}}\sum_{\bm{\xi}\in\mathcal{A}^{i,{\rm bulk}}_{N-1}}\frac{N^{\alpha}}{Z_{N}\binom{N}{A}}\frac{\bm{m}^{\bm{\xi}}}{\bm{g!}(\bm{\xi})}\sum_{u,w\in S}m_{i}(u)r^{*}_{i}(u,w)\left\{ \mathbb{W}_{x}(\bm{\xi}+\delta^{i}_{u})-\mathbb{W}_{x}(\bm{\xi}+\delta^{i}_{w})\right\} ^{2}\right)^{\frac{1}{2}}.
\end{align*}
As done in \eqref{eq:sector-3}, we may divide the square of difference
between $F_{N}$ and its average into the sum of square of differences
between individual values. Then, by the inclusion relation $\mathcal{A}^{i,{\rm bulk}}_{N-1}\subset\Omega^{i,-}_{N-1}$
for each $i\in\{1,2\}$, we obtain that the above term is bounded
above by 
\[
c\sqrt{\mathcal{D}_{N}(F_{N})}\times\sqrt{\mathcal{D}_{N}(\mathbb{W}_{x};(\mathcal{G}^{x}_{N})^{c})},
\]
which is $o(\theta^{-1}_{N})$ as desired by Lemmas \ref{lem:PE1}
and \ref{lem:neg-1}.
\end{proof}

Now, we handle the edge part $\langle\!\langle F_{N},-\mathcal{L}^{*}_{N}\mathbb{W}_{x}\rangle\!\rangle_{\mathcal{A}^{{\rm edge}}_{N-1}}$.
To this end, we separate the collection $\mathcal{A}^{{\rm edge}}_{N-1}=\mathcal{A}^{1,{\rm edge}}_{N-1}\times\mathcal{A}^{2,{\rm edge}}_{N-1}$
with certain admissible departure collections. For $y\in S_{\star}\setminus\{x\}$,
let
\begin{align*}
\mathcal{A}^{i,x}_{N-1} & :=\left\{ \bm{\xi}\in\mathcal{A}^{i,-}_{N-1}:|\bm{\xi}(x)|=N(1-2\epsilon)-1,\quad|\bm{\xi}(x)|+|\bm{\xi}(z)|<N-\pi_{N}-1,\enspace\forall z\in S_{\star}\setminus\{x\}\right\} ,\\
\mathcal{A}^{i,y}_{N-1} & :=\left\{ \bm{\xi}\in\mathcal{A}^{i,-}_{N-1}:|\bm{\xi}(y)|=N(1-2\epsilon)-1,\quad|\bm{\xi}(x)|+|\bm{\xi}(y)|<N-\pi_{N}-1\right\} ,\\
\mathcal{A}^{i,x\mid y}_{N-1} & :=\left\{ \bm{\xi}\in\mathcal{A}^{i,-}_{N-1}:|\bm{\xi}(x)|=N(1-2\epsilon)-1,\quad|\bm{\xi}(x)|+|\bm{\xi}(y)|=N-\pi_{N}-1\right\} ,\\
\mathcal{A}^{i,y\mid x}_{N-1} & :=\left\{ \bm{\xi}\in\mathcal{A}^{i,-}_{N-1}:|\bm{\xi}(y)|=N(1-2\epsilon)-1,\quad|\bm{\xi}(x)|+|\bm{\xi}(y)|=N-\pi_{N}-1\right\} ,\\
\mathcal{A}^{i,x,y}_{N-1} & :=\left\{ \bm{\xi}\in\mathcal{A}^{i,-}_{N-1}:|\bm{\xi}(x)|,|\bm{\xi}(y)|<N(1-2\epsilon)-1,\quad|\bm{\xi}(x)|+|\bm{\xi}(y)|=N-\pi_{N}-1\right\} .
\end{align*}
Since $\epsilon$ is sufficiently small, one can notice that all collections
above are disjoint and form a partition of $\mathcal{A}^{i,{\rm edge}}_{N-1}$.
For each collection, the corresponding admissible departure site is
given by 
\[
\Xi^{i}_{x}(\bm{\xi})=\begin{cases}
S\setminus\{x\} & \text{if}\quad\bm{\xi}\in\mathcal{A}^{i,x}_{N-1},\\
S\setminus\{y\} & \text{if}\quad\bm{\xi}\in\mathcal{A}^{i,y}_{N-1},\enspace y\in S_{\star}\setminus\{x\},\\
S\setminus\{x,y\} & \text{if}\quad\bm{\xi}\in\mathcal{A}^{i,x\mid y}_{N-1}\cup\mathcal{A}^{i,y\mid x}_{N-1}\cup\mathcal{A}^{i,x,y}_{N-1},\enspace y\in S_{\star}\setminus\{x\}.
\end{cases}
\]
In this regard, we may decompose the edge part as (cf. \eqref{eq:inner-prod-sim})
\begin{equation}
\begin{aligned}\langle\!\langle F_{N},-\mathcal{L}^{*}_{N}\mathbb{W}_{x}\rangle\!\rangle_{\mathcal{A}^{{\rm edge}}_{N-1}} & =\langle\!\langle F_{N},-\mathcal{L}^{*}_{N}\mathbb{W}_{x}\rangle\!\rangle_{\mathcal{A}^{x}_{N-1}}\\
 & +\sum_{y\in S_{\star}\setminus\{x\}}\langle\!\langle F_{N},-\mathcal{L}^{*}_{N}\mathbb{W}_{x}\rangle\!\rangle_{\mathcal{A}^{y}_{N-1}}+\sum_{y\in S_{\star}\setminus\{x\}}\langle\!\langle F_{N},-\mathcal{L}^{*}_{N}\mathbb{W}_{x}\rangle\!\rangle_{\mathcal{A}^{x\mid y}_{N-1}}\\
 & +\sum_{y\in S_{\star}\setminus\{x\}}\langle\!\langle F_{N},-\mathcal{L}^{*}_{N}\mathbb{W}_{x}\rangle\!\rangle_{\mathcal{A}^{y\mid x}_{N-1}}+\sum_{y\in S_{\star}\setminus\{x\}}\langle\!\langle F_{N},-\mathcal{L}^{*}_{N}\mathbb{W}_{x}\rangle\!\rangle_{\mathcal{A}^{x,y}_{N-1}},
\end{aligned}
\label{eq:edge-dec}
\end{equation}
where $\mathcal{A}^{x}_{N-1}:=\mathcal{A}^{1,x}_{N-1}\times\mathcal{A}^{2,x}_{N-1}$,
$\mathcal{A}^{y}_{N-1}:=\mathcal{A}^{1,y}_{N-1}\times\mathcal{A}^{2,y}_{N-1}$,
and so on.
\begin{lem}
\label{lem:neg-3}We have 
\[
\langle\!\langle F_{N},-\mathcal{L}^{*}_{N}\mathbb{W}_{x}\rangle\!\rangle_{\mathcal{A}^{{\rm edge}}_{N-1}}=o(\theta^{-1}_{N}).
\]
\end{lem}

\begin{proof}
We prove the lemma according to the decomposition \eqref{eq:edge-dec}.
First, if $\bm{\xi}\in\mathcal{A}^{i,x}_{N-1}$ or $\mathcal{A}^{i,x|y}_{N-1}$
for each $i\in\{1,2\}$, $y\in S_{\star}\setminus\{x\}$, then $|\bm{\xi}(x)|=N(1-2\epsilon)-1$
implies $\bm{\xi}+\delta^{i}_{u},\bm{\xi}+\delta^{i}_{w}\in\overline{\mathcal{D}}^{x}_{N}$
(cf. \eqref{eq:DJ-cl}), thus $\mathbb{W}_{x}(\bm{\xi}+\delta^{i}_{u})=\mathbb{W}_{x}(\bm{\xi}+\delta^{i}_{w})=1$
by Lemma \ref{lem:Wx-prop}. Therefore, by \eqref{eq:inner-prod-sim},
\[
\langle\!\langle F_{N},-\mathcal{L}^{*}_{N}\mathbb{W}_{x}\rangle\!\rangle_{\mathcal{A}^{x}_{N-1}}=\sum_{y\in S_{\star}\setminus\{x\}}\langle\!\langle F_{N},-\mathcal{L}^{*}_{N}\mathbb{W}_{x}\rangle\!\rangle_{\mathcal{A}^{x\mid y}_{N-1}}=0.
\]
Similarly, if $\bm{\xi}\in\mathcal{A}^{i,y}_{N-1}$ or $\mathcal{A}^{i,y|x}_{N-1}$
for $y\in S_{\star}\setminus\{x\}$, then $|\bm{\xi}(y)|=N(1-2\epsilon)-1$
thus $\bm{\xi}+\delta^{i}_{u},\bm{\xi}+\delta^{i}_{w}\in\overline{\mathcal{D}}^{y}_{N}$,
and $\mathbb{W}_{x}(\bm{\xi}+\delta^{i}_{w})=\mathbb{W}_{x}(\bm{\xi}+\delta^{i}_{u})=0$.
Therefore,
\[
\sum_{y\in S_{\star}\setminus\{x\}}\langle\!\langle F_{N},-\mathcal{L}^{*}_{N}\mathbb{W}_{x}\rangle\!\rangle_{\mathcal{A}^{y}_{N-1}}=\sum_{y\in S_{\star}\setminus\{x\}}\langle\!\langle F_{N},-\mathcal{L}^{*}_{N}\mathbb{W}_{x}\rangle\!\rangle_{\mathcal{A}^{y\mid x}_{N-1}}=0.
\]
It thus suffices to prove that, for each $y\in S_{\star}\setminus\{x\}$,
\[
\langle\!\langle F_{N},-\mathcal{L}^{*}_{N}\mathbb{W}_{x}\rangle\!\rangle_{\mathcal{A}^{x,y}_{N-1}}=o(\theta^{-1}_{N}).
\]
To this end, consider a configuration $\bm{\xi}\in\mathcal{A}^{i,x,y}_{N-1}$
and a jump $\bm{\xi}+\delta^{i}_{u}\to\bm{\xi}+\delta^{i}_{w}$ for
$u\in\Xi^{i}_{x}(\bm{\xi})$ and $w\in S$. By the Lipschitz property
of $\mathbb{W}_{x}$ at \eqref{eq:Lip2} and the boundedness of $m_{i}$,
$r^{*}_{i}$, and $F_{N}$, the absolute value of the left-hand side
of the previous display is bounded by
\begin{equation}
\frac{c}{N}\nu_{N}\left(\bm{\eta}:|\bm{\eta}(x)|,|\bm{\eta}(y)|<N(1-2\epsilon)-1,\enspace|\bm{\eta}(x)|+|\bm{\eta}(y)|=N-\pi_{N}-1\right)\le\frac{c'}{N\pi^{\alpha}_{N}(N\epsilon)^{\alpha-1}}\ll\frac{1}{N^{1+\alpha}},\label{eq:neg-3-1}
\end{equation}
where the inequality follows from Lemma \ref{lem:s4}, and the asymptotic
follows from \eqref{eq:piN} that $\pi^{\alpha}_{N}\gg N$. This concludes
the proof.
\end{proof}

\begin{proof}[Proof of Lemma \ref{lem:res-2}]
 From \eqref{eq:res-2-1} and \eqref{eq:res-2-2}, we have
\[
\langle F_{N},-\mathcal{L}^{*}_{N}\mathbb{W}_{x}\rangle_{\nu_{N},(\mathcal{G}^{x}_{N})^{c}}=\langle\!\langle F_{N},-\mathcal{L}^{*}_{N}\mathbb{W}_{x}\rangle\!\rangle_{\mathcal{A}^{{\rm bulk}}_{N-1}}+\langle\!\langle F_{N},-\mathcal{L}^{*}_{N}\mathbb{W}_{x}\rangle\!\rangle_{\mathcal{A}^{{\rm edge}}_{N-1}}.
\]
Lemmas \ref{lem:neg-2} and \ref{lem:neg-3} imply that the right-hand
side equals $o(\theta^{-1}_{N})$ as desired.
\end{proof}

\subsection{\label{sec5.3}Dominant Part}

According to the analysis conducted in the previous subsection, we
only need to estimate the term $\langle F_{N},-\mathcal{L}^{*}_{N}\mathbb{W}_{x}\rangle_{\nu_{N},\mathcal{G}^{x}_{N}}$,
which would give the leading order term of $\langle F_{N},-\mathcal{L}^{*}_{N}\mathbb{W}_{x}\rangle_{\nu_{N}}$.
\begin{lem}
\label{lem:res-3}For each $x\in S_{\star}$,
\begin{align*}
 & \langle F_{N},-\mathcal{L}^{*}_{N}\mathbb{W}_{x}\rangle_{\nu_{N},\mathcal{G}^{x}_{N}}\\
 & =N^{-1-\alpha}\sum_{y\in S_{\star}\setminus\{x\}}\frac{1}{\kappa_{\star}\Gamma(\alpha)\mathcal{I}_{\alpha}}\left(\frac{\rho}{c_{1}(x,y)}+\frac{1-\rho}{c_{2}(x,y)}\right)^{-1}\left[f_{N}(x)-f_{N}(y)\right]+o(\theta^{-1}_{N})+\theta^{-1}_{N}o_{\epsilon}(1),
\end{align*}
where $f_{N}$ is defined at \eqref{eq:fN-def}.
\end{lem}

First, we show that the inner product equals zero on $\mathcal{D}^{y}_{N}$
for all $y\in S_{\star}$.
\begin{lem}
\label{lem:dom-1}For each $x\in S_{\star}$ and for all $y\in S_{\star}$,
we have 
\[
\langle F_{N},-\mathcal{L}^{*}_{N}\mathbb{W}_{x}\rangle_{\nu_{N},\mathcal{D}^{y}_{N}}=0.
\]
\end{lem}

\begin{proof}
It suffices to show that $\mathbb{W}_{x}(\bm{\eta})=\mathbb{W}_{x}(\bm{\eta}-\delta^{i}_{u}+\delta^{i}_{w})$
for all $\bm{\eta}\in\mathcal{D}^{y}_{N}$, $x,y\in S_{\star}$, $i\in\{1,2\}$,
and $u,w\in S$. If $y=x$ then both $\bm{\eta}$ and $\bm{\eta}-\delta^{i}_{u}+\delta^{i}_{w}$
lie in $\overline{\mathcal{D}}^{x}_{N}$, thus \eqref{eq:Wx-prop}
implies $\mathbb{W}_{x}(\bm{\eta})=\mathbb{W}_{x}(\bm{\eta}-\delta^{i}_{u}+\delta^{i}_{w})=1$.
On the other hand, if $y\ne x$, then both lie in $\overline{\mathcal{D}}^{y}_{N}$
thus \eqref{eq:Wx-prop} again implies $\mathbb{W}_{x}(\bm{\eta})=\mathbb{W}_{x}(\bm{\eta}-\delta^{i}_{u}+\delta^{i}_{w})=0$,
as desired.
\end{proof}

Thus, according to \eqref{eq:G-def}, only the $\{\mathcal{J}^{x,y}_{N}\}_{y\in S_{\star}\setminus\{x\}}$
part remains. Since the collection is disjoint, we can estimate each
term $\langle F_{N},-\mathcal{L}^{*}_{N}\mathbb{W}_{x}\rangle_{\nu_{N},\mathcal{J}^{x,y}_{N}}$
separately for each $y\in S_{\star}\setminus\{x\}$. As in the previous
subsection, we will decompose $\mathcal{J}^{x,y}_{N}$ by separating
the admissible departure sites. Since it is impossible that $\bm{\eta}\in\mathcal{J}^{x,y}_{N}$
jumps into $\mathcal{J}^{x,y'}_{N}$ for $y\neq y'$ by a one-particle
movement, we can analyze the sets separately. 

As in \eqref{eq:Xix-def}, define
\[
\Xi^{i}_{x,y}(\bm{\xi}):=\left\{ u\in S:\bm{\xi}+\delta^{i}_{u}\in\mathcal{J}^{x,y}_{N}\right\} .
\]
Define a collection $\mathcal{B}^{i,y,-}_{N-1}$ as 
\[
\mathcal{B}^{i,y,-}_{N-1}:=\left\{ \bm{\xi}\in\Omega^{i,-}_{N-1}:\Xi^{i}_{x,y}(\bm{\xi})\neq\emptyset\right\} =\bigcup_{u\in S}\left\{ \bm{\xi}\in\Omega^{i,-}_{N-1}:\bm{\xi}+\delta^{i}_{u}\in\mathcal{J}^{x,y}_{N}\right\} .
\]
Decompose this set as $\mathcal{B}^{i,y,-}_{N-1}=\mathcal{B}^{i,y,{\rm bulk}}_{N-1}\cup\mathcal{B}^{i,y,{\rm edge}}_{N-1}$
where
\begin{equation}
\begin{aligned}\mathcal{B}^{i,y,{\rm bulk}}_{N-1} & :=\left\{ \bm{\xi}\in\Omega^{i,-}_{N-1}:\Xi^{i}_{x,y}(\bm{\xi})=S\right\} \\
 & =\left\{ \bm{\xi}\in\Omega^{i,-}_{N-1}:|\bm{\xi}(x)|+|\bm{\xi}(y)|\geq N-\pi_{N},\quad|\bm{\xi}(x)|,|\bm{\xi}(y)|<N(1-2\epsilon)-1\right\} .
\end{aligned}
\label{eq:Bbulk-def}
\end{equation}
In this way, we may write
\begin{equation}
\langle F_{N},-\mathcal{L}^{*}_{N}\mathbb{W}_{x}\rangle_{\nu_{N},\mathcal{J}^{x,y}_{N}}=\langle\!\langle F_{N},-\mathcal{L}^{*}_{N}\mathbb{W}_{x}\rangle\!\rangle_{\mathcal{B}^{y,{\rm bulk}}_{N-1}}+\langle\!\langle F_{N},-\mathcal{L}^{*}_{N}\mathbb{W}_{x}\rangle\!\rangle_{\mathcal{B}^{y,{\rm edge}}_{N-1}},\label{eq:inn-prod-xy-dec}
\end{equation}
where the double-angle brackets at the right-hand side are defined
as in \eqref{eq:inner-prod-sim}, but with $\Xi^{i}_{x}$ replaced
by $\Xi^{i}_{x,y}$.

Next, we decompose the collection $\mathcal{B}^{i,y,{\rm edge}}_{N-1}$
as
\[
\mathcal{B}^{i,y,{\rm edge}}_{N-1}=\mathcal{B}^{i,x}_{N-1}\cup\mathcal{B}^{i,y}_{N-1}\cup\mathcal{B}^{i,x,y}_{N-1}\cup\mathcal{C}^{i,x}_{N-1}\cup\mathcal{C}^{i,y}_{N-1}\cup\mathcal{C}^{i,x,y}_{N-1},
\]
where 
\begin{align*}
\mathcal{B}^{i,x}_{N-1} & :=\left\{ \bm{\xi}\in\mathcal{B}^{i,y,-}_{N-1}:|\bm{\xi}(x)|+|\bm{\xi}(y)|=N-\pi_{N}-1,\enspace|\bm{\xi}(x)|=N(1-2\epsilon)-1,\enspace|\bm{\xi}(y)|<N(1-2\epsilon)-1\right\} ,\\
\mathcal{B}^{i,y}_{N-1} & :=\left\{ \bm{\xi}\in\mathcal{B}^{i,y,-}_{N-1}:|\bm{\xi}(x)|+|\bm{\xi}(y)|=N-\pi_{N}-1,\enspace|\bm{\xi}(x)|<N(1-2\epsilon)-1,\enspace|\bm{\xi}(y)|=N(1-2\epsilon)-1\right\} ,\\
\mathcal{B}^{i,x,y}_{N-1} & :=\left\{ \bm{\xi}\in\mathcal{B}^{i,y,-}_{N-1}:|\bm{\xi}(x)|+|\bm{\xi}(y)|=N-\pi_{N}-1,\enspace|\bm{\xi}(x)|,|\bm{\xi}(y)|<N(1-2\epsilon)-1\right\} ,
\end{align*}
and 
\begin{align*}
\mathcal{C}^{i,x}_{N-1} & :=\left\{ \bm{\xi}\in\mathcal{B}^{i,y,-}_{N-1}:|\bm{\xi}(x)|+|\bm{\xi}(y)|\geq N-\pi_{N},\enspace|\bm{\xi}(x)|=N(1-2\epsilon)-1,\enspace|\bm{\xi}(y)|<N(1-2\epsilon)-1\right\} ,\\
\mathcal{C}^{i,y}_{N-1} & :=\left\{ \bm{\xi}\in\mathcal{B}^{i,y,-}_{N-1}:|\bm{\xi}(x)|+|\bm{\xi}(y)|\geq N-\pi_{N},\enspace|\bm{\xi}(x)|<N(1-2\epsilon)-1,\enspace|\bm{\xi}(y)|=N(1-2\epsilon)-1\right\} ,\\
\mathcal{C}^{i,x,y}_{N-1} & :=\left\{ \bm{\xi}\in\mathcal{B}^{i,y,-}_{N-1}:|\bm{\xi}(x)|+|\bm{\xi}(y)|\geq N-\pi_{N},\enspace|\bm{\xi}(x)|,|\bm{\xi}(y)|=N(1-2\epsilon)-1\right\} .
\end{align*}
Since $\epsilon$ is sufficiently small, we have $\mathcal{C}^{i,x,y}_{N-1}=\emptyset$.
For each set, the corresponding admissible departure sites are given
by
\[
\Xi^{i}_{x,y}(\bm{\xi})=\begin{cases}
\{y\} & \textnormal{if}\quad\bm{\xi}\in\mathcal{B}^{i,x}_{N-1},\\
\{x\} & \textnormal{if}\quad\bm{\xi}\in\mathcal{B}^{i,y}_{N-1},\\
\{x,y\} & \textnormal{if}\quad\bm{\xi}\in\mathcal{B}^{i,x,y}_{N-1},\\
S\setminus\{x\} & \textnormal{if}\quad\bm{\xi}\in\mathcal{C}^{i,x}_{N-1},\\
S\setminus\{y\} & \textnormal{if}\quad\bm{\xi}\in\mathcal{C}^{i,y}_{N-1}.
\end{cases}
\]
According to the decomposition,
\begin{equation}
\begin{aligned}\langle\!\langle F_{N},-\mathcal{L}^{*}_{N}\mathbb{W}_{x}\rangle\!\rangle_{\mathcal{B}^{y,{\rm edge}}_{N-1}} & =\langle\!\langle F_{N},-\mathcal{L}^{*}_{N}\mathbb{W}_{x}\rangle\!\rangle_{\mathcal{B}^{x}_{N-1}}+\langle\!\langle F_{N},-\mathcal{L}^{*}_{N}\mathbb{W}_{x}\rangle\!\rangle_{\mathcal{B}^{y}_{N-1}}\\
 & +\langle\!\langle F_{N},-\mathcal{L}^{*}_{N}\mathbb{W}_{x}\rangle\!\rangle_{\mathcal{B}^{x,y}_{N-1}}+\langle\!\langle F_{N},-\mathcal{L}^{*}_{N}\mathbb{W}_{x}\rangle\!\rangle_{\mathcal{C}^{x}_{N-1}}+\langle\!\langle F_{N},-\mathcal{L}^{*}_{N}\mathbb{W}_{x}\rangle\!\rangle_{\mathcal{C}^{y}_{N-1}}.
\end{aligned}
\label{eq:dom-edge-dec}
\end{equation}

\begin{lem}
\label{lem:dom-2}We have 
\[
\langle\!\langle F_{N},-\mathcal{L}^{*}_{N}\mathbb{W}_{x}\rangle\!\rangle_{\mathcal{B}^{y,{\rm edge}}_{N-1}}=o(N^{-(\alpha+1)}).
\]
 
\end{lem}

\begin{proof}
The proof is quite similar to that of Lemma \ref{lem:neg-3}. For
$\bm{\xi}\in\mathcal{B}^{i,x}_{N-1}\cup\mathcal{C}^{i,x}_{N-1}$ we
have $\mathbb{W}_{x}(\bm{\xi}+\delta^{i}_{w})=\mathbb{W}_{x}(\bm{\xi}+\delta^{i}_{u})=1$,
whereas for $\bm{\xi}\in\mathcal{B}^{i,y}_{N-1}\cup\mathcal{C}^{i,y}_{N-1}$
we have $\mathbb{W}_{x}(\bm{\xi}+\delta^{i}_{w})=\mathbb{W}_{x}(\bm{\xi}+\delta^{i}_{u})=0$.
Only remaining term to consider is 
\begin{equation}
\langle\!\langle F_{N},-\mathcal{L}^{*}_{N}\mathbb{W}_{x}\rangle\!\rangle_{\mathcal{B}^{x,y}_{N-1}},\label{eq:dom-2-1}
\end{equation}
which can be written as 
\[
-\sum_{i\in\{1,2\}}\sum_{\bm{\xi}\in\mathcal{B}^{i,x,y}_{N-1}}\sum_{u\in\Xi^{i}_{x,y}(\bm{\xi})}\sum_{w\in S}\frac{N^{\alpha}}{Z_{N}\binom{N}{A}}\frac{\bm{m}^{\bm{\xi}}}{\bm{g!}(\bm{\xi})}m_{i}(u)r^{*}_{i}(u,w)F_{N}(\bm{\xi}+\delta^{i}_{u})\left[\mathbb{W}_{x}(\bm{\xi}+\delta^{i}_{w})-\mathbb{W}_{x}(\bm{\xi}+\delta^{i}_{u})\right].
\]
Then, the absolute value of \eqref{eq:dom-2-1} is bounded above by
\[
\frac{c}{N}\sum_{i\in\{1,2\}}\sum_{\bm{\xi}\in\mathcal{B}^{i,x,y}_{N-1}}\frac{N^{\alpha}}{Z_{N}\binom{N}{A}}\frac{\bm{m}^{\bm{\xi}}}{\bm{g!}(\bm{\xi})},
\]
by the boundedness of $m_{i}$, $r^{*}_{i}$, $F_{N}$, and the Lipschitz
continuity \eqref{eq:Lip2} of $\mathbb{W}_{x}$. Thus, the same $\nu_{N}$
argument conducted at \eqref{eq:neg-3-1} guarantees that the above
term is bounded above by $cN^{-\alpha}\pi^{-\alpha}_{N}$. From \eqref{eq:piN}
we have $\pi^{\alpha}_{N}\gg N$, thus the lemma follows.
\end{proof}

Finally, we only to consider the term $\langle\!\langle F_{N},-\mathcal{L}^{*}_{N}\mathbb{W}_{x}\rangle\!\rangle_{\mathcal{B}^{y,{\rm bulk}}_{N-1}}$,
which has an explicit form:
\[
-\sum_{i\in\{1,2\}}\sum_{\bm{\xi}\in\mathcal{B}^{i,y,{\rm bulk}}_{N-1}}\sum_{u,w\in S}\frac{N^{\alpha}}{Z_{N}\binom{N}{A}}\frac{\bm{m}^{\bm{\xi}}}{\bm{g!}(\bm{\xi})}m_{i}(u)r^{*}_{i}(u,w)F_{N}(\bm{\xi}+\delta^{i}_{u})\left[\mathbb{W}_{x}(\bm{\xi}+\delta^{i}_{w})-\mathbb{W}_{x}(\bm{\xi}+\delta^{i}_{u})\right],
\]
since all elements in $S$ can be admissible departure site for the
configurations in $\mathcal{B}^{i,y,{\rm bulk}}_{N-1}$. By the stationarity
of $m_{i}$, one can rewrite the above term as 
\[
-\sum_{i\in\{1,2\}}\sum_{\bm{\xi}\in\mathcal{B}^{i,y,{\rm bulk}}_{N-1}}\sum_{u,w\in S}\frac{N^{\alpha}}{Z_{N}\binom{N}{A}}\frac{\bm{m}^{\bm{\xi}}}{\bm{g!}(\bm{\xi})}m_{i}(u)r^{*}_{i}(u,w)\left[F_{N}(\bm{\xi}+\delta^{i}_{u})-\overline{F}^{i}_{N}(\bm{\xi})\right]\left[\mathbb{W}_{x}(\bm{\xi}+\delta^{i}_{w})-\mathbb{W}_{x}(\bm{\xi}+\delta^{i}_{u})\right],
\]
as in the proof of Lemma \ref{lem:neg-2} (cf. \eqref{eq:Gibar-def}).
According to Lemmas \ref{lem:Wx-prop} and \ref{lem:Wx-incr}, we
may write
\begin{equation}
\mathbb{W}_{x}(\bm{\xi}+\delta^{i}_{w})-\mathbb{W}_{x}(\bm{\xi}+\delta^{i}_{u})=\mathcal{W}^{i}_{x,y,1}(\bm{\xi};u,w)+\mathcal{W}^{i}_{x,y,2}(\bm{\xi};u,w),\label{eq:Wx-diff-dec-1}
\end{equation}
where
\begin{equation}
\begin{aligned}\mathcal{W}^{i}_{x,y,1}(\bm{\xi};u,w) & =\frac{1}{\tau_{N}(x,y)c_{i}(x,y)}\left(\frac{\mathfrak{h}^{*}_{i,x,y}(w)-\mathfrak{h}^{*}_{i,x,y}(u)}{N}\right)\frac{|\bm{\xi}(x)|^{\alpha}|\bm{\xi}(y)|^{\alpha}}{N^{2\alpha}\mathcal{I}_{\alpha}},\\
\mathcal{W}^{i}_{x,y,2}(\bm{\xi};u,w) & =O\left(\frac{1}{N}\right)o_{\epsilon}(1)+O\left(\frac{\pi_{N}}{N^{2}}\right)+O\left(\frac{\xi_{1}(x)}{N^{2}}-\frac{|\bm{\xi}(x)|}{N^{2}}\rho_{N}\right).
\end{aligned}
\label{eq:Wx-diff-dec-2}
\end{equation}

\begin{lem}
\label{lem:dom-3}We have
\begin{align*}
-\sum_{i\in\{1,2\}}\sum_{\bm{\xi}\in\mathcal{B}^{i,y,{\rm bulk}}_{N-1}}\sum_{u,w\in S}\frac{N^{\alpha}}{Z_{N}\binom{N}{A}}\frac{\bm{m}^{\bm{\xi}}}{\bm{g!}(\bm{\xi})}m_{i}(u)r^{*}_{i}(u,w) & \left[F_{N}(\bm{\xi}+\delta^{i}_{u})-\overline{F}^{i}_{N}(\bm{\xi})\right]\mathcal{W}^{i}_{x,y,2}(\bm{\xi};u,w)\\
 & =o(N^{-1-\alpha})+O(N^{-1-\alpha})o_{\epsilon}(1).
\end{align*}
\end{lem}

\begin{proof}
By the Cauchy--Schwarz inequality, the absolute value of the left-hand
side is bounded above by 
\begin{align*}
 & \left(\sum_{i\in\{1,2\}}\sum_{\bm{\xi}\in\mathcal{B}^{i,y,{\rm bulk}}_{N-1}}\sum_{u,w\in S}\frac{N^{\alpha}}{Z_{N}\binom{N}{A}}\frac{\bm{m}^{\bm{\xi}}}{\bm{g!}(\bm{\xi})}m_{i}(u)r^{*}_{i}(u,w)\left[F_{N}(\bm{\xi}+\delta^{i}_{u})-\overline{F}^{i}_{N}(\bm{\xi})\right]^{2}\right)^{\frac{1}{2}}\\
 & \times\left(\sum_{i\in\{1,2\}}\sum_{\bm{\xi}\in\mathcal{B}^{i,y,{\rm bulk}}_{N-1}}\sum_{u,w\in S}\frac{N^{\alpha}}{Z_{N}\binom{N}{A}}\frac{\bm{m}^{\bm{\xi}}}{\bm{g!}(\bm{\xi})}m_{i}(u)r^{*}_{i}(u,w)\mathcal{W}^{i}_{x,y,2}(\bm{\xi};u,w)^{2}\right)^{\frac{1}{2}}.
\end{align*}
By Lemmas \ref{lem:PE1}, \ref{lem:PE2}, and the non-negativity of
summands, the first term has an order of $O(\theta^{-1/2}_{N})$.
Hence, it suffices to show that the second term is negligible:
\begin{equation}
\sum_{i\in\{1,2\}}\sum_{\bm{\xi}\in\mathcal{B}^{i,y,{\rm bulk}}_{N-1}}\sum_{u,w\in S}\frac{N^{\alpha}}{Z_{N}\binom{N}{A}}\frac{\bm{m}^{\bm{\xi}}}{\bm{g!}(\bm{\xi})}m_{i}(u)r^{*}_{i}(u,w)\mathcal{W}^{i}_{x,y,2}(\bm{\xi};u,w)^{2}=o(N^{-1-\alpha})+O(N^{-1-\alpha})o_{\epsilon}(1).\label{eq:dom-3-1}
\end{equation}
By Young's inequality and \eqref{eq:Wx-diff-dec-2}, it suffices to
check \eqref{eq:dom-3-1} where $\mathcal{W}^{i}_{x,y,2}(\bm{\xi};u,w)^{2}$
is replaced with
\[
\left(\frac{\xi_{1}(x)}{N^{2}}-\frac{|\bm{\xi}(x)|}{N^{2}}\rho_{N}\right)^{2}\qquad\text{or}\qquad\frac{o_{\epsilon}(1)}{N^{2}}+\frac{\pi^{2}_{N}}{N^{4}}.
\]
Since the estimation is identical for each $i\in\{1,2\}$, it suffices
to consider the case of $i=1$. First, consider the summation 
\begin{equation}
\sum_{\bm{\xi}\in\mathcal{B}^{1,y,{\rm bulk}}_{N-1}}\sum_{u,w\in S}\frac{N^{\alpha}}{Z_{N}\binom{N}{A}}\frac{\bm{m}^{\bm{\xi}}}{\bm{g!}(\bm{\xi})}m_{1}(u)r^{*}_{1}(u,w)\left(\frac{\xi_{1}(x)}{N^{2}}-\frac{|\bm{\xi}(x)|}{N^{2}}\rho_{N}\right)^{2}.\label{eq:dom-3-1.5}
\end{equation}
By \eqref{eq:Lip1}, Lemma \ref{lem:s1}, and the uniform boundedness
of the graph constants, this is bounded by
\[
cN^{\alpha-4}\sum_{\bm{\xi}\in\mathcal{B}^{1,y,{\rm bulk}}_{N-1}}\frac{1}{\binom{N}{A}}\frac{\bm{m}^{\bm{\xi}}}{\bm{g!}(\bm{\xi})}\left(\xi_{1}(x)-|\bm{\xi}(x)|\rho_{N}\right)^{2}.
\]
Recall \eqref{eq:Bbulk-def}. Let us assume that $\bm{\xi}$ has $k\in\llbracket0,\pi_{N}-1\rrbracket$
particles on $S\setminus\{x,y\}$, $p\in\llbracket2N\epsilon-k+1,N(1-2\epsilon)-2\rrbracket$
particles at $x$, and let $\bm{\sigma}$ be the corresponding restricted
configuration on $S\setminus\{x,y\}$ with $k$ particles. Then, the
last display can be written explicitly as
\begin{equation}
\frac{cN^{\alpha-4}}{\binom{N}{A}}\sum^{\pi_{N}-1}_{k=0}\sum^{k}_{r=0}\sum_{\bm{\sigma}\in\Omega_{S\setminus\{x,y\}}(r,k-r)}\frac{\bm{m}^{\bm{\sigma}}}{\bm{g!}(\bm{\sigma})}\sum^{N(1-2\epsilon)-2}_{p=2N\epsilon-k+1}\sum^{A-r-1}_{q=0}\frac{{p \choose q}{N-k-p-1 \choose A-r-q-1}}{\mathfrak{a}(p)\mathfrak{a}(N-k-p-1)}(q-p\rho_{N})^{2}.\label{eq:dom-3-2}
\end{equation}
Notice that $\frac{\binom{p}{q}\binom{N-k-p-1}{A-r-q-1}}{\binom{N-k-1}{A-r-1}}$
is the probability mass function $\mathbb{P}(H=q)$ of a hypergeometric
random variable $H$ with mean $\mathbb{E}[H]=\frac{p(A-r-1)}{N-k-1}$
and variance
\[
\mathbb{V}{\rm ar}[H]=\frac{p(A-r-1)}{N-k-1}\frac{B-(k-r)}{N-k-1}\frac{N-k-1-p}{N-k-2}\leq p.
\]
Thus, one can notice that for some constant $c>0$,
\begin{align*}
\sum^{A-r-1}_{q=0}\frac{{p \choose q}{N-k-p-1 \choose A-r-q-1}}{\binom{N-k-1}{A-r-1}}(q-p\rho_{N})^{2} & \le2\sum^{A-r-1}_{q=0}\frac{{p \choose q}{N-k-p-1 \choose A-r-q-1}}{\binom{N-k-1}{A-r-1}}\left(q-\frac{p(A-r-1)}{N-k-1}\right)^{2}+2p^{2}\left(\frac{A-r-1}{N-k-1}-\rho_{N}\right)^{2}\\
 & \le2\mathbb{V}{\rm ar}[H]+p^{2}\frac{c\pi^{2}_{N}}{N^{2}}\le2p+p^{2}\frac{c\pi^{2}_{N}}{N^{2}}.
\end{align*}
Thus, the term \eqref{eq:dom-3-2} is bounded above by 
\[
\frac{cN^{\alpha-4}}{\binom{N}{A}}\sum^{\pi_{N}-1}_{k=0}\sum^{k}_{r=0}\binom{N-k-1}{A-r-1}\sum_{\bm{\sigma}\in\Omega_{S\setminus\{x,y\}}(r,k-r)}\frac{\bm{m}^{\bm{\sigma}}}{\bm{g!}(\bm{\sigma})}\sum^{N(1-2\epsilon)}_{p=N\epsilon}\frac{2p+p^{2}\frac{c\pi^{2}_{N}}{N^{2}}}{p^{\alpha}(N(1-\epsilon)-p)^{\alpha}}.
\]
By Lemma \ref{lem:s1} (applied to $N\leftarrow\pi_{N}-1$) implies
that 
\[
\sum_{\bm{\sigma}\in\Omega_{S\setminus\{x,y\}}(r,k-r)}\frac{\bm{m}^{\bm{\sigma}}}{\bm{g!}(\bm{\sigma})}\le\frac{c{k \choose r}}{\mathfrak{a}(k)}.
\]
Thus, the penultimate display is bounded by
\begin{equation}
\frac{cN^{\alpha-4}}{\binom{N}{A}}\sum^{\pi_{N}-1}_{k=0}\sum^{k}_{r=0}\frac{\binom{N-k-1}{A-r-1}{k \choose r}}{\mathfrak{a}(k)}\sum^{N(1-2\epsilon)}_{p=N\epsilon}\frac{2p+p^{2}\frac{c\pi^{2}_{N}}{N^{2}}}{p^{\alpha}(N(1-\epsilon)-p)^{\alpha}}\le cN^{\alpha-4}\sum^{N(1-2\epsilon)}_{p=N\epsilon}\frac{2p+p^{2}\frac{c\pi^{2}_{N}}{N^{2}}}{p^{\alpha}(N(1-\epsilon)-p)^{\alpha}}.\label{eq:dom-3-3}
\end{equation}
It is straightforward that the last term is bounded by $O(N^{-\alpha-2})+O(N^{-\alpha-1}\pi^{2}_{N}N^{-2})\ll N^{-\alpha-1}$
since $\pi_{N}\ll N$.

Finally, consider the summation
\[
\sum_{\bm{\xi}\in\mathcal{B}^{1,y,{\rm bulk}}_{N-1}}\sum_{u,w\in S}\frac{N^{\alpha}}{Z_{N}\binom{N}{A}}\frac{\bm{m}^{\bm{\xi}}}{\bm{g!}(\bm{\xi})}m_{1}(u)r^{*}_{1}(u,w)\left(\frac{o_{\epsilon}(1)}{N^{2}}+\frac{\pi^{2}_{N}}{N^{4}}\right).
\]
According to the same characterization, this is bounded by
\begin{equation}
cN^{\alpha}\sum^{N(1-2\epsilon)}_{p=N\epsilon}\frac{1}{p^{\alpha}(N(1-\epsilon)-p)^{\alpha}}\left(\frac{o_{\epsilon}(1)}{N^{2}}+\frac{\pi^{2}_{N}}{N^{4}}\right)\le cN^{-\alpha-1}\left(o_{\epsilon}(1)+\frac{\pi^{2}_{N}}{N^{2}}\right).\label{eq:dom-3-4}
\end{equation}
The last term inside the parenthesis is clearly $o(1)+o_{\epsilon}(1)$.
Combining \eqref{eq:dom-3-3} and \eqref{eq:dom-3-4} proves \eqref{eq:dom-3-1}
as desired.
\end{proof}

\begin{lem}
\label{lem:dom-4}We have
\begin{align*}
 & -\sum_{i\in\{1,2\}}\sum_{\bm{\xi}\in\mathcal{B}^{i,y,{\rm bulk}}_{N-1}}\sum_{u,w\in S}\frac{N^{\alpha}}{Z_{N}\binom{N}{A}}\frac{\bm{m}^{\bm{\xi}}}{\bm{g!}(\bm{\xi})}m_{i}(u)r^{*}_{i}(u,w)\left[F_{N}(\bm{\xi}+\delta^{i}_{u})-\overline{F}^{i}_{N}(\bm{\xi})\right]\mathcal{W}^{i}_{x,y,1}(\bm{\xi};u,w)\\
 & =\frac{N^{-1-\alpha}}{\kappa_{\star}\Gamma(\alpha)\mathcal{I}_{\alpha}}\left(\frac{\rho}{c_{1}(x,y)}+\frac{1-\rho}{c_{2}(x,y)}\right)^{-1}\left[f_{N}(x)-f_{N}(y)\right]+o(N^{-1-\alpha})+N^{-1-\alpha}o_{\epsilon}(1).
\end{align*}
\end{lem}

\begin{proof}
By \eqref{eq:Wx-diff-dec-2}, the left-hand side can be expressed
as
\begin{equation}
\begin{aligned}-\sum_{i\in\{1,2\}} & \sum_{\bm{\xi}\in\mathcal{B}^{i,y,{\rm bulk}}_{N-1}}\frac{\theta^{-1}_{N}}{\mathcal{I}_{\alpha}\tau_{N}(x,y)c_{i}(x,y)Z_{N}\binom{N}{A}}\frac{\bm{m}^{\bm{\xi}}}{\bm{g!}(\bm{\xi})}|\bm{\xi}(x)|^{\alpha}|\bm{\xi}(y)|^{\alpha}\\
 & \times\sum_{u,w\in S}m_{i}(u)r^{*}_{i}(u,w)\left[F_{N}(\bm{\xi}+\delta^{i}_{u})-\overline{F}^{i}_{N}(\bm{\xi})\right]\left[\mathfrak{h}^{*}_{i,x,y}(w)-\mathfrak{h}^{*}_{i,x,y}(u)\right].
\end{aligned}
\label{eq:dom-4-1}
\end{equation}
We have
\[
\sum_{w\in S}r^{*}_{i}(u,w)\left[\mathfrak{h}^{*}_{i,x,y}(w)-\mathfrak{h}^{*}_{i,x,y}(u)\right]=L^{*}_{i}\mathfrak{h}^{*}_{i,x,y}(u)\qquad\text{for each}\quad u\in S,
\]
where $L^{*}_{i}$ is the infinitesimal generator of the adjoint random
walk generated by $r^{*}_{i}(\cdot,\cdot)$. Moreover, by the harmonic
property of $\mathfrak{h}^{*}_{i,x,y}$, the term $L^{*}_{i}\mathfrak{h}^{*}_{i,x,y}(u)$
is zero on $S\setminus\{x,y\}$. Thus, using the property $-L^{*}_{i}\mathfrak{h}^{*}_{i,x,y}(x)=L^{*}_{i}\mathfrak{h}^{*}_{i,x,y}(y)=c_{i}(x,y)$
for each $i\in\{1,2\}$, the summation in $u,w$ at \eqref{eq:dom-4-1}
is equal to
\[
-c_{i}(x,y)\left[F_{N}(\bm{\xi}+\delta^{i}_{x})-F_{N}(\bm{\xi}+\delta^{i}_{y})\right].
\]
Inserting this into \eqref{eq:dom-4-1} rewrites the term as
\[
\sum_{i\in\{1,2\}}\sum_{\bm{\xi}\in\mathcal{B}^{i,y,{\rm bulk}}_{N-1}}\frac{\theta^{-1}_{N}}{\mathcal{I}_{\alpha}\tau_{N}(x,y)Z_{N}\binom{N}{A}}\frac{\bm{m}^{\bm{\xi}}}{\bm{g!}(\bm{\xi})}|\bm{\xi}(x)|^{\alpha}|\bm{\xi}(y)|^{\alpha}\left[F_{N}(\bm{\xi}+\delta^{i}_{x})-F_{N}(\bm{\xi}+\delta^{i}_{y})\right].
\]
Let us define
\[
\Delta_{p}(n_{1},n_{2}):=\left\{ (q_{1},q_{2})\in\mathbb{N}^{2}_{0}:q_{1}+q_{2}=p,\enspace q_{1}\in\llbracket0,n_{1}\rrbracket,\enspace q_{2}\in\llbracket0,n_{2}\rrbracket\right\} .
\]
Then, the penultimate display becomes
\begin{equation}
\begin{aligned}\frac{\theta^{-1}_{N}}{\mathcal{I}_{\alpha}\tau_{N}(x,y)Z_{N}\binom{N}{A}} & \sum^{\pi_{N}-1}_{k=0}\sum^{k}_{r=0}\sum_{\bm{\sigma}\in\Omega_{S\setminus\{x,y\}}(r,k-r)}\frac{\bm{m}^{\bm{\sigma}}}{\bm{g!}(\bm{\sigma})}\sum^{N(1-2\epsilon)-2}_{p=2N\epsilon-k+1}\binom{N-k}{A-r}\\
 & \,\Bigg[\sum_{(q_{1},q_{2})\in\Delta_{p}(A-r-1,B-(k-r))}\frac{\binom{p}{q_{1}}\binom{N-k-p-1}{A-r-q_{1}-1}}{\binom{N-k}{A-r}}\left[F_{N}(\bm{\sigma};q_{1}+1,q_{2})-F_{N}(\bm{\sigma};q_{1},q_{2})\right]\\
 & \qquad+\sum_{(q_{1},q_{2})\in\Delta_{p}(A-r,B-(k-r)-1)}\frac{\binom{p}{q_{1}}\binom{N-k-p-1}{A-r-q_{1}}}{\binom{N-k}{A-r}}\left[F_{N}(\bm{\sigma};q_{1},q_{2}+1)-F_{N}(\bm{\sigma};q_{1},q_{2})\right]\Bigg]\,,
\end{aligned}
\label{eq:dom-4-2}
\end{equation}
where $F_{N}(\bm{\sigma};i,j):=F_{N}(\bm{\eta})$ for the configuration
$\bm{\eta}\in\Omega_{N}$ such that 
\begin{equation}
\bm{\eta}(z)=\begin{cases}
\bm{\sigma}(z) & \text{if}\quad z\in S\setminus\{x,y\},\\
(i,j) & \text{if}\quad z=x,\\
(A-|\sigma_{1}|-i,B-|\sigma_{2}|-j) & \text{if}\quad z=y.
\end{cases}\label{eq:dom-4-3}
\end{equation}
Denote by $\overline{F}^{\bm{\sigma},p}_{N}$ the average of $F_{N}$
on the submanifold $\Delta_{p}(A-r,B-(k-r))$ given $\bm{\sigma}$:
\[
\overline{F}^{\bm{\sigma},p}_{N}:=\sum_{(q_{1},q_{2})\in\Delta_{p}(A-r,B-(k-r))}\frac{\binom{p}{q_{1}}\binom{N-k-p}{A-r-q_{1}}}{\binom{N-k}{A-r}}F_{N}(\bm{\sigma};q_{1},q_{2}).
\]
Then,
\begin{equation}
\begin{aligned}\overline{F}^{\bm{\sigma},p+1}_{N}-\overline{F}^{\bm{\sigma},p}_{N}= & \sum_{(q_{1},q_{2})\in\Delta_{p}(A-r,B-(k-r))}\frac{\binom{p}{q_{1}}\binom{N-k-p-1}{A-r-q_{1}-1}}{\binom{N-k}{A-r}}\left[F_{N}(\bm{\sigma};q_{1}+1,q_{2})-F_{N}(\bm{\sigma};q_{1},q_{2})\right]\\
 & +\sum_{(q_{1},q_{2})\in\Delta_{p}(A-r,B-(k-r))}\frac{\binom{p}{q_{1}}\binom{N-k-p-1}{A-r-q_{1}}}{\binom{N-k}{A-r}}\left[F_{N}(\bm{\sigma};q_{1},q_{2}+1)-F_{N}(\bm{\sigma};q_{1},q_{2})\right],
\end{aligned}
\label{eq:dom-4-4}
\end{equation}
thus we may rewrite \eqref{eq:dom-4-2} as
\begin{equation}
\begin{aligned} & \frac{\theta^{-1}_{N}}{\tau_{N}(x,y)Z_{N}\mathcal{I}_{\alpha}}\sum^{\pi_{N}-1}_{k=0}\sum^{k}_{r=0}\frac{\binom{N-k}{A-r}}{\binom{N}{A}}\sum_{\bm{\sigma}\in\Omega_{S\setminus\{x,y\}}(r,k-r)}\frac{\bm{m}^{\bm{\sigma}}}{\bm{g!}(\bm{\sigma})}\sum^{N(1-2\epsilon)-2}_{p=2N\epsilon-k+1}\left[\overline{F}^{\bm{\sigma},p+1}_{N}-\overline{F}^{\bm{\sigma},p}_{N}\right]\\
 & =\frac{\theta^{-1}_{N}}{\tau_{N}(x,y)Z_{N}\mathcal{I}_{\alpha}}\sum^{\pi_{N}-1}_{k=0}\sum^{k}_{r=0}\frac{\binom{N-k}{A-r}}{\binom{N}{A}}\sum_{\bm{\sigma}\in\Omega_{S\setminus\{x,y\}}(r,k-r)}\frac{\bm{m}^{\bm{\sigma}}}{\bm{g!}(\bm{\sigma})}\left[\overline{F}^{\bm{\sigma},N(1-2\epsilon)-1}_{N}-\overline{F}^{\bm{\sigma},2N\epsilon-k+1}_{N}\right].
\end{aligned}
\label{eq:dom-4-5}
\end{equation}
Divide the summation in $k\in\llbracket0,\pi_{N}-1\rrbracket$ at
the right-hand side of \eqref{eq:dom-4-5} into $\llbracket0,\ell_{N}\rrbracket$
and $\llbracket\ell_{N}+1,\pi_{N}-1\rrbracket$. First, consider the
summation $\sum^{\pi_{N}-1}_{k=\ell_{N}+1}$. By Lemma \ref{lem:s1},
\[
\frac{1}{\binom{k}{r}}\sum_{\bm{\sigma}\in\Omega_{S\setminus\{x,y\}}(r,k-r)}\frac{\bm{m}^{\bm{\sigma}}}{\bm{g!}(\bm{\sigma})}\leq\frac{c}{k^{\alpha}},
\]
which implies 
\[
\sum^{\pi_{N}-1}_{k=\ell_{N}+1}\sum^{k}_{r=0}\frac{\binom{N-k}{A-r}}{\binom{N}{A}}\sum_{\bm{\sigma}\in\Omega_{S\setminus\{x,y\}}(r,k-r)}\frac{\bm{m}^{\bm{\sigma}}}{\bm{g!}(\bm{\sigma})}\leq\frac{c}{\ell^{\alpha-1}_{N}}\ll1.
\]
Thus, we have
\[
\left|\frac{\theta^{-1}_{N}}{\tau_{N}(x,y)Z_{N}\mathcal{I}_{\alpha}}\sum^{\pi_{N}-1}_{k=\ell_{N}+1}\sum^{k}_{r=0}\frac{\binom{N-k}{A-r}}{\binom{N}{A}}\sum_{\bm{\sigma}\in\Omega_{S\setminus\{x,y\}}(r,k-r)}\frac{\bm{m}^{\bm{\sigma}}}{\bm{g!}(\bm{\sigma})}\left[\overline{F}^{\bm{\sigma},N(1-2\epsilon)-1}_{N}-\overline{F}^{\bm{\sigma},2N\epsilon-k+1}_{N}\right]\right|\ll\theta^{-1}_{N},
\]
by the uniform boundedness of $F_{N}$. To estimate the remaining
summation, which is
\begin{equation}
\frac{\theta^{-1}_{N}}{\tau_{N}(x,y)Z_{N}\mathcal{I}_{\alpha}}\sum^{\ell_{N}}_{k=0}\sum^{k}_{r=0}\frac{\binom{N-k}{A-r}}{\binom{N}{A}}\sum_{\bm{\sigma}\in\Omega_{S\setminus\{x,y\}}(r,k-r)}\frac{\bm{m}^{\bm{\sigma}}}{\bm{g!}(\bm{\sigma})}\left[\overline{F}^{\bm{\sigma},N(1-2\epsilon)-1}_{N}-\overline{F}^{\bm{\sigma},2N\epsilon-k+1}_{N}\right],\label{eq:dom-4-6}
\end{equation}
we first prove the following lemma:
\begin{lem}
\label{lem:dom-5}We have
\[
\sum^{\ell_{N}}_{k=0}\sum^{k}_{r=0}\frac{\binom{N-k}{A-r}}{\binom{N}{A}}\sum_{\bm{\sigma}\in\Omega_{S\setminus\{x,y\}}(r,k-r)}\frac{\bm{m}^{\bm{\sigma}}}{\bm{g!}(\bm{\sigma})}\left[\overline{F}^{\bm{\sigma},N(1-2\epsilon)-1}_{N}-\overline{F}^{\bm{\sigma},N-k}_{N}\right]^{2}=o_{\epsilon}(1),
\]
and similarly,
\[
\sum^{\ell_{N}}_{k=0}\sum^{k}_{r=0}\frac{\binom{N-k}{A-r}}{\binom{N}{A}}\sum_{\bm{\sigma}\in\Omega_{S\setminus\{x,y\}}(r,k-r)}\frac{\bm{m}^{\bm{\sigma}}}{\bm{g!}(\bm{\sigma})}\left[\overline{F}^{\bm{\sigma},2N\epsilon-k+1}_{N}-\overline{F}^{\bm{\sigma},0}_{N}\right]^{2}=o_{\epsilon}(1).
\]
\end{lem}

\begin{proof}
By symmetry, we only prove the first estimate. By the Cauchy--Schwarz
inequality,
\begin{equation}
\begin{aligned} & \left[\overline{F}^{\bm{\sigma},N(1-2\epsilon)-1}_{N}-\overline{F}^{\bm{\sigma},N-k}_{N}\right]^{2}\\
 & \le\left(\sum^{N-k-1}_{p=N(1-2\epsilon)-1}\frac{\left[\overline{F}^{\bm{\sigma},p+1}_{N}-\overline{F}^{\bm{\sigma},p}_{N}\right]^{2}}{\mathfrak{a}(p)\mathfrak{a}(N-p-k-1)}\right)\left(\sum^{N-k-1}_{p=N(1-2\epsilon)-1}\mathfrak{a}(p)\mathfrak{a}(N-p-k-1)\right).
\end{aligned}
\label{eq:dom-5-1}
\end{equation}
 Then by Jensen's inequality, via \eqref{eq:dom-4-4}, the first summation
at the second line is bounded above by 
\begin{align*}
\sum^{N-k-1}_{p=N(1-2\epsilon)-1} & \sum_{(q_{1},q_{2})\in\Delta_{p}(A-r,B-(k-r))}\frac{\frac{\binom{p}{q_{1}}\binom{N-k-p-1}{A-r-q_{1}-1}}{\binom{N-k}{A-r}}\left[F_{N}(\bm{\sigma};q_{1}+1,q_{2})-F_{N}(\bm{\sigma};q_{1},q_{2})\right]^{2}}{\mathfrak{a}(p)\mathfrak{a}(N-p-k-1)}\\
+ & \sum^{N-k-1}_{p=N(1-2\epsilon)-1}\sum_{(q_{1},q_{2})\in\Delta_{p}(A-r,B-(k-r))}\frac{\frac{\binom{p}{q_{1}}\binom{N-k-p-1}{A-r-q_{1}}}{\binom{N-k}{A-r}}\left[F_{N}(\bm{\sigma};q_{1},q_{2}+1)-F_{N}(\bm{\sigma};q_{1},q_{2})\right]^{2}}{\mathfrak{a}(p)\mathfrak{a}(N-p-k-1)}.
\end{align*}
Note the following relations: 
\begin{align*}
\frac{\binom{p}{q_{1}}\binom{N-k-p-1}{A-r-q_{1}-1}}{\mathfrak{a}(p)\mathfrak{a}(N-p-k-1)} & =\frac{1}{\bm{g!}(q_{1},q_{2})}\frac{1}{\bm{g!}(A-r-q_{1}-1,B-q_{2}-(k-r))},\\
\frac{\binom{p}{q_{1}}\binom{N-k-p-1}{A-r-q_{1}}}{\mathfrak{a}(p)\mathfrak{a}(N-p-k-1)} & =\frac{1}{\bm{g!}(q_{1},q_{2})}\frac{1}{\bm{g!}(A-r-q_{1},B-q_{2}-(k-r)-1)},
\end{align*}
Thus, 
\[
\sum^{\ell_{N}}_{k=0}\sum^{k}_{r=0}\frac{\binom{N-k}{A-r}}{\binom{N}{A}}\sum_{\bm{\sigma}\in\Omega_{S\setminus\{x,y\}}(r,k-r)}\frac{\bm{m}^{\bm{\sigma}}}{\bm{g!}(\bm{\sigma})}\left(\sum^{N-k-1}_{p=N(1-2\epsilon)-1}\frac{\left[\bar{F}^{\bm{\sigma},p+1}_{N}-\bar{F}^{\bm{\sigma},p}_{N}\right]^{2}}{\mathfrak{a}(p)\mathfrak{a}(N-p-k-1)}\right)
\]
is bounded above by
\begin{align*}
\frac{1}{\binom{N}{A}}\sum^{\ell_{N}}_{k=0}\sum^{k}_{r=0}\sum_{\bm{\sigma}\in\Omega_{S\setminus\{x,y\}}(r,k-r)}\frac{\bm{m}^{\bm{\sigma}}}{\bm{g!}(\bm{\sigma})}\sum^{N-k-1}_{p=N(1-2\epsilon)-1} & \sum_{(q_{1},q_{2})\in\Delta_{p}(A-r,B-(k-r))}\\
 & \frac{\left[F_{N}(\bm{\sigma};q_{1}+1,q_{2})-F_{N}(\bm{\sigma};q_{1},q_{2})\right]^{2}}{\bm{g!}(q_{1},q_{2})\bm{g!}(A-r-q_{1}-1,B-q_{2}-(k-r))}\\
+\frac{1}{\binom{N}{A}}\sum^{\ell_{N}}_{k=0}\sum^{k}_{r=0}\sum_{\bm{\sigma}\in\Omega_{S\setminus\{x,y\}}(r,k-r)}\frac{\bm{m}^{\bm{\sigma}}}{\bm{g!}(\bm{\sigma})} & \sum^{N-k-1}_{p=N(1-2\epsilon)-1}\sum_{(q_{1},q_{2})\in\Delta_{p}(A-r,B-(k-r))}\\
 & \frac{\left[F_{N}(\bm{\sigma};q_{1},q_{2}+1)-F_{N}(\bm{\sigma};q_{1},q_{2})\right]^{2}}{\bm{g!}(q_{1},q_{2})\bm{g!}(A-r-q_{1},B-q_{2}-(k-r)-1)}.
\end{align*}
From the definition \eqref{eq:dom-4-3}, by enlarging the summation
region to $\Omega^{i,-}_{N-1}$, the above term is bounded above by
\begin{equation}
\frac{1}{\binom{N}{A}}\sum_{i\in\{1,2\}}\sum_{\bm{\xi}\in\Omega^{i,-}_{N-1}}\frac{\bm{m}^{\bm{\xi}}}{\bm{g!}(\bm{\xi})}\left[F_{N}(\bm{\xi}+\delta^{i}_{x})-F_{N}(\bm{\xi}+\delta^{i}_{y})\right]^{2}.\label{eq:dom-5-2}
\end{equation}
As argued in the proof of \eqref{eq:sector-5},
\[
\left[F_{N}(\bm{\xi}+\delta^{i}_{x})-F_{N}(\bm{\xi}+\delta^{i}_{y})\right]^{2}\leq c\sum_{u,w\in S}m_{i}(u)r^{*}_{i}(u,w)\left[F_{N}(\bm{\xi}+\delta^{i}_{w})-F_{N}(\bm{\xi}+\delta^{i}_{u})\right]^{2}.
\]
Inserting this into \eqref{eq:dom-5-2} implies that \eqref{eq:dom-5-2}
is bounded above by $cN^{-\alpha}\mathcal{D}_{N}(F_{N})$. Since 
\[
\sum^{N-k-1}_{p=N(1-2\epsilon)-1}\mathfrak{a}(p)\mathfrak{a}(N-p-k-1)=N^{2\alpha+1}o_{\epsilon}(1),
\]
uniformly over all $k\in\llbracket0,\ell_{N}\rrbracket$, thus we
conclude that 
\[
\sum^{\ell_{N}}_{k=0}\sum^{k}_{r=0}\frac{\binom{N-k}{A-r}}{\binom{N}{A}}\sum_{\bm{\sigma}\in\Omega_{S\setminus\{x,y\}}(r,k-r)}\frac{\bm{m}^{\bm{\sigma}}}{\bm{g!}(\bm{\sigma})}\left[\overline{F}^{\bm{\sigma},N(1-2\epsilon)-1}_{N}-\overline{F}^{\bm{\sigma},N-k}_{N}\right]^{2}\leq cN^{\alpha+1}\mathcal{D}_{N}(F_{N})\times o_{\epsilon}(1)=o_{\epsilon}(1),
\]
where the equality follows from Lemma \ref{lem:PE1}. 
\end{proof}

Now, we return to estimate \eqref{eq:dom-4-6}. By the Cauchy--Schwarz
inequality,
\[
\left|\frac{\theta^{-1}_{N}}{\tau_{N}(x,y)Z_{N}\mathcal{I}_{\alpha}}\sum^{\ell_{N}}_{k=0}\sum^{k}_{r=0}\frac{\binom{N-k}{A-r}}{\binom{N}{A}}\sum_{\bm{\sigma}\in\Omega_{S\setminus\{x,y\}}(r,k-r)}\frac{\bm{m}^{\bm{\sigma}}}{\bm{g!}(\bm{\sigma})}\left[\overline{F}^{\bm{\sigma},N-k}_{N}-\overline{F}^{\bm{\sigma},N(1-2\epsilon)-1}_{N}\right]\right|
\]
is bounded above by
\begin{align*}
 & \left(\frac{\theta^{-1}_{N}}{\tau_{N}(x,y)Z_{N}\mathcal{I}_{\alpha}}\sum^{\ell_{N}}_{k=0}\sum^{k}_{r=0}\frac{\binom{N-k}{A-r}}{\binom{N}{A}}\sum_{\bm{\sigma}\in\Omega_{S\setminus\{x,y\}}(r,k-r)}\frac{\bm{m}^{\bm{\sigma}}}{\bm{g!}(\bm{\sigma})}\right)^{\frac{1}{2}}\\
 & \times\left(\frac{\theta^{-1}_{N}}{\tau_{N}(x,y)Z_{N}\mathcal{I}_{\alpha}}\sum^{\ell_{N}}_{k=0}\sum^{k}_{r=0}\frac{\binom{N-k}{A-r}}{\binom{N}{A}}\sum_{\bm{\sigma}\in\Omega_{S\setminus\{x,y\}}(r,k-r)}\frac{\bm{m}^{\bm{\sigma}}}{\bm{g!}(\bm{\sigma})}\left[\overline{F}^{\bm{\sigma},N-k}_{N}-\overline{F}^{\bm{\sigma},N(1-2\epsilon)-1}_{N}\right]^{2}\right)^{\frac{1}{2}}.
\end{align*}
The first term is $O(\theta^{-1/2}_{N})$ by Lemma \ref{lem:s1},
and the second term is $o_{\epsilon}(1)\theta^{-1/2}_{N}$ by Lemma
\ref{lem:dom-5}. Therefore, we have
\begin{equation}
\begin{aligned}\frac{\theta^{-1}_{N}}{\tau_{N}(x,y)Z_{N}\mathcal{I}_{\alpha}} & \sum^{\ell_{N}}_{k=0}\sum^{k}_{r=0}\frac{\binom{N-k}{A-r}}{\binom{N}{A}}\sum_{\bm{\sigma}\in\Omega_{S\setminus\{x,y\}}(r,k-r)}\frac{\bm{m}^{\bm{\sigma}}}{\bm{g!}(\bm{\sigma})}\overline{F}^{\bm{\sigma},N(1-2\epsilon)-1}_{N}\\
= & \frac{\theta^{-1}_{N}}{\tau_{N}(x,y)Z_{N}\mathcal{I}_{\alpha}}\sum^{\ell_{N}}_{k=0}\sum^{k}_{r=0}\frac{\binom{N-k}{A-r}}{\binom{N}{A}}\sum_{\bm{\sigma}\in\Omega_{S\setminus\{x,y\}}(r,k-r)}\frac{\bm{m}^{\bm{\sigma}}}{\bm{g!}(\bm{\sigma})}\overline{F}^{\bm{\sigma},N-k}_{N}+o_{\epsilon}(1)\theta^{-1}_{N}.
\end{aligned}
\label{eq:dom-4-7}
\end{equation}
Similarly, we have
\begin{equation}
\begin{aligned}\frac{\theta^{-1}_{N}}{\tau_{N}(x,y)Z_{N}\mathcal{I}_{\alpha}} & \sum^{\ell_{N}}_{k=0}\sum^{k}_{r=0}\frac{\binom{N-k}{A-r}}{\binom{N}{A}}\sum_{\bm{\sigma}\in\Omega_{S\setminus\{x,y\}}(r,k-r)}\frac{\bm{m}^{\bm{\sigma}}}{\bm{g!}(\bm{\sigma})}\overline{F}^{\bm{\sigma},2N\epsilon-k+1}_{N}\\
= & \frac{\theta^{-1}_{N}}{\tau_{N}(x,y)Z_{N}\mathcal{I}_{\alpha}}\sum^{\ell_{N}}_{k=0}\sum^{k}_{r=0}\frac{\binom{N-k}{A-r}}{\binom{N}{A}}\sum_{\bm{\sigma}\in\Omega_{S\setminus\{x,y\}}(r,k-r)}\frac{\bm{m}^{\bm{\sigma}}}{\bm{g!}(\bm{\sigma})}\overline{F}^{\bm{\sigma},0}_{N}+o_{\epsilon}(1)\theta^{-1}_{N}.
\end{aligned}
\label{eq:dom-4-8}
\end{equation}
By \eqref{eq:dom-4-7} and \eqref{eq:dom-4-8}, the term at \eqref{eq:dom-4-6}
equals
\begin{equation}
\frac{\theta^{-1}_{N}}{\tau_{N}(x,y)Z_{N}\mathcal{I}_{\alpha}}\sum^{\ell_{N}}_{k=0}\sum^{k}_{r=0}\frac{\binom{N-k}{A-r}}{\binom{N}{A}}\sum_{\bm{\sigma}\in\Omega_{S\setminus\{x,y\}}(r,k-r)}\frac{\bm{m}^{\bm{\sigma}}}{\bm{g!}(\bm{\sigma})}\left[\overline{F}^{\bm{\sigma},N-k}_{N}-\overline{F}^{\bm{\sigma},0}_{N}\right]+o_{\epsilon}(1)\theta^{-1}_{N}.\label{eq:dom-4-9}
\end{equation}

Finally, notice that all the configurations that appear in the definition
of $\overline{F}^{\bm{\sigma},N-k}_{N}$ (resp. $\overline{F}^{\bm{\sigma},0}_{N}$)
belong to $\mathcal{E}^{x}_{N}$ (resp. $\mathcal{E}^{y}_{N}$) by
definition. Thus by Proposition \ref{prop1},
\[
\lim_{N\to\infty}\left|\overline{F}^{\bm{\sigma},N-k}_{N}-f_{N}(x)\right|=0,\qquad\lim_{N\to\infty}\left|\overline{F}^{\bm{\sigma},0}_{N}-f_{N}(y)\right|=0.
\]
From these and collecting \eqref{eq:dom-4-1}, \eqref{eq:dom-4-2},
\eqref{eq:dom-4-5}, \eqref{eq:dom-4-6}, and \eqref{eq:dom-4-9},
we conclude that the left-hand side of Lemma \ref{lem:dom-4} is equal
to
\begin{equation}
\frac{\theta^{-1}_{N}}{\tau_{N}(x,y)Z_{N}\mathcal{I}_{\alpha}}\sum^{\ell_{N}}_{k=0}\sum^{k}_{r=0}\frac{\binom{N-k}{A-r}}{\binom{N}{A}}\sum_{\bm{\sigma}\in\Omega_{S\setminus\{x,y\}}(r,k-r)}\frac{\bm{m}^{\bm{\sigma}}}{\bm{g!}(\bm{\sigma})}\left[f_{N}(x)-f_{N}(y)\right]+(o(1)+o_{\epsilon}(1))\theta^{-1}_{N}.\label{eq:dom-4-10}
\end{equation}
Note that as in the proof of Proposition \ref{prop:ZN-limit},
\[
\frac{1}{\tau_{N}(x,y)Z_{N}\mathcal{I}_{\alpha}}\sum^{\ell_{N}}_{k=0}\sum^{k}_{r=0}\frac{\binom{N-k}{A-r}}{\binom{N}{A}}\sum_{\bm{\sigma}\in\Omega_{S\setminus\{x,y\}}(r,k-r)}\frac{\bm{m}^{\bm{\sigma}}}{\bm{g!}(\bm{\sigma})}
\]
converges, as $N\to\infty$, to
\[
\frac{\Gamma(\alpha)^{\kappa_{\star}-2}\left(\prod_{z\in S\setminus S_{\star}}\Gamma_{z}\right)}{\left(\frac{\rho}{c_{1}(x,y)}+\frac{1-\rho}{c_{2}(x,y)}\right)\kappa_{\star}\Gamma(\alpha)^{\kappa_{\star}-1}\left(\prod_{z\in S\setminus S_{\star}}\Gamma_{z}\right)\mathcal{I}_{\alpha}}=\frac{1}{\left(\frac{\rho}{c_{1}(x,y)}+\frac{1-\rho}{c_{2}(x,y)}\right)\kappa_{\star}\Gamma(\alpha)\mathcal{I}_{\alpha}}.
\]
Substituting this to \eqref{eq:dom-4-10}, it equals
\[
\frac{\theta^{-1}_{N}\left[f_{N}(x)-f_{N}(y)\right]}{\left(\frac{\rho}{c_{1}(x,y)}+\frac{1-\rho}{c_{2}(x,y)}\right)\kappa_{\star}\Gamma(\alpha)\mathcal{I}_{\alpha}}+(o(1)+o_{\epsilon}(1))\theta^{-1}_{N},
\]
as claimed in Lemma \ref{lem:dom-4}. This finishes the proof.
\end{proof}

\begin{proof}[Proof of Lemma \ref{lem:res-3}]
 By Lemma \ref{lem:dom-1} and the discussion after,
\[
\langle F_{N},-\mathcal{L}^{*}_{N}\mathbb{W}_{x}\rangle_{\nu_{N},\mathcal{G}^{x}_{N}}=\sum_{y\in S_{\star}\setminus\{x\}}\langle F_{N},-\mathcal{L}^{*}_{N}\mathbb{W}_{x}\rangle_{\nu_{N},\mathcal{J}^{x,y}_{N}}.
\]
By \eqref{eq:inn-prod-xy-dec} and Lemma \ref{lem:dom-2},
\[
\langle F_{N},-\mathcal{L}^{*}_{N}\mathbb{W}_{x}\rangle_{\nu_{N},\mathcal{J}^{x,y}_{N}}=\langle\!\langle F_{N},-\mathcal{L}^{*}_{N}\mathbb{W}_{x}\rangle\!\rangle_{\mathcal{B}^{y,{\rm bulk}}_{N-1}}+o(N^{-1-\alpha}).
\]
By \eqref{eq:Wx-diff-dec-1}, Lemmas \ref{lem:dom-3}, and \ref{lem:dom-4},
\[
\langle\!\langle F_{N},-\mathcal{L}^{*}_{N}\mathbb{W}_{x}\rangle\!\rangle_{\mathcal{B}^{y,{\rm bulk}}_{N-1}}=\frac{N^{-1-\alpha}}{\kappa_{\star}\Gamma(\alpha)\mathcal{I}_{\alpha}}\left(\frac{\rho}{c_{1}(x,y)}+\frac{1-\rho}{c_{2}(x,y)}\right)^{-1}\left[f_{N}(x)-f_{N}(y)\right]+o(N^{-1-\alpha})+N^{-1-\alpha}o_{\epsilon}(1).
\]
Gathering the three displays completes the proof of Lemma \ref{lem:res-3}.
\end{proof}

Finally, we are ready to prove Proposition \ref{prop2}.
\begin{proof}[Proof of Proposition \ref{prop2}]
 Substituting Lemmas \ref{lem:res-1}, \ref{lem:res-2}, and \ref{lem:res-3}
to \eqref{eq:inner-product}, we obtain that
\[
\lambda\nu_{N}(\mathcal{E}^{x}_{N})f_{N}(x)+\sum_{y\in S_{\star}\setminus\{x\}}\frac{f_{N}(x)-f_{N}(y)}{\left(\frac{\rho}{c_{1}(x,y)}+\frac{1-\rho}{c_{2}(x,y)}\right)\kappa_{\star}\mathcal{I}_{\alpha}\Gamma(\alpha)}={\bf g}(x)\nu_{N}(\mathcal{E}^{x}_{N})+o(1)+o_{\epsilon}(1).
\]
Since the left-hand side is independent of $\epsilon$, by first taking
the limit superior as $N\to\infty$ and then sending $\epsilon\to0$,
and recalling \eqref{eq:lMC-gen} and Theorem \ref{thm:cond}, we
obtain
\[
(\lambda-\mathfrak{L}_{\mathbb{X}})f_{N}(x)={\bf g}(x)+o(1)\qquad\text{for all}\quad x\in S_{\star}.
\]
This implies that
\[
f_{N}(x)=(\lambda-\mathfrak{L}_{\mathbb{X}})^{-1}{\bf g}(x)+o(1)={\bf f}(x)+o(1),
\]
which concludes the proof of Proposition \ref{prop2}.
\end{proof}

\begin{proof}[Proof of Theorem \ref{thm:main}-(2)/(3)]
 Clearly, Propositions \ref{prop1} and \ref{prop2} imply Theorem
\ref{thm:res}. Then, the equivalence statement in \cite[Theorem 2.3]{LMS25}
verifies that Theorem \ref{thm:res} implies parts (2) and (3) of
Theorem \ref{thm:main}. This concludes the proof.
\end{proof}

\begin{acknowledgement*}
SK would like to thank Claudio Landim for the suggestion of the model
and helpful discussions. SK and SL have been supported by the Basic
Science Research Program through the National Research Foundation
of Korea funded by the Ministry of Science and ICT (RS-2025-00518980,
RS-2026-25518141), the Yonsei University Research Fund of 2026 (2026-22-0181),
and the POSCO Science Fellowship of POSCO TJ Park Foundation.
\end{acknowledgement*}

\appendix

\section{\label{secA}Asymptotic Summation of Reciprocals}

Here, we record miscellaneous summation asymptotics, which will be
useful in Section \ref{sec2} when we conduct the stationary analysis.
We start with an inequality.
\begin{lem}
\label{lem:comb-bound}Fix positive real numbers $x_{1},y_{1},\dots,x_{k},y_{k}>0$.
Then for any integers $N,n,p_{1},\dots,p_{k}\ge0$ such that $p_{1}+\cdots+p_{k}=N$
and $0\le n\le N$,
\[
\frac{1}{\binom{N}{n}}\sum_{\substack{q_{1},\dots,q_{k}\ge0:\\
q_{1}+\cdots+q_{k}=n
}
}\prod^{k}_{j=1}\left[{p_{j} \choose q_{j}}x^{q_{j}}_{j}y^{p_{j}-q_{j}}_{j}\right]\le c\prod^{k}_{j=1}\left(\frac{nx_{j}+(N-n)y_{j}}{N}\right)^{p_{j}},
\]
where $c>0$ depends only on the fixed constants $x_{1},y_{1},\dots,x_{k},y_{k}$.
\end{lem}

\begin{proof}
If $n=0$, then both the left-hand side and the product in the right-hand
side equal $\prod^{k}_{j=1}y^{p_{j}}_{j}$. If $n=N$, then both values
equal $\prod^{k}_{j=1}x^{p_{j}}_{j}$.

Henceforth, assume that $0<n<N$. Let us denote by $\Sigma$ the summation
in the left-hand side. By an elementary combinatorial argument, $\Sigma$
equals the coefficient of $z^{n}$ in $\prod^{\kappa}_{j=1}(x_{j}z+y_{j})^{p_{j}}$.
Thus by Cauchy's integral formula, we obtain that
\[
\Sigma=\frac{1}{2\pi i}\oint\frac{\prod^{\kappa}_{j=1}(x_{j}z+y_{j})^{p_{j}}}{z^{n+1}}\,{\rm d}z=\frac{1}{2\pi}\int^{\pi}_{-\pi}\frac{\prod^{\kappa}_{j=1}(x_{j}re^{i\theta}+y_{j})^{p_{j}}}{r^{n}e^{in\theta}}\,{\rm d}\theta,
\]
where the contour integral is taken over the circle with radius $r>0$
around its center $0$. Then, notice that
\[
|x_{j}re^{i\theta}+y_{j}|^{2}=(x_{j}r+y_{j})^{2}-2rx_{j}y_{j}(1-\cos\theta)\le(x_{j}r+y_{j})^{2}e^{-2\beta_{j}\theta^{2}},
\]
where $\beta_{j}:=2rx_{j}y_{j}/(\pi^{2}(x_{j}r+y_{j})^{2})>0$. At
the inequality, we used $1-\cos\theta\ge2\theta^{2}/\pi^{2}$ for
$\theta\in[-\pi,\pi]$. Substituting this to the penultimate equation,
we obtain that
\begin{equation}
\Sigma\le cr^{-n}\prod^{\kappa}_{j=1}(x_{j}r+y_{j})^{p_{j}}\int^{\pi}_{-\pi}e^{-\beta_{\min}N\theta^{2}}\,{\rm d}\theta\le\frac{c'r^{-n}}{\sqrt{\beta_{\min}N}}\prod^{\kappa}_{j=1}(x_{j}r+y_{j})^{p_{j}},\label{eq:comb-1}
\end{equation}
where we used that $p_{1}+\cdots+p_{k}=N$. Above, $\beta_{\min}:=\min_{j\in\llbracket1,k\rrbracket}\beta_{j}>0$.
Note that
\[
\beta_{\min}=\min_{j\in\llbracket1,k\rrbracket}\frac{2rx_{j}y_{j}}{\pi^{2}(x_{j}r+y_{j})^{2}}\ge\frac{cr}{(r+1)^{2}},
\]
where $c>0$ depends on the values of $x_{1},y_{1},\dots,x_{k},y_{k}>0$.

Now, let us take $r:=\frac{n}{N-n}>0$. Then, Stirling's formula and
the last display imply that
\begin{align*}
\frac{1}{{N \choose n}}\frac{r^{-n}}{\sqrt{\beta_{\min}N}} & \le c\frac{n^{n}(N-n)^{N-n}\sqrt{n(N-n)}}{N^{N}\sqrt{N}}\frac{N}{\sqrt{n}\sqrt{N-n}}\frac{n^{-n}(N-n)^{n}}{\sqrt{N}}=c\left(\frac{N-n}{N}\right)^{N}.
\end{align*}
Substituting this to \eqref{eq:comb-1}, we conclude that
\[
\frac{\Sigma}{\binom{N}{n}}\le c\left(\frac{N-n}{N}\right)^{N}\prod^{k}_{j=1}\left(\frac{nx_{j}}{N-n}+y_{j}\right)^{p_{j}}=c\prod^{k}_{j=1}\left(\frac{nx_{j}}{N}+\frac{(N-n)y_{j}}{N}\right)^{p_{j}},
\]
as desired, where at the equality we used $N=p_{1}+\cdots+p_{k}$.
\end{proof}

Next, we record an elementary asymptotic summation lemma of reciprocals.
We could not find a direct reference for the results, so we included
a short proof for completeness.
\begin{lem}
\label{lem:asymp}For fixed $k\ge1$ and $\alpha\in\mathbb{R}$, as
$N\to\infty$,
\[
\sum_{\substack{p_{1},\dots,p_{k}\ge1:\\
p_{1}+\cdots+p_{k}=N
}
}\frac{1}{p^{\alpha}_{1}p^{\alpha}_{2}\cdots p^{\alpha}_{k}}\simeq\begin{cases}
k\bm{\zeta}(\alpha)^{k-1}N^{-\alpha} & \text{if}\quad\alpha>1,\\
\frac{k(\log N)^{k-1}}{N} & \text{if}\quad\alpha=1,\\
\frac{\bm{\Gamma}(1-\alpha)^{k}}{\bm{\Gamma}(k(1-\alpha))}N^{k(1-\alpha)-1} & \text{if}\quad\alpha<1.
\end{cases}
\]
Above, $\bm{\Gamma}(\cdot)$ and $\bm{\zeta}(\cdot)$ denote the usual
gamma and zeta functions, respectively.
\end{lem}

\begin{proof}
Denote by $S_{N}(k)$ the left-hand side of the lemma. First, consider
the case of $\alpha<1$. By the Riemann integration theory,
\begin{align*}
 & S_{N}(k)=N^{k(1-\alpha)-1}\frac{1}{N^{k-1}}\sum_{\substack{p_{1},\dots,p_{k}\ge1:\\
p_{1}+\cdots+p_{k}=N
}
}\frac{1}{\left(\frac{p_{1}}{N}\right)^{\alpha}\left(\frac{p_{2}}{N}\right)^{\alpha}\cdots\left(\frac{p_{k}}{N}\right)^{\alpha}}\\
 & \simeq N^{k(1-\alpha)-1}\int_{\substack{\{x_{1},\dots,x_{k-1}\ge0:\\
x_{1}+\cdots+x_{k-1}\le1\}
}
}\frac{1}{x^{\alpha}_{1}\cdots x^{\alpha}_{k-1}(1-x_{1}-\cdots-x_{k-1})^{\alpha}}\,{\rm d}x_{k-1}\,\cdots\,{\rm d}x_{2}\,{\rm d}x_{1}.
\end{align*}
It is easy to check that the last integral value equals $\bm{\Gamma}(1-\alpha)^{k}/\bm{\Gamma}(k(1-\alpha))$
via induction and change of variables; we omit the detail.

Next, let $\alpha=1$. It is clear that $S_{N}(k)$ is the $x^{N}$-coefficient
of the function $(-\log(1-x))^{k}$. It is well known (e.g., \cite[Sec. 26.8.8]{OLBC10})
that
\[
(-\log(1-x))^{k}=k!\sum^{\infty}_{n=k}{n \brack k}\frac{x^{n}}{n!},
\]
where ${n \brack k}$ is the unsigned Stirling number of the first
kind, so that
\[
S_{N}(k)=\frac{k!}{N!}{N \brack k}\simeq\frac{k!}{N!}\frac{(N-1)!}{(k-1)!}(\log N)^{k-1}=\frac{k}{N}(\log N)^{k-1}\qquad\text{as}\quad N\to\infty.
\]
The asymptotic above is elementary; see for instance \cite[Sec. 26.8.40]{OLBC10}.
This concludes the $\alpha=1$ case.

Finally, suppose that $\alpha>1$. It is clear that $S_{N}(1)=N^{-\alpha}$.
Let us assume the result for $k-1$ and let $k\ge2$. Substituting
$n=N-p_{k}$, we rewrite as
\[
N^{\alpha}S_{N}(k)=\sum^{N-1}_{n=1}\frac{N^{\alpha}}{(N-n)^{\alpha}}\sum_{\substack{p_{1},\dots,p_{k-1}\ge1:\\
p_{1}+\cdots+p_{k-1}=n
}
}\frac{1}{p^{\alpha}_{1}\cdots p^{\alpha}_{k-1}}=\sum^{N-1}_{n=1}\left(1-\frac{n}{N}\right)^{-\alpha}S_{n}(k-1).
\]
We divide the summation into two parts:
\[
N^{\alpha}S_{N}(k)=\sum_{1\le n\le\frac{N}{2}}\left(1-\frac{n}{N}\right)^{-\alpha}S_{n}(k-1)+\sum_{\frac{N}{2}<n\le N-1}\left(1-\frac{n}{N}\right)^{-\alpha}S_{n}(k-1)=:X_{N}+Y_{N}.
\]
First, we may write
\[
X_{N}=\sum^{\infty}_{n=1}\left(1-\frac{n}{N}\right)^{-\alpha}S_{n}(k-1){\bf 1}\left\{ n\le\frac{N}{2}\right\} .
\]
Considering this as an integration of $x_{N}(n):=(1-\frac{n}{N})^{-\alpha}S_{n}(k-1){\bf 1}\{n\le\frac{N}{2}\}$,
we have $\lim_{N\to\infty}x_{N}(n)=S_{n}(k-1)$ and, by the induction
hypothesis,
\[
|x_{N}(n)|\le2^{\alpha}S_{n}(k-1)\le\frac{c2^{\alpha}}{n^{\alpha}},\qquad\text{where}\quad\sum^{\infty}_{n=1}\frac{c2^{\alpha}}{n^{\alpha}}<\infty.
\]
Thus, the dominated convergence theorem implies that
\begin{equation}
\lim_{N\to\infty}X_{N}=\sum^{\infty}_{n=1}S_{n}(k-1)=\sum_{p_{1},\dots,p_{k-1}\ge1}\frac{1}{p^{\alpha}_{1}\cdots p^{\alpha}_{k-1}}=\bm{\zeta}(\alpha)^{k-1}.\label{eq:AN-lim}
\end{equation}
Next, substituting $n\leftarrow N-n$, we write
\[
Y_{N}=\sum_{1\le n<\frac{N}{2}}\left(\frac{n}{N}\right)^{-\alpha}S_{N-n}(k-1)=\sum^{\infty}_{n=1}n^{-\alpha}N^{\alpha}S_{N-n}(k-1){\bf 1}\left\{ n<\frac{N}{2}\right\} .
\]
Let $y_{N}(n):=n^{-\alpha}N^{\alpha}S_{N-n}(k-1){\bf 1}\{n<\frac{N}{2}\}$.
Then, by the induction hypothesis, $\lim_{N\to\infty}y_{N}(n)=n^{-\alpha}(k-1)\zeta(\alpha)^{k-2}$.
Moreover, again by the hypothesis,
\[
|y_{N}(n)|\le n^{-\alpha}\frac{N^{\alpha}c}{(N-n)^{\alpha}}{\bf 1}\left\{ n<\frac{N}{2}\right\} \le\frac{2^{\alpha}c}{n^{\alpha}},
\]
which is summable in $n\ge1$. Therefore, the dominated convergence
theorem again gives
\begin{equation}
\lim_{N\to\infty}Y_{N}=\sum^{\infty}_{n=1}n^{-\alpha}(k-1)\bm{\zeta}(\alpha)^{k-2}=(k-1)\bm{\zeta}(\alpha)^{k-1}.\label{eq:BN-lim}
\end{equation}
Combining \eqref{eq:AN-lim} and \eqref{eq:BN-lim} yields the desired
result for $k$, concluding the proof of the case of $\alpha>1$.
\end{proof}

\begin{lem}
\label{lem:asymp2}Recall \eqref{eq:a-def}. For fixed $k\ge1$ and
$\alpha\in\mathbb{R}$, as $N\to\infty$,
\[
\sum_{\substack{p_{1},\dots,p_{k}\ge0:\\
p_{1}+\cdots+p_{k}=N
}
}\frac{1}{\mathfrak{a}(p_{1})\mathfrak{a}(p_{2})\cdots\mathfrak{a}(p_{k})}\simeq\begin{cases}
k(1+\bm{\zeta}(\alpha))^{k-1}N^{-\alpha} & \text{if}\quad\alpha>1,\\
\frac{k(\log N)^{k-1}}{N} & \text{if}\quad\alpha=1,\\
\frac{\bm{\Gamma}(1-\alpha)^{k}}{\bm{\Gamma}(k(1-\alpha))}N^{k(1-\alpha)-1} & \text{if}\quad\alpha<1.
\end{cases}
\]
\end{lem}

\begin{proof}
For $\alpha\le1$, if at least one $p_{j}$ is zero in the summation,
then Lemma \ref{lem:asymp} implies that the result gives a scale
strictly smaller than the ones in the right-hand side. Thus, again,
Lemma \ref{lem:asymp} implies that
\[
\sum_{\substack{p_{1},\dots,p_{k}\ge0:\\
p_{1}+\cdots+p_{k}=N
}
}\frac{1}{\mathfrak{a}(p_{1})\mathfrak{a}(p_{2})\cdots\mathfrak{a}(p_{k})}\simeq\sum_{\substack{p_{1},\dots,p_{k}\ge1:\\
p_{1}+\cdots+p_{k}=N
}
}\frac{1}{p^{\alpha}_{1}p^{\alpha}_{2}\cdots p^{\alpha}_{k}}\simeq\begin{cases}
\frac{k(\log N)^{k-1}}{N} & \text{if}\quad\alpha=1,\\
\frac{\bm{\Gamma}(1-\alpha)^{k}}{\bm{\Gamma}(k(1-\alpha))}N^{k(1-\alpha)-1} & \text{if}\quad\alpha<1.
\end{cases}
\]
On the other hand, suppose that $\alpha>1$. Then, the zero cases
contribute equally at the same scale. Thus, denoting by $m$ the number
of zeros, we obtain via Lemma \ref{lem:asymp} that
\begin{align*}
\sum_{\substack{p_{1},\dots,p_{k}\ge0:\\
p_{1}+\cdots+p_{k}=N
}
}\frac{1}{\mathfrak{a}(p_{1})\mathfrak{a}(p_{2})\cdots\mathfrak{a}(p_{k})} & =\sum^{k-1}_{m=0}{k \choose m}\sum_{\substack{p_{1},\dots,p_{k-m}\ge1:\\
p_{1}+\cdots+p_{k-m}=N
}
}\frac{1}{p^{\alpha}_{1}p^{\alpha}_{2}\cdots p^{\alpha}_{k-m}}\\
 & \simeq\sum^{k-1}_{m=0}{k \choose m}(k-m)\bm{\zeta}(\alpha)^{k-m-1}N^{-\alpha}=k(1+\bm{\zeta}(\alpha))^{k-1}N^{-\alpha},
\end{align*}
concluding the proof.
\end{proof}

The next two lemmas will be useful in Section \ref{sec2.1}.
\begin{lem}
\label{lem:asymp3}Let $\alpha>1$. For any $\ell<N/2$,
\[
\sum_{\substack{0\le p_{1},\dots,p_{k}<N-\ell:\\
p_{1}+\cdots+p_{k}=N
}
}\frac{1}{\mathfrak{a}(p_{1})\mathfrak{a}(p_{2})\cdots\mathfrak{a}(p_{k})}\le\frac{c}{(\ell+1)^{\alpha-1}N^{\alpha}},
\]
where $c>0$ does not depend on $\ell$.
\end{lem}

\begin{proof}
The inequality is clear for $k=1$ since the left-hand side is zero.
Assume the induction hypothesis for $k-1$ and fix $k\ge2$. As $p_{1}$
varies from $0$ to $N-\ell-1$, the situation changes at the value
$\ell$. Namely, if $\ell<p_{1}\le N-\ell-1$ then the other variables
$p_{2},\dots,p_{k}$ are automatically smaller than $N-\ell$, thus
the corresponding part becomes
\[
\sum^{N-\ell-1}_{p_{1}=\ell+1}\frac{1}{\mathfrak{a}(p_{1})}\sum_{\substack{p_{2},\dots,p_{k}\ge0:\\
p_{2}+\cdots+p_{k}=N-p_{1}
}
}\frac{1}{\mathfrak{a}(p_{2})\cdots\mathfrak{a}(p_{k})}\le\sum^{N-\ell-1}_{p_{1}=\ell+1}\frac{c}{p^{\alpha}_{1}(N-p_{1})^{\alpha}},
\]
where the inequality follows from Lemma \ref{lem:asymp2}. By symmetry,
the right hand side is bounded by
\begin{equation}
\sum^{N/2}_{p_{1}=\ell+1}\frac{2c}{p^{\alpha}_{1}(N-p_{1})^{\alpha}}\le\sum^{N/2}_{p_{1}=\ell+1}\frac{c'N^{-\alpha}}{p^{\alpha}_{1}}\le c''N^{-\alpha}(\ell+1)^{-\alpha+1},\label{eq:as2-1}
\end{equation}
as desired. On the other hand, if $0\le p_{1}\le\ell$ then the induction
hypothesis comes into play:
\begin{equation}
\sum^{\ell}_{p_{1}=0}\frac{1}{\mathfrak{a}(p_{1})}\sum_{\substack{0\le p_{2},\dots,p_{k}<(N-p_{1})-(\ell-p_{1}):\\
p_{2}+\cdots+p_{k}=N-p_{1},
}
}\frac{1}{\mathfrak{a}(p_{2})\cdots\mathfrak{a}(p_{k})}\le\sum^{\ell}_{p_{1}=0}\frac{1}{\mathfrak{a}(p_{1})}\frac{c}{(\ell-p_{1}+1)^{\alpha-1}N^{\alpha}}.\label{eq:as2-2}
\end{equation}
At the inequality, we used $2(N-p_{1})>N$. Since 
\[
\sum^{\ell}_{p_{1}=1}p^{-\alpha}_{1}(\ell-p_{1}+1)^{-\alpha+1}=\frac{\ell+1}{2}\sum^{\ell}_{p_{1}=1}p^{-\alpha}_{1}(\ell-p_{1}+1)^{-\alpha},
\]
the right-hand side of \eqref{eq:as2-2} is bounded by $cN^{-\alpha}(\ell+1)^{-\alpha+1}$
using the same logic as in \eqref{eq:as2-1}.
\end{proof}

\begin{lem}
\label{lem:asymp4}For $\alpha>1$ and $\ell'<\ell/2<N/4$,
\[
\sum_{\substack{0\le p_{1},\dots,p_{k}<N-\ell:\\
p_{1}+\cdots+p_{k}=N,\\
p_{j}+p_{j'}<N-\ell',\ \forall j\ne j'
}
}\frac{1}{\mathfrak{a}(p_{1})\mathfrak{a}(p_{2})\cdots\mathfrak{a}(p_{k})}\leq\frac{c}{(\ell+1)^{\alpha-1}(\ell'+1)^{\alpha-1}N^{\alpha}}.
\]
\end{lem}

\begin{proof}
It clear for $k=1$. Assume for $k-1$ and suppose that $k\ge2$.
With a similar decomposition according to the range of $p_{1}$, the
left-hand side becomes
\[
\begin{aligned}\sum^{N-\ell-1}_{p_{1}=\ell-\ell'+1}\frac{1}{\mathfrak{a}(p_{1})}\sum_{\substack{0\le p_{2},\dots,p_{k}<(N-p_{1})-\ell':\\
p_{2}+\cdots+p_{k}=N-p_{1}
}
} & \frac{1}{\mathfrak{a}(p_{2})\cdots\mathfrak{a}(p_{k})}+\sum^{\ell-\ell'}_{p_{1}=\ell'+1}\frac{1}{\mathfrak{a}(p_{1})}\sum_{\substack{0\le p_{2},\dots,p_{k}<(N-p_{1})-(\ell-p_{1}):\\
p_{2}+\cdots+p_{k}=N-p_{1}
}
}\frac{1}{\mathfrak{a}(p_{2})\cdots\mathfrak{a}(p_{k})}\\
 & +\sum^{\ell'}_{p_{1}=0}\frac{1}{\mathfrak{a}(p_{1})}\sum_{\substack{0\le p_{2},\dots,p_{k}<(N-p_{1})-(\ell-p_{1}):\\
p_{2}+\cdots+p_{k}=N-p_{1},\\
p_{j}+p_{j'}<(N-p_{1})-(\ell'-p_{1}),\ \forall j\ne j'\ge2
}
}\frac{1}{\mathfrak{a}(p_{2})\cdots\mathfrak{a}(p_{k})}.
\end{aligned}
\]
We handle the three terms separately. The first term is bounded as
desired via Lemma \ref{lem:asymp3} since
\[
\sum^{N-\ell-1}_{p_{1}=\ell-\ell'+1}\frac{1}{\mathfrak{a}(p_{1})}\frac{c}{(\ell'+1)^{\alpha-1}(N-p_{1})^{\alpha}}\le\frac{c'}{(\ell-\ell'+1)^{\alpha-1}N^{\alpha}(\ell'+1)^{\alpha-1}}\le\frac{c''}{(\ell+1)^{\alpha-1}N^{\alpha}(\ell'+1)^{\alpha-1}},
\]
where the first inequality follows similarly as in \eqref{eq:as2-1},
and the second inequality holds since $\ell>2\ell'$. The second term
is bounded as desired via Lemma \ref{lem:asymp3} since
\[
\sum^{\ell-\ell'}_{p_{1}=\ell'+1}\frac{1}{\mathfrak{a}(p_{1})}\frac{c}{(\ell-p_{1}+1)^{\alpha-1}(N-p_{1})^{\alpha}}\le\sum^{\ell-\ell'}_{p_{1}=\ell'+1}\frac{c'}{p^{\alpha}_{1}(\ell-p_{1}+1)^{\alpha-1}N^{\alpha}}\le\frac{c''}{(\ell+1)^{\alpha-1}(\ell'+1)^{\alpha-1}N^{\alpha}},
\]
where the second inequality follows similarly as in \eqref{eq:as2-2}.
Finally, the third term is bounded as desired via the induction hypothesis
since
\begin{align*}
\sum^{\ell'}_{p_{1}=0}\frac{1}{\mathfrak{a}(p_{1})}\frac{c}{(\ell-p_{1}+1)^{\alpha-1}(\ell'-p_{1}+1)^{\alpha-1}(N-p_{1})^{\alpha}} & \le\frac{c'}{(\ell+1)^{\alpha-1}N^{\alpha}}\sum^{\ell'}_{p_{1}=0}\frac{1}{\mathfrak{a}(p_{1})(\ell'-p_{1}+1)^{\alpha-1}}\\
 & \le\frac{c''}{(\ell+1)^{\alpha-1}N^{\alpha}(\ell'+1)^{\alpha-1}},
\end{align*}
where the first inequality follows from $\ell>2\ell'$, and the second
one follows similarly as in \eqref{eq:as2-2}. These three displays
conclude the proof via induction.
\end{proof}

On the other hand, the next lemma will be handy in Section \ref{sec2.2}.
\begin{lem}
\label{lem:asymp5}Assume that $\alpha=1$. Then for any $\ell<N/2$,
\[
\sum_{\substack{0\le p_{1},\dots,p_{k}<N-\ell:\\
p_{1}+\cdots+p_{k}=N
}
}\frac{1}{\mathfrak{a}(p_{1})\mathfrak{a}(p_{2})\cdots\mathfrak{a}(p_{k})}\le\frac{c(\log N)^{k-2}(\log N-\log(\ell+1))}{N},
\]
where $c>0$ does not depend on $\ell$.
\end{lem}

\begin{proof}
The initial step $k=1$ is clear since the left-hand side is zero.
Suppose that the lemma holds for $k-1$ and consider $k\ge2$. Summing
first for $p_{1}$, the left-hand side of the lemma equals
\[
\left[\sum^{N-\ell-1}_{p_{1}=\ell+1}\sum_{\substack{p_{2},\dots,p_{k}\ge0:\\
p_{2}+\cdots+p_{k}=N-p_{1}
}
}+\sum^{\ell}_{p_{1}=0}\sum_{\substack{0\le p_{2},\dots,p_{k}<(N-p_{1})-(\ell-p_{1}):\\
p_{2}+\cdots+p_{k}=N-p_{1}
}
}\right]\frac{1}{\mathfrak{a}(p_{1})\mathfrak{a}(p_{2})\cdots\mathfrak{a}(p_{k})}.
\]
By Lemma \ref{lem:asymp2}, the first part is easily bounded above
by
\[
\sum^{N-\ell-1}_{p_{1}=\ell+1}\frac{1}{p_{1}}\frac{c(\log N)^{k-2}}{N-p_{1}}\le\frac{c'}{N}(\log N)^{k-2}(\log N-\log(\ell+1)).
\]
For the second part, we apply the induction hypothesis to bound it
above by
\[
\sum^{\ell}_{p_{1}=0}\frac{1}{\mathfrak{a}(p_{1})}\times\frac{c(\log N)^{k-3}(\log(N-p_{1})-\log(\ell-p_{1}+1))}{N-p_{1}}.
\]
Thus, we are left to prove that
\[
\sum^{\ell}_{p_{1}=0}\frac{\log(N-p_{1})-\log(\ell-p_{1}+1)}{\mathfrak{a}(p_{1})\mathfrak{a}(N-p_{1})}\le c\frac{\log N(\log N-\log(\ell+1))}{N}.
\]
If $p_{1}\le\frac{\ell+1}{2}$, then $\frac{N-p_{1}}{\ell-p_{1}+1}\le\frac{2N}{\ell+1}$,
thus
\[
\sum^{\frac{\ell+1}{2}}_{p_{1}=0}\frac{\log(N-p_{1})-\log(\ell-p_{1}+1)}{\mathfrak{a}(p_{1})\mathfrak{a}(N-p_{1})}\le\sum^{\frac{\ell+1}{2}}_{p_{1}=0}\frac{\log2+\log N-\log(\ell+1)}{\mathfrak{a}(p_{1})\mathfrak{a}(N-p_{1})}\le\frac{c\log N(\log N-\log(\ell+1))}{N}.
\]
Finally, if $\frac{\ell+1}{2}<p_{1}\le\ell$, then $\log(N-p_{1})-\log(\ell-p_{1}+1)\le\log N$,
thus
\[
\sum_{\frac{\ell+1}{2}<p_{1}\le\ell}\frac{\log(N-p_{1})-\log(\ell-p_{1}+1)}{\mathfrak{a}(p_{1})\mathfrak{a}(N-p_{1})}\le\sum_{\frac{\ell+1}{2}<p_{1}\le\ell}\frac{\log N}{\mathfrak{a}(p_{1})\mathfrak{a}(N-p_{1})}\le c\log N\frac{\log N-\log(\ell+1)}{N}.
\]
The last two inequalities conclude the proof via induction.
\end{proof}

\end{document}